\documentclass[11pt,reqno]{amsart}
\usepackage{booktabs}
\usepackage[utf8]{inputenc}
\usepackage[T1]{fontenc}
\usepackage{amssymb,amsmath,amsthm,amsfonts,eufrak,enumerate}
\usepackage{mathabx}
\usepackage{ulem}
\usepackage[english]{babel}
\usepackage[dvipsnames]{xcolor}
\usepackage{enumitem}
\usepackage[pagebackref, colorlinks = true, linkcolor = teal, urlcolor  = teal, citecolor = violet]{hyperref}
\usepackage{tikz}
\newcommand*\circled[1]{\tikz[baseline=(char.base)]{
            \node[shape=circle,draw,inner sep=2pt] (char) {#1};}}

\usepackage{xr}

\newtheorem{thm}{Theorem}[section]
\newtheorem{ass}{Assumption}[section]
\newtheorem{prop}[thm]{Proposition}

\newtheorem{cor}[thm]{Corollary}
\theoremstyle{plain}
\newtheorem{lem}[thm]{Lemma}
\theoremstyle{plain}
\theoremstyle{definition}
\newtheorem{defi}[thm]{Definition}
\theoremstyle{remark}

\newtheorem{rmq}{Remark}

\newcommand{\luc}{\lesssim_{uc}}
\newcommand{\aaa}{\mathfrak{a}}
\newcommand{\deux}{\mathfrak{c}_1}
\newcommand{\cun}{\mathfrak{c}_3}

\newcommand{\ran}{r_{n}}

\newcommand{\E}{\mathbb{E}}

\newcommand{\R}{\mathbb{R}}
\newcommand{\ER}{\mathbb{R}}

\newcommand{\N}{\mathbb{N}}

\newcommand{\PE}{\mathbb{P}}
\newcommand{\lld}{\lambda_d}
\newcommand{\LL}{\mathcal{L}}

\newcommand{\cal}{\mathcal}

\newcommand{\Ur}{\zeta}

\newcommand{\Rd}{\mathbb{R}^d}
\newcommand{\Rq}{\mathbb{R}^q}
\newcommand{\Pth}{\mathbb{P}_{\theta}}
\newcommand{\Y}{\bar{X}}
\newcommand{\Yg}{\bar{X}^{\gamma}}

\newcommand{\Pthz}{\mathbb{P}_{\theta^{\star}}}
\renewcommand{\and}{\mbox{ and }}
\newcommand{\IWC}{(\mathbf{I_{W_1}(c)})}
\newcommand{\UPI}{(\mathbf{PI_U})}
\newcommand{\CPU}{C^U_P}
\newcommand{\ic}{\alpha_c}
\newcommand{\AL}{(\mathbf{A}_L)}
\newcommand{\AnL}{(\mathbf{A}_{n L})}
\newcommand{\SC}{(\mathbf{SC}_\rho)}

\newcommand{\tgno}{\widehat{\theta}^{\gamma}_{n,t_N}}

\newcommand{\nabdob}{\nabla W}
\newcommand{\dob}{W}
\newcommand{\hesdob}{\nabla^2 W}

\newcommand{\Xn}{X^{(n)}}

\newcommand{\bX}{\bar{X}}
\newcommand{\tX}{\tilde{X}}
\newcommand{\epsi}{\varepsilon}
\newcommand{\un}{\underline}

\newcommand{\mte}{\mathfrak{e}}

\newcommand{\Hcu}{\mathbf{(H}^{q,r}_{\mathfrak{KL}})}

\newcommand{\id}{\mathbf{I}_d}

\newcommand{\troiss}{{\mathfrak{c}_2}}
\newcommand{\cpar}{C_{{\rm par}}}
\newcommand{\ra}[1]{\renewcommand{\arraystretch}{#1}}

\newcommand{\tcb}[1]{{#1}}

\newcommand{\nsp}[1]{\|{#1}\|_{\star}}
\newcommand{\flip}{[f]_1}
\newcommand{\dflip}{[Df]_{1,\star}}
\newcommand{\dw}{\Upsilon}

\usepackage{xargs}
\usepackage[prependcaption]{todonotes}
\newcommandx{\clem}[2][1=]{\todo[inline, author={Clem}, linecolor=blue,backgroundcolor=blue!25,bordercolor=blue,#1]{#2}}
\newcommandx{\clemnote}[2][1=]{\todo[author={Clem}, linecolor=blue,backgroundcolor=blue!25,bordercolor=blue,#1]{#2}}

\newcommandx{\seb}[2][1=]{\todo[inline, author={Seb}, linecolor=green,backgroundcolor=green!25,bordercolor=green,#1]{#2}}
\newcommandx{\sebnote}[2][1=]{\todo[author={Seb}, linecolor=green,backgroundcolor=green!25,bordercolor=green,#1]{#2}}

\newcommandx{\fab}[2][1=]{\todo[inline, author={Fab}, linecolor=green,backgroundcolor=red!25,bordercolor=red,#1]{#2}}
\newcommandx{\fabnote}[2][1=]{\todo[author={Fab}, linecolor=red,backgroundcolor=red!25,bordercolor=red,#1]{#2}}

\author{S\'ebastien Gadat,  Fabien Panloup and Cl\'ement Pellegrini }

\address{S. Gadat: Toulouse School of Economics, CNRS UMR 5314\\ Universit\'e Toulouse 1 Capitole\\
Esplanade de l'Universit\'e, Toulouse, France.\\
${}^{\ddagger}$ Institut Universitaire de France.}
\address{F. Panloup: Universit\'e d'Angers, CNRS, LAREMA, SFR MATHSTIC, F-49000 Angers, France.} 
\address{C. Pellegrini: Institut de Math\'ematiques de Toulouse; UMR5219, UPS IMT, F-31062 Toulouse Cedex 9, France.}%
\email{\url{sebastien.gadat(at)tse-fr.eu},  \url{fabien.panloup(at)univ-angers.fr}, 
\url{clement.pellegrini(at)math.univ-toulouse.fr}}%

\title{On the cost of Bayesian posterior mean strategy for log-concave models}

\date{\today}

\begin{document}
	
\thanks{S. Gadat acknowledges funding from the French National Research Agency (ANR) under the Investments for the Future program (Investissements d'Avenir, grant ANR-17-EURE-0010) and for the grant MaSDOL - 19-CE23-0017-01.}

\thanks{
The authors gratefully acknowledge Patrick Cattiaux, Max Fathi, Gersende Fort, Nathael Gozlan and Ald\'eric Joulin for stimulating discussions and valuable insights during the development of this work and the anonymous referees for comments that led to a significant improvement of the work.}

\begin{abstract} In this paper, we investigate the problem of computing Bayesian estimators using  Langevin Monte-Carlo type  approximation. The novelty of this paper is to consider together the statistical and numerical counterparts (in a  general log-concave setting). More precisely, we address the following question: given $n$ observations in $\ER^q$ distributed under an unknown probability $\PE_{\theta^\star}$, $\theta^\star\in\ER^d$, what is the optimal numerical strategy and its cost for the approximation of $\theta^\star$ with the Bayesian posterior mean?

\smallskip

\noindent To answer this question, we establish some quantitative statistical bounds related to the underlying  Poincar\'e constant of the model and  establish new results about the numerical approximation of Gibbs measures by Cesaro averages of Euler schemes of (over-damped) Langevin diffusions. These last results {are mainly based on} some quantitative controls on the solution of the related \textit{Poisson equation} of the (over-damped) Langevin diffusion {in strongly and weakly convex settings}.

\end{abstract}

\maketitle





\section{Introduction}

\subsection{Log-concave statistical models \label{sec:log-concave-model}}

In this paper, we consider a statistical model $(\Pth)_{\theta \in \mathbb{R}^d}$ parametrized by a parameter $\theta \in \mathbb{R}^d$. 
We assume that each distribution $\Pth$ defines a probability measure on   $(\mathbb{R}^q,\mathcal{B}(\mathbb{R}^q))$
and that all the distributions $\Pth$ are absolutely continuous with respect to the Lebesgue measure $\lambda_q$, we denote by $\pi_\theta$ the corresponding density:

$$
\forall \xi \in \Rq \qquad 
\pi_\theta(\xi):=\frac{\text{d}\Pth}{\text{d}\lambda_q}(\xi).
$$
We assume that we observe  $n$ i.i.d. realizations $(\xi_1,\ldots,\xi_n)$, sampled according to  $\Pthz$  where $\theta^{\star}$ is an unknown parameter. We are then interested in Bayesian statistical procedures designed to recover $\theta^\star$.
In all the paper, we restrict our study to the specific class of \textit{log-concave models} where the distributions are described by:
\begin{equation}\label{eq:ptheta}
\pi_{\theta}(\xi) :=  e^{-U(\xi,\theta)},
\end{equation} where  {$(\xi,\theta) \longmapsto U(\xi,\theta)=- \log(\pi_{\theta}(\xi))$} is assumed to be a convex function.
 Note that implicitly, the normalizing constant 
$
Z_{\theta} := \int_{\R^q} e^{-U(\xi,\theta)} \text{d}\xi
$
is assumed to be equal to $1$, which is not restrictive up to a modification of $U$.

 
{Besides the Gaussian toy model that trivially fall{s} into our framework, log-concave statistical models have a longstanding history in a wide range of applied mathematics and it seems almost  impossible to  enumerate exhaustively the range of possible applications.
 {For instance, the log-concave setting appears} with exponential families thanks to the
Pitman-Koopman-Darmois Theorem, in  extreme value theory, tests (chi-square distributions), Bayesian statistics among others.  Log-concave distributions also play a central role in probability and functional analysis (\cite{Poincare_log_concave,Bobkov-aop}), or geometry (see \textit{e.g.} \cite{KLS}). 
 The log-concave property is commonly used in economics (for example the density of customer's utility parameters is generally assumed to  satisfy this property \cite{Bagnoli_economics}), 
in game theory (see \textit{e.g.} \cite{Caplin_competition} and \cite{Tirole_contracts}), in political science and social choice (see \textit{e.g.}  \cite{Caplin_election}) or in econometrics (for example through the Roy model, see \textit{e.g.} \cite{Heckman_roy}){.}

 An important example comes from all distributions that are built with a multivariate convex function $U: \mathbb{R}^d \times \mathbb{R}^q \longrightarrow \mathbb{R}$ and where the first $d$ coordinates are considered as the hidden parameter $\theta$ while the $q$ other ones are the observations.
Starting from a convex function $U: \mathbb{R}^q \longrightarrow \mathbb{R}^q$,  translation models $(\xi,\theta)\longmapsto e^{-U(\xi-\theta)}$ also generate typical examples of log-concave models in $(\xi,\theta)$.
Many other distributions satisfy the log-concave property: Gumbel and Weibull distributions with a shape parameter larger than 1. 
 In particular, this includes a large class of parametric probability distributions such as Gaussian or Laplace models, logistic regression models, Subbotin distributions,  Gamma or Wishart distributions, Beta and uniform distributions on real intervals among others.   
We refer to \cite{SW} for a detailed  survey on properties of log-concave distributions, to \cite{Walther} for a list of modeling and application issues and to \cite{Samworth} for non-parametric density estimation procedures of log-concave distributions.
}

\subsection{Bayesian estimation of $\theta^{\star}$}
We briefly sketch Bayesian strategies for estimating $\theta^{\star}$.

\noindent
 Considering a prior distribution $\Pi_0$ on  $\theta$,  we assume that $\Pi_0$ is absolutely continuous with respect to the Lebesgue measure $\lld$ on $\mathbb{R}^d$. 
Without any {possible confusion} with the familiy of densities  $(\pi_\theta)_{\theta \in \R^d}$, the associated density of the prior is denoted by $\pi_0$
and we further assume that $\pi_0$ is also log-concave: $\pi_0(\theta):= e^{-V_0(\theta)}$
where $V_0$ encodes the prior knowledge on $\theta$. We emphasize that this last assumption is not restrictive since the prior distribution is chosen by the user.
We  denote by {$\pi_n$  the density of the posterior distribution} (that depends on the  observations $\mathbf{\xi^n}:= (\xi_1,\ldots,\xi_n)$) given by:
$\pi_n(\theta)\,\propto \,\pi_{0}(\theta)\prod_{k=1}^n \pi_{\theta}(\xi_k ).$
The posterior distribution is a data-driven probability distribution that  may be written as:
\begin{equation}\label{eq:def_wn}
\forall \theta\in \Rd \qquad \pi_n(\theta)=\frac{e^{-W_n(\mathbf{\xi^n},\theta)}}{Z_n(\mathbf{\xi^n})}\pi_0(\theta)
\quad \text{where}
\quad 
 W_n(\mathbf{\xi^n},\theta)=\sum_{i=1}^nU(\xi_i,\theta).
\end{equation}  The quantity $Z_n(\mathbf{\xi^n})$ corresponds to the normalizing constant and depends as well on $\mathbf{\xi^n}$:
$$
Z_n(\mathbf{\xi^n}) = \int_{\R^d} \pi_0(\theta) e^{-W_n(\mathbf{\xi^n},\theta)} \text{d} \theta.
$$
It is well known that  the posterior distribution enjoys consistency properties (see \textit{e.g.} \cite{Schwartz,IbraHas}):  under mild assumptions on the prior distribution and on the statistical model, the posterior distribution concentrates its mass around $\theta \in \R^d$ whose distribution is close to $\mathbb{P}_{\theta^{\star}}$.
 
With additional metric and identifiability assumptions, some stronger results may be obtained   in general parametric or non-parametric models. We refer to the seminal contribution of \cite{GhosalGhoshvdVaart} and the references therein,  to the work of \cite{CasvdV} for an extension to less standard situations of high-dimensional models and to \cite{vZvdV} for infinite dimensional models. 
The posterior distribution may be used to define 
Bayesian estimators, in particular, we shall introduce the popular posterior mean estimator of $\theta^{\star}$ defined by:
\begin{equation}
\label{eq:posterior_mean}
\widetilde{\theta}_n = \int_{\mathbb{R}^d} \theta \pi_n(\theta) \text{d}\theta.
\end{equation}
This estimator is usually consistent and a popular issue  is to establish \textit{rates of convergence} towards $\theta^{\star}$ through some $L^p$-criterion, \textit{i.e.} find a sequence $\epsilon_n \longrightarrow 0$ such that:
\begin{equation}\label{eq:statbound}
\mathbb{E}_{\theta^\star} \left(|\widetilde{\theta}_n-\theta^\star|^p\right) \leq \epsi_n^p,
\end{equation}
for a given $p>1$. Nevertheless, when a such bound is obtained,   the story is not over. Actually, the theoretical posterior mean given by Equation \eqref{eq:posterior_mean} being generally not explicit, the practical use of the above statistical bounds certainly requires to provide computable algorithms that may approximate $\widetilde{\theta}_n$.  In particular, it is legitimate to look for a tractable estimator 
$\widehat{\theta}_n$ that approaches an $\epsi_n$-neighborhood of $\widetilde{\theta}_n$ with as less operations as possible. 

\subsection{Langevin Monte Carlo discretization and practical estimator}\label{sec:intro_disc}

To approximate $\widetilde{\theta}_n$, it is commonly used  to write $\pi_n$  as a Gibbs field:
\begin{equation}\label{eq:tildeW}
\pi_n(\theta)\, \propto \,  \exp(-\widetilde W_n(\mathbf{\xi^n},\theta)) \quad \text{with} \quad 
\widetilde{W}_n(\theta)  = \sum_{i=1}^n U(\xi_i,\theta)+\log \left(\frac{1}{\pi_0(\theta)}\right).\end{equation}
%
Under some mild assumptions on $\widetilde W_n$, it is well known that a such probability measure is the unique invariant distribution   of the (over-damped) Langevin diffusion defined by:
\begin{eqnarray}\label{eq:diffusion}
d\Xn_t =- \nabla\widetilde W_n(\Xn_t)dt + \sqrt 2 dB_t,
\end{eqnarray} 
where $(B_t)$ is a $d$-dimensional standard Brownian motion.
Thus, the probability $\pi_n$ can be approximated using the long-time ergodic convergence of $(\Xn_t)_{t\ge0}$ towards $\pi_n$. One can  distinguish two types of convergences towards $\pi_n$: the convergence of the distribution of $\Xn_t$ as $t\rightarrow+\infty$ or the $a.s.$ convergence of the occupation measure  of $(\Xn_t)_{t\ge0}$. 
 
 Here, we build   our algorithm with the second type of convergence, which requires only one path of the diffusion.
 We are thus led to consider the occupation measure applied to the identity function denoted by $\id$. In this case, this is nothing but the Cesaro average:
\begin{equation}
\label{def:cesaro_shift}
\forall n \in \N^\star \quad \forall t >0 \qquad \widehat{\theta}_{n,t} := \frac{1}{t} \int_{0}^t\Xn_s \text{d}s.
\end{equation}

{In  \eqref{def:cesaro_shift}, Cesaro averages are based on the ``true'' diffusion but to obtain a tractable algorithm, we need to introduce Cesaro averages of some discretization schemes of  \eqref{def:cesaro_shift}.}
For this purpose,  we consider a  positive step size $\gamma$ and introduce a constant step-size explicit Euler-Maruyama scheme related to $\Xn$ {(omitting the index $n$ for simplicity)}:
\begin{equation}
\label{eq:Euler_def}
\forall k \ge 0 \quad \Y_{t_{k+1}}:=\Y_{t_k}-\gamma \nabla \widetilde{W}_n(\Y_{t_k})+\sqrt{2} \Ur_{k+1}  \quad  \text{with} \quad t_k=k \gamma,
\end{equation}
where   for all $k\ge1$, $\Ur_k={B}_{{t_{k}}}-{B}_{t_{k-1}}$ and $({B}_t)_{t\ge0}$ is a standard $d$-dimensional Brownian motion. 
%
It is possible to define a continuous affine interpolation of \eqref{eq:Euler_def} but from a practical point of view, it will be more comfortable to consider some initialization and ending times in the discrete grid $(t_k)_{k \ge 0}$. 

For any time horizon $t_N>0$, the approximation of $\tgno$ with a step-size $\gamma$ is given by:
\begin{equation}\label{def:tgn}
\tgno :=
\displaystyle\frac{1}{N}{\sum_{j=0}^{N-1}  \Y_{t_j}},
\end{equation}
which corresponds to the Cesaro average of the discretized trajectory \eqref{eq:Euler_def}. 


{This Cesaro construction first appeared in \cite{talay}  where some convergence properties of the empirical measure of the Euler scheme with constant step size were investigated. In a series of recent papers (among others, see \cite{lamberton_pages,LP03,pagespanloup2012} or \cite{pagespanloup2018} for a \textit{multilevel} extension), the decreasing-step setting has also been studied. Compared with these papers, the novelty of our work is that, we  propose some non-asymptotic quantitative bounds (see Section \ref{sec:contribplan} for details about the corresponding results). 
Note that for ease of presentation, we prefered to mainly consider the (less technical) constant step setting.}

\subsection{Contributions and plan of the paper.} \label{sec:contribplan} 
{Our main results are stated in Section \ref{sec:main}}. We address two main points:


\begin{itemize}
\item \textbf{The Bayesian consistency}: \textit{i.e.}  the  distance between the  posterior mean $\widetilde{\theta}_n$ and $\theta^\star$.
\end{itemize}

The Bayesian consistency is studied in  Section \ref{sec:bayes}. First, Theorem \ref{theo:consistency} derives an upper bound of the $L^p$ loss in terms of the \textit{Poincar\'e constant} of the model (which is assumed to satisfy a uniformity condition).  Compared with the literature on this problem (see $e.g.$ \cite{Jordan-Wainwright}), this result is  written under general assumptions on the family of log-concave models and does not require specific assumptions related to a dynamical system. Second, Theorem \ref{theo:lower_bound} proves that our upper bound is minimax optimal for a large class of models.

%

\begin{itemize}
\item \textbf{The  numerical question} related to the approximation of $\widetilde{\theta}_n$ by a computable algorithm. 
\end{itemize}

 Our main results about Cesaro-type LMC and optimal tuning of the parameters (in terms of $n$ and $d$) are Theorems \ref{theo:learning_weak_convex} and \ref{theo:learning_strong_convex} whose complete study is devoted to Section \ref{sec:disc_c} and Section \ref{sec:discretisation}. Our approach based on the \textit{Poisson equation}, $i.e.$ on the inversion of the infinitesimal generator of the diffusion with a quantitative point of view, will allow to obtain Theorem \ref{cor:thmPoisson} (weakly convex case) and Theorem \ref{cor:thmPoissonSC} (strongly convex case).
 On this numerical topic, our main  contribution is the (almost) quantitative study of the weakly convex case (with the help of the Kurdyka-\L ojasiewicz inequality introduced in $\Hcu$) based on quantitative controls on the solution of the \textit{Poisson equation} (and on its derivatives). Nevertheless, the Poisson approach seems also of interest in the strongly convex case, the results being very close to \cite{durmus} but slightly improve the dependence in the parameters of the model. This is probably due to the fact that Poisson approach does not require to couple the Euler scheme with the continuous process (see Remark \ref{rmqdurmusmus} for details). 
In the two cases, we obtain some  explicit bounds on the distance between the discretized Cesaro average and the posterior mean. These bounds lead to  an optimal tuning in terms of $n$ and $d$ in the weakly convex setting.


%
 In short, for both weakly and strongly convex situations, for some given $n$ and $d$, we first exhibit an  $\epsi_n$ that upper-bounds the ($L^p$-type) error between $\widetilde{\theta}_n$ and $\theta^\star$. Then, for this $\epsi_n$, we  aim at tuning  the procedure \eqref{def:tgn} in order to {obtain} an $\epsi_n$-approximation\footnote{By $\epsi$-approximation, we mean an approximation of the target with an $L^p$-error {of the order} ${\cal O}(\epsi)$.} of $\widetilde{\theta}_n$ with a minimal computational cost. 
  This cost $N_{\varepsilon_n}$ (number of iterations of the Euler scheme) will be explicited  as an amount of operations upper bounded by $n^a d^b$, where $a$ and $b$ are some positive numbers that will be explicited below.
Our main contribution states that Bayesian learning can be optimally performed in polynomial time less than $n^a d^b$ operations for typically weakly convex situations with explicit $a$ and $b$, and in even
$ n d$ operations in strongly convex cases.\\
 {For an improved readability of the work, numerous proofs and technical results are deferred to the Appendix.

{\subsection{Notations }

\begin{itemize}
\item{}  The notation $a\lesssim_{uc}b$ means that $a\le c b$ where $c$ is a \textit{universal constant}, $i.e.$ independent of any of the relevant parameters of the problems. 
\item{} {The usual scalar product on $\ER^d$ is denote by $\langle\,,\,\rangle$ and the induced Euclidean norm by $|\,.\,|$. The set $\mathbb{M}_{d,d}$ refers to the set of real $d \times d$ matrices. The Frobenius norm on $\mathbb{M}_{d,d}$ is denoted by $\|\,.\,\|_F$: for any $A\in\mathbb{M}_{d,d}$, $\|A\|_F^2=\sum_{1\le i,j\le d} A_{i,j}^2.$ We also denote by $\nsp{.}$ the spectral norm, defined for any matrix $A$ by $\nsp{A}=\sup_{|x|=1} |Ax| = \sqrt{\bar{\lambda}_{A^t A}} $, where $\bar{\lambda}_{A^t A}$ refers to the maximal eigenvalue of the symmetric matrix ${A^t A}$.} 
\item{} {For a Lipschitz continuous function $f:\ER^d\rightarrow\ER^{d'}$, we denote by $[f]_1$ its related Lipschitz constant. When $f$ is ${\cal C}^1$, its Jacobian matrix denoted by $Df$ maps from $\ER^d$ to $\mathbb{M}_{d,d'}$. When $Df$ is Lipschitz {for} the spectral norm $\nsp{.}$, we denote by $\dflip$ the related Lipschitz constant. As usual the notations $\nabla$ and $\Delta$ will hold for the gradient and the Laplacian.}  
\item{}  In the sequel,  we call upon to different distances on the space of probability measures: the $p$-Wasserstein distance is denoted by $\mathcal W_p$, $p\in[1,\infty[$ (mainly used with $p=1$ or $p=2$ in this paper). We recall that
$$\mathcal{W}_p(\mathbb{P}_{\theta_1},\mathbb{P}_{\theta_2})=\left(\inf_{\pi\in\Pi(\mathbb{P}_{\theta_1},\mathbb{P}_{\theta_1})}\int_{\mathbb R^q}d(x,y)^p\pi(x,y)\right)^{1/p},$$
where $\Pi(\mathbb{P}_{\theta_1},\mathbb{P}_{\theta_2})$ stands for the set of all coupling measures of $\mathbb{P}_{\theta_1}$ and $\mathbb{P}_{\theta_2}$. The Kullback-Leibler divergence is denoted by $KL(\mathbb{P}_{\theta_1},\mathbb{P}_{\theta_2})$ and is classically defined by
$$KL(\mathbb{P}_{\theta_1},\mathbb{P}_{\theta_2})=\int_{\mathbb R^q}\log\left(\frac{\pi_{\theta_1}(\xi)}{\pi_{\theta_2}(\xi)}\right)\pi(\theta_1)(\xi)d\xi.$$
\end{itemize}
}

\section{Main results and discussion \label{sec:main}}
%
%
%
In our work, two sources of randomness are considered. The first one is derived from the observations {$\mathbf{\xi^n}= (\xi_1,\ldots,\xi_n)$}. The notations $\mathbb{P}_{\theta}$ and $\mathbb{E}_{\theta}$ refer to the probability and expectation on {the unknown distribution of} the sampling process. The second source of randomness is related to the posterior distribution $\pi_n$ over $\Rd$: for any Borelian  $\mathcal{B}$ of $\Rd$, $\pi_n(\mathcal{B})$ is the probability of $\mathcal{B}$ when $\theta$ is sampled according to $\pi_n$, conditionally to the observations. Hence, $\mathbb{E}_{\pi_n}$ is the expectation when $\theta \sim \pi_n$, conditionally to {$\mathbf{\xi^n}$}.\\
\subsection{Functional inequality and Assumption $\UPI$}\label{sec:functional}

For any measure $\mu$  and $f\in L^1(\mu)$,   $\mu(f)$ refers to the mean value of $f$, and when  $f\in L^2(\mu)$, $Var_{\mu}(f)$ is the variance of $f$:
$$
\mu(f) := \int_{\mathbb{R}^q} f(\xi) \text{d}\mu(\xi) \qquad \text{and} \qquad 
Var_{\mu}(f) := \int_{\mathbb{R}^q} [f(\xi)-\mu(f)]^2 \text{d}\mu(\xi).
$$
A crucial property of log-concave measures is that they satisfy a Poincar\'e inequality. This will be used extensively in the rest of the paper. We refer to \cite{Ledoux,BGL} for a complete presentation and some applications on concentration inequalities and Markov processes.

\begin{defi}[Poincar\'e inequality]
A measure $\mu$ satisfies a Poincar\'e inequality with  $C_P(\mu)$ if
$$
\forall f \in L^2(\mu) \qquad Var_{\mu}(f) \leq  C_P(\mu) \mu(\vert\nabla f\vert^2).
$$
 \end{defi}

We   remind an important result obtained in \cite{Bobkov-aop} (see also 
 \cite{Poincare_log_concave}) that establishes the existence of a Poincaré inequality for every log-concave probability distribution.
 \begin{thm}[\label{theo:log-concave_poincare}\cite{Bobkov-aop,KLS}]
 Every log-concave measure $\mu$ satisfies a Poincar\'e inequality: a universal constant $K$ exists such that:
$$
 C_P(\mu) \leq 4 K^2 Var_{\mu}(\id),
 $$
 where $\id$ refers to the identity map: $\id: \xi \longmapsto \xi$.
 \end{thm}
Since $(\xi,\theta) \longmapsto U(\xi,\theta)$ is a convex function, 
Theorem \ref{theo:log-concave_poincare} implies  that for all $\theta \in \mathbb{R}^d$, $\mathbb{P}_{\theta}$ satisfies a Poincaré inequality of constant $C_P(\mathbb{P}_{\theta})$.
We introduce   an assumption that stands for a uniform   bound of the collection of Poincaré constants $C_P(\mathbb{P}_{\theta})$.

\begin{ass}[Uniform Poincaré Inequality $\UPI$\label{ass:upi}]
A constant $\CPU$ exists such that
$$
\forall \theta \in \R^d \qquad C_P(\mathbb{P}_{\theta}) \leq \CPU.
$$
\end{ass}

We emphasize that according to Theorem \ref{theo:log-concave_poincare}, a uniform bound on the variance of each distribution $\mathbb{P}_{\theta}$ over $\theta \in \R^d$ entails $\UPI$.
{Note that the uniform upper bound on the Poincar\'e constant involved in the family of distributions $(\pi_{\theta})_{\theta \in \theta}$ is needed to obtain some concentration rates that are independent from the value of $\theta$.}



\subsection{Bayesian consistency\label{sec:sub:bayes}}
\subsubsection{Assumptions $\AL$ and $\IWC$}




We introduce some mild assumptions necessary to
obtain some consistency  rates of $\widetilde{\theta}_n$.
First, we  handle smooth functions $(\xi,\theta) \longmapsto U(\xi,\theta)$ 
and assume that:
\begin{ass}[Assumption $\AL$]
$U$ satisfies the $\mathcal{C}^1_L$ hypothesis: \textit{i.e.} the partial gradient of $U$ with respect to $\theta$ is a L-Lipschitz function of $\theta$ and of $\xi$:
$$
\forall \xi  \in \mathbb{R}^q \quad \forall \theta_1,\theta_2 \in \R^d: \qquad 
|\nabla_\theta U(\xi,\theta_1)-\nabla_\theta U(\xi,\theta_2)| \leq L |\theta_1-\theta_2|,
$$
and
$$
\forall (\xi_1,\xi_2)  \in \mathbb{R}^q \quad \forall \theta  \in \R^d: \qquad 
|\nabla_\theta U(\xi_1,\theta)-\nabla_\theta U(\xi_2,\theta)| \leq L |\xi_1-\xi_2|.
$$
\end{ass}

{\begin{rmq} Let us mention that Assumption $\AL$ is standard in the optimization community (see \cite{Nesterov,Bubeck}). Essentially, this allows to quantify the error made when using a first order Taylor expansion. In optimization theory, $L$-smooth functions are then used to produce descent lemmas. Here it will be used for lower-bounding the normalizing constants of the Bayesian posterior distributions (see Equation \eqref{eq:minoration_intermediaire}).
This assumption also appears in the recent contribution \cite{Jordan-Wainwright} but the L-smooth property is only assumed for the function $\theta \longmapsto \mathbb{E}_{\theta^\star}[U(\xi,\theta)]$ associated with either a strong or weak convex assumption on $U$  (Assumptions (S.1) or (W.1) of \cite{Jordan-Wainwright}).
\end{rmq}}

%
To make the estimation problem tractable, we need to manipulate statistically identifiable models.
If statistical identifiability is a free result for any $L^1$ location model when $U(\xi,\theta)=U(\xi-\theta)$, this is no longer the case in general   models, even log-concave ones. We therefore introduce an identifiability assumption, which will useful for our theoretical results.

\begin{ass}[Assumption $\IWC$ - Wasserstein identifiability]
A strictly increasing map $c: \mathbb{R}_+ \longrightarrow \mathbb{R}_+$ and a $1$-Lipschitz function $\Psi$ exist such that $c(0)=0$ and :
$$
{\forall \theta_1,\theta_2\,\in\,\ER^d,\quad}|\pi_{\theta_1}(\Psi)-\pi_{\theta_2}(\Psi)| \geq c(|\theta_1-\theta_2|).
$$
Furthermore, we shall assume that a pair $(b_1,b_2)\in \{\mathbb{R}_+^{\star}\}^2$ and $\ic$>0 exists such that:
\begin{equation}\label{eq:ic}
\forall \Delta \geq 0 \qquad 
c(\Delta) \geq b_1 \Delta^{\ic} \mathbf{1}_{\Delta \leq 1} + b_2 [\log(\Delta)+1] \mathbf{1}_{\Delta \geq 1}. 
\end{equation}
\end{ass}

\begin{rmq}\noindent  Assumption
$\IWC$ is a quantitative identifiability condition. It implies in particular that
$$
\mathcal{W}_1(\mathbb{P}_{\theta_1},\mathbb{P}_{\theta_2}) \geq c(|\theta_1-\theta_2|),
$$
where $\mathcal{W}_1(\mathbb{P}_{\theta_1},\mathbb{P}_{\theta_2})$ refers to the Wasserstein distance between $\mathbb{P}_{\theta_1}$ and $\mathbb{P}_{\theta_2}$. The constants $b_1$ and $b_2$ are not fundamental in our forthcoming analysis but  $\ic$  plays a central role:  it  asserts how the distributions $\pi_{\theta_1+\delta}$ moves from  $\pi_{\theta_1}$ for small values of $\delta$ (see   Section \ref{sec:sub:discussion}).
{It is straightforward to verify that $\IWC$ is satisfied in the location models with $\Psi=\id$ since in that case
$
|\pi_{\theta_1}(\id)-\pi_{\theta_2}(\id)|=|\theta_1-\theta_2|. 
$}

All along our work, we consider $b_1$ and $L$ involved in Assumption $\IWC$ and $\AL$ as fixed quantities, that describe the variations of $U$ and therefore the variations of the distribution $(\mathbb{P}_{\theta})_{\theta \in \mathbb{R}^d}$. Since $b_1$ quantifies how $\pi_{\theta}$ moves when $\theta$ moves, a link exists between $b_1$ and $L$ since oppositely $L$ quantifies in $\AL$ an upper bound of variation of $U(\theta,.)$ itself, and therefore an upper bound of variation of $\pi_{\theta}$. Even though an interesting question that is related to contamination models (see \textit{e.g.} \cite{Cai-Wu,Laurent-Marteau-Maugis}), we have chosen to leave  open the analysis between $b_1$ and $L$, and simply assume in what follows that $b_1>0$ and $L>0$ are fixed constants independent from $n$ and $d$.
\end{rmq}

\begin{rmq}
Our identifiability condition with the Wasserstein distance is different from what is usually studied with mixture models where label permutations need to be considered. In mixture models, $\theta$ commonly refers to a collection of weights of components $(\omega_i)_{1 \leq i \leq K}$ +locations of components $(p_i)_{1 \leq i \leq K}$ and identifiability is then related to a conditions that is close to:
$$
d(\mathbb{P}_{\theta_1},\mathbb{P}_{\theta_2}) \geq \mathcal{W}_1(\mu_{\theta_1},\mu_{\theta_2}) \quad \text{with} \quad \mu_{\theta} = \sum_{i=1}^K \omega_i(\theta) \delta_{p_i(\theta)},
$$
where $\mathcal{W}_1(\mu_{\theta_1},\mu_{\theta_2})$ is the Wasserstein distance between the mixing distributions $\mu_{\theta_1}$ and $\mu_{\theta_2}$.
We refer to \cite{Ho1,Ho2} and the references therein for some recent statistical work on identifiability and consistency rates in finite mixture models. Another popular subject of investigation is the effet of the size of the weights $(\omega_i)_{1 \leq i \leq K}$ on the identifiability and estimation, in particular in contamination models. We refer for example to \cite{Gadat_mix} for a detailed understanding of the effect of the contamination level on the identifiability of Gaussian mixture models, and therefore on $b_1$ involved in $\IWC$.
However, we emphasize that the model we are considering are log-concave, which is far from being the case with mixtures.
\end{rmq}

\begin{rmq}
Besides our log-concave framework, we could imagine use $\IWC$ in more general statistical models such as location-scale mixture models even though these models are from being log-concave and much more challenging in terms of identifiability. We refer to \cite{Gadat_mix,Heinrich-Kahn} and the references therein for some recent contributions in this direction. As an example, using the dual formulation of the $W_1$ distance, it may be shown that for a location/scale Gaussian model $\mathcal{N}_{\mu,\sigma^2}$, $\ic=1$ for the dependency on the means $|\mu_1-\mu_2|$. Regarding the dependency in terms of $(\sigma_1,\sigma_2)$, the same strategy with $f(x)=|x|$ yields $\ic=1$ for $|\sigma_1-\sigma_2|$. In the multivariate setting, the situation is already more difficult, and we refer to \cite{Dowson} for an exact value of the $W_2$ distance that upper bounds the $W_1$ one and it seems reasonnable to think that $\ic=1$ in this situation for location/scale Gaussian models.
\end{rmq}

\begin{rmq}
{ This natural separation assumption is related to the hypothesis testing in the statistical model. Statistical test has a longstanding history in Bayesian literature (see \textit{e.g.} \cite{LeCam,Birge,GhosalGhoshvdVaart,CasvdV} among others). In general, the former papers build some global statistical tests using metric considerations with covering arguments on the statistical models. In particular this is done using Hellinger distance or the Kullback-Leibler divergence. 
Here, our assumption is related to a separation with the help of the $\mathcal{W}_1$  distance over  $\mathbb{P}_{\theta}$.  The function $\Psi$ involved in $\IWC$ will be then used to build a global test.}
\end{rmq}

In a sense, a link  exists between the conjunction of $\AL + \IWC$ and a metric complexity (in terms of covering numbers) as used in the seminal contribution \cite{GhosalGhoshvdVaart}[Equation (2.2)]. In particular, it is  straightforward to prove that if $N(\epsi,\theta\cap K,d_{KL})$ is the covering number of the statistical model with the KL  divergence and if $K$ is a compact subset of $\R^d$, then
$$
\log N(\epsi, K,d_{KL}) \lesssim d \log(\sqrt{L} \epsi^{-1}).
$$
Hence, $\IWC$ shall be thought of as a way to both ``compactify'' the space where   $\theta$ is living and  make a local link between $d(\mathbb{P}_\theta,\mathbb{P}_{\theta^\star})$ and $|\theta-\theta^\star|$. This is useful to avoid sieve considerations (see \textit{i.e.} \cite{GhosalGhoshvdVaart,vZvdV,ShenWasserman} for example) and {this} allows to quantify the tail behaviour of the posterior distribution $\pi_n$ far away from $\theta^\star$ (see \textit{e.g.} \cite{CasvdV}).

Finally, $\IWC$ { also has} a tight link with Assumption $(W.1)$ of \cite{Jordan-Wainwright} that assumes that if $F(\theta):= \E_{\theta^\star} [\log \pi_\theta(\xi)]$, then for $h$ a non-decreasing convex function such that $h(0)=0$:
$$
\langle \nabla F(\theta),\theta^\star-\theta \rangle \ge h(|\theta-\theta^\star|),
$$
For any $\theta \in \R^d$, we introduce $\theta_t=\theta^\star+t(\theta-\theta^\star)$ and $f_\theta(t)=KL(\mathbb{P}_{\theta_t},\mathbb{P}_{\theta^\star})$.
 A straightforward computation shows that Assumption $(W.1)$ implies that:
\begin{align*}
KL(\mathbb{P}_{\theta},\mathbb{P}_{\theta^\star})& = f_\theta(1)= f_\theta(0)+\int_{0}^1 f'_\theta(s) \text{d} s \\
& = \int_{0}^1 \langle \nabla F(\theta_s),\theta^\star-\theta \rangle \text{d}s  = \int_{0}^1 \frac{\langle \nabla F(\theta_s),\theta^\star-\theta_s \rangle}{s} \text{d}s \ge \int_{0}^1 \frac{h(s |\theta-\theta^\star|) }{s} \text{d}s \\
& \ge \int_{0}^1 \frac{h(0)+ h'(0) s |\theta-\theta^*|}{s} \text{d}s = h'(0) |\theta-\theta^*|,\\
\end{align*}
where we used in the last line the convexity of $h$. The previous inequality is not trivial when $h'(0)>0$.
Oppositely, if $h'(0)=0$, it is reasonable to assume that $h$ is $\beta$-H\"older around $\theta^\star$ with $\beta>1$ and we obtain again a local inequality of the form $KL(\mathbb{P}_{\theta},\mathbb{P}_{\theta^\star}) \gtrsim  |\theta-\theta^\star|^{\beta}$ near $\theta^\star$. In the meantime, since $h$ is non-decreasing, it is also possible to use
$
\int_{0}^1 \frac{h(s |\theta-\theta^\star|) }{s} \text{d}s \ge \int_{1/2}^1 \frac{h(s |\theta-\theta^\star|) }{s} \text{d}s \ge h\left(\frac{|\theta-\theta^\star|}{2}\right)$. Therefore, $(W.1)$ implies that a $\beta \ge 1$ exists such that:
$$
KL(\mathbb{P}_{\theta},\mathbb{P}_{\theta^\star}) \gtrsim |\theta-\theta^\star|^\beta \wedge  h\left(\frac{|\theta-\theta^\star|}{2}\right).
$$
The link between $\IWC$ and $(W.1)$ is then made with the help of functional inequalities: if  $\mathbb{P}_{\theta^\star}$ is strongly log-concave, then $\mathbb{P}_{\theta^\star}$ satisfies the $T_1$ inequality (see \textit{e.g.} \cite{Ledoux}) and $W_1(\mathbb{P}_{\theta},\mathbb{P}_{\theta^\star}) \lesssim \sqrt{KL(\mathbb{P}_{\theta},\mathbb{P}_{\theta^\star})}$. More generally, if $P_{\theta^\star}$ satisfies a sub-Gaussian concentration inequality, \cite{Bobkov-Gotze} shows that this Talagrand inequality still holds. Hence, both $\IWC$ and $(W.1)$ of \cite{Jordan-Wainwright} implies a lower bound on the KL-divergence in many log-concave reasonable situations, and the existence of convenient statistical tests.



\subsubsection{Bayesian consistency \label{sec:sub:bayes}}
\bigskip\hphantom{}

\noindent\textit{Upper bound}
{We obtain the next result \textit{ which is still valid besides the log-concave settings}. We have chosen to keep this setting for the sake of readability, even though $\UPI$ is sufficient here to guarantee the result.}

The next result states an upper bound on the $L^p$ loss between the posterior mean $\widetilde{\theta}_n$ and $\theta^{\star}$. We define
\begin{equation}\label{def:epsn}
 \epsi_n^2 := \left(\CPU L^2\right)^{1/\ic}  \left( \frac{d \log n}{n}\right)^{1/\ic},
 \end{equation}
and consider situations where $n$ is large enough (with respect to $d$) in order to  guarantee that $b_1 \epsi_n^{\ic} \leq 1$ where $b_1$ is defined in $\IWC$. We then have the following result.

\begin{thm}\label{theo:consistency} If $\pi_0 =e^{-V_0}$ is a $\mathcal{C}^1(\R^d,\R)$ log-concave prior with $V_0 \in \mathcal{C}^1_1$, if $\UPI$, $\AL$ and $\IWC$ hold, and if $(d,n)$ are such that $b_1 \epsi_n^{\ic} \leq 1$, then:
$$
\forall p > 1 \qquad 
\left(\mathbb{E }_{\theta^{\star}} \left[\left|\widetilde{\theta}_n-\theta^{\star}\right|^p\right] \right)^{1/p} \luc
K(U)  
\left( \frac{  d\log(n)}{n}\right)^{{\frac{1}{2 \ic}}},
$$
with  $K(U)= \left(\sqrt{\CPU} L\right)^{\frac{1}{\ic}}$
and $\CPU$ is the Poincaré constant given in  $\UPI$ (Assumption \ref{ass:upi}).
\end{thm}

\begin{rmq}\label{rem:d_vs_n}
In particular, when $p=2$, we recover the standard mean square error rate.
We emphasize that our consistency result makes sense only when considering some situations where the dimension $d$ is smaller than $n$ the number of observations, so that $\varepsilon_n$ is asymptotically vanishing when $n \rightarrow + \infty$. Therefore, our model is not concerned by high dimensional situations where $d$ may be larger than $n$.

If the separation provided by $\IWC$ is sharp, \textit{i.e.} when $\ic=1$, the $L^2$ loss is proportional to {$\sqrt{\frac{d}{n}}$} (up to a $\log$-term), which is the optimal loss in many statistical models. Our upper bound is deteriorated when $\ic$ increases, \textit{i.e.} when the separation of the distributions $\mathbb{P}_{\theta}$ near $\mathbb{P}_{\theta^{\star}}$ is ``flat'', \textit{i.e.} when  the Wasserstein distance $\mathcal{W}_1(\mathbb{P}_\theta,\mathbb{P}_{\theta^{\star}}) \sim b_1 \|\theta-\theta^{\star}\|^{1+\epsi}$ for $\epsi>0$ near $\theta^{\star}$. Below in Theorem \ref{theo:lower_bound}, we extend this optimality result to a larger range of values for $\ic$.
\end{rmq}
 It is worth mentionning the recent work of \cite{Jordan-Wainwright} that provides a similar analysis and derives a similar contraction rate: 
Using the notations of \cite{Jordan-Wainwright}, we observe that the  value $\alpha$ of their Corollary 1 corresponds to $\alpha_c+1$, essentially because of the definition of $F$ that is the Kullback Leibler divergence between $\pi_{\theta}$ and $\pi_{\theta^\star}$ and remarking that
$$
\langle \nabla F(\theta),\theta-\theta^\star \rangle = \mathbb{E}_{\theta^\star}[\langle \theta-\theta^\star,\partial_{\theta} U(\xi,\theta)] \lesssim |\theta-\theta^\star|^{1+\alpha_{c}}.
$$
Therefore, in the worst situation where $\beta=0$ in \cite{Jordan-Wainwright},  the associated rate of convergence is $(d/n)^{\frac{1}{2(\alpha-1)}} = (d/n)^{\frac{1}{2\alpha_c}}$, which is exactly the  result we obtain in Theorem \ref{theo:consistency}.
 In particular, in the situation of the Bayesian logistic regression, the noise inherited from this model is bounded so that $\beta=0$, which leads to the same result both for Theorem \ref{theo:consistency} and Corollary 1 of \cite{Jordan-Wainwright}.

\bigskip
 				
 				\noindent
\textit{Lower bound}
  In order to assess the accuracy of the Bayesian posterior mean estimator, we study and derive a lower bound of estimation following the minimax paradigm. 
In this view, we introduce  $\mathcal{F}_{\ic}^L$ the set of
 models $(\mathbb{P}_{\theta})_{\theta \in \R^d}$ such that $\IWC$, $\UPI$ and $\AL$ hold.}

We then define $r_{n,d}(\ic)$ as the following \textit{minimax risk} given by:
$$
r_{n,d}(\ic) := \inf_{\hat{\vartheta}_n} \sup_{\mathcal{F}_{\ic}^L} \E\left[\|\hat{\vartheta}_n - \theta^\star\|_2\right],
$$
where the above infimum is taken over the set of all possible estimators
$\hat{\vartheta}_n$ constructed from the observations $\mathbf{\xi^n}$. Hence,
$r_{n,d}(\ic)$  is the best achievable $L^2$ risk of estimation of $\theta^\star$ in the worst situation  under the assumptions of Theorem \ref{theo:consistency}. We obtain the following result.

\begin{thm}\label{theo:lower_bound} Over the class $\mathcal{F}_{\ic}^L$, a constant $\kappa(L)$ exists such that:
$$
r_{n,d}(\ic) \ge \kappa(L) \left( \frac{d}{n} \right)^{\frac{1}{2\ic}}
$$
\end{thm}
We shall conclude from Theorem \ref{theo:consistency} and Theorem \ref{theo:lower_bound} 
that the Bayesian posterior mean estimator achieves an optimal rate of estimation in terms of $n$ and $d$ up to a $\log(n)^{1/\ic}$ term. 

\subsubsection{Comments}		
We provide below several comments on Theorem \ref{theo:consistency}.
First, we emphasize that our results apply in any log-concave situations included the weakly convex case.
					 	
\noindent 
Theorem \ref{theo:consistency} describes the behaviour of the posterior \textit{mean} $\int \theta \text{d} \pi_n(\theta)$ and not of the entire posterior distribution as classically studied in Bayesian statistics. In a sense, such a result seems less informative than the knowledge of the behaviour of the entire distribution $\pi_n$. However, a good behaviour of the posterior mean requires a sharp control of the tails of the posterior distributions whereas a ``contraction rate'' relative to the Hellinger distance (or with other distances such as the Kullback-Leibler or total variation ones) sometimes blur{s} the tail behaviour of the posterior. We refer to \cite{CasvdV} for an illustration in high dimensional linear models of the efforts needed to extend posterior concentration to posterior mean consistency.

\noindent 
If we now pay attention to the convergence rate obtained in Theorem \ref{theo:consistency}, we emphasize that when the separation is sharp, \textit{i.e.} when $\alpha_c=1$ in $\IWC$, we obtain the standard minimax convergence rate $\frac{d}{n}$ up to a $\log(n)$ term, and this result is in a sense non-asymptotic (a universal constant could be exhibited with a price of a huge technicalicity). In comparison with \cite{Jordan-Wainwright}, we also obtain{ed} a slightly better convergence rate  of $(d/n)^{1/\ic}$  (see Corollary 1 of \cite{Jordan-Wainwright}). But in general situations, their value of $\beta$ is equal to $1$   so that the rates derived from Theorem \ref{theo:consistency} and Corollary 1 of \cite{Jordan-Wainwright} are equivalent.

\subsection{Discretization in the weakly convex settings} 

 The main contribution of our paper is to carry out a precise study of the weakly convex situation (\textit{e.g.} when the underlying function is convex but not strongly convex). If we could derive the convergence rate of the continuous Cesaro average with the help of simple functional inequalities and Poincare constants, it is no longer the case when dealing with discretized Cesaro average. To bypass this difficulty, we are then led to introduce some functional inequalities that are slightly more parametric (see $\Hcu$ below).
\subsubsection{General result} {In this section, we focus on the discretization procedure in the general weakly convex situation.}
 We still assume   $\AL$  and introduce for any convex function $W$ {a set of inequalities that are related to the Kurdyka-\L ojasiewicz inequality (see \textit{e.g.} \cite{Lojasiewicz,Kurdyka}). This type of inequality is widely used 
in optimization and geometry to extend results from the strongly convex case to the more difficult weakly convex settings (see \textit{e.g.} \cite{ Bolte,Gadat-Panloup}).}
{For this purpose, for any matrix $B$, we denote  by ${\rm Sp}(B)$ the spectrum of B.}
We set:
\begin{equation}\label{def:lambda}
\bar{\lambda}_{\hesdob}(x) := \sup \{{\rm Sp}(\nabla^2 \dob(x))\} \quad \text{and} \quad \underline{\lambda}_{\hesdob}(x) = \inf \{{\rm Sp}(\nabla^2 \dob(x))\}.
\end{equation}
 

{\begin{ass}[Assumption $\Hcu$]  {$\dob$ is a positive ${\cal C}^2$-function with $\min \dob=_{uc} 1$} and there exist $0\le q\le r<1$ {and positive $\deux$ and $\troiss$} such that,
$$\forall x\in\ER^d,   \qquad \deux \dob(x)^{-r}\le \underline{\lambda}_{\hesdob}(x)\le \bar{\lambda}_{\hesdob}(x)\le \troiss \dob(x)^{-q}.$$
\end{ass}}

{Assumption $\Hcu$ implies that (see Lemma \ref{lem:gradW} in appendix for details)
$$ \limsup_{|x|\rightarrow+\infty} |\nabla W|^2 W^{q-1}(x)<+\infty\quad \textnormal{and}\quad  \liminf_{|x|\rightarrow+\infty} |\nabla W|^2 W^{r-1}(x){>-\infty}.$$
In the case $r=q$, it implies a global standard Kurdyka-\L ojasiewicz assumption (see \textit{e.g.} \cite{Gadat-Panloup} for a recent application for stochastic gradient descent, and the seminal contributions of \cite{Lojasiewicz}).}
\noindent  The case $r=1$ corresponds to the limiting Laplace case, whereas $r=0$ (and $q=0$) is associated to the strongly convex case. In particular, the complexity of the procedure will of course increase with $r$. {Finally, note that since $\nabla W$ is $L$-Lipschitz,
the eigenvalues of $\nabla^2 W(x)$ are uniformly bounded, which implies  that 
$q\ge0$ in $\Hcu$ is given by $\AL$.
Assumption 2.4 $\Hcu$  relates the behaviour of the spectrum of the Hessian (which translates some information on the curvature of the landscape) with the size of the potential itself and
 $r>q \ge 0$ is associated to some lack of strong convexity in some directions of the landscape function. Interestingly, thanks to $\Hcu$ and the convexity assumption, some tights relationship exist between $W(x)$, $|\nabla W(x)|$ and the eigenvalues of $\nabla^2 W(x)$.
}

\begin{rmq}\label{rmqexample1} 
\noindent Let us consider $\Hcu$ with $\dob(x)=(1+|x|^2)^p$ with $p\in(1/2,1]$. We verify that
$$\nabdob(x)=2p  W(x)^{(p-1)/p }x\quad\textnormal{and}\quad (\hesdob(x))_{ij}=4p(p-1) x_i x_j W(x)^{(p-2)/p}+2p\delta_{ij}W(x)^{(p-1)/p}.$$
Moreover, for any vector $u$ with $|u|=1$,
$$\langle \hesdob(x) u, u\rangle =2p(1+|x|^2)^{p-1}\left(1-2(1-p)\frac{\langle x,u\rangle ^2}{1+|x|^2}\right).$$
It implies that:
$$\underline{\lambda}_{\hesdob}(x)\ge 2p(1-2(1-p))(1+|x|^2)^{p-1}=2p(1-2(1-p)) \dob(x)^{-\frac{1-p}{p}},$$
and 
$$\overline{\lambda}_{\hesdob}(x)\le  (1+|x|^2)^{p-1}=2p \dob(x)^{-\frac{1-p}{p}}.$$
This entails $\Hcu$ with $r=q=(1-p)/p$.

\end{rmq}
Under $\Hcu$, we are able to obtain exponential  bounds for the Euler scheme denoted by  $\Yg$ and for the diffusion $X$ (see  Lemma \ref{lem:expbounds1cont} and Proposition \ref{lem:expbounds2scheme} of in Appendix B),  which in turn involve the following quantitative controls of the moments. 

\begin{prop} 
\label{lem:expbounds22}  Assume $\Hcu$.  {Assume that $\dob(x^\star)\le 1$}. Suppose that  $\gamma\le\gamma_0:= \frac{1}{4 d L + 1}$, then
 a locally finite positive function $p\mapsto c_p$ on $(0,+\infty)$, independent of $\dob$,  such that for any $p>0$,
$$\sup_{t\ge0} \E_x[\dob ^p (X_{t})]+\sup_{t\ge0} \E_x[\dob ^p (\Yg_{t})]\le c_p \left(  \dob^p (x)+\dw^p\right),$$
where $\dw$ is a positive number that satisfies:
$$1\le \dw\le c_r {(\troiss\vee L)^{\frac{1}{1+q-r}}}{\deux^{-\frac{1}{1-r}}}\log(1+dL) d^{\frac{1}{1+q-r}},$$ 
where $c_r$ is a constant depending only on $r$.
\begin{rmq} Note that in the particular case $q=0$, we obtain $\dw\le c_r (L/\deux)^{\frac{1}{1-r}} \log(1+dL) d^{\frac{1}{1-r}}$ so that when $r=q=0$ (that corresponds to the strongly convex setting) we obtain $\dw\,\propto \,\frac{L}{\deux} d\log (1+dL)$.
It is worth saying that the previous bound is universal, but may be improved in some situations with specific statistical models.
\end{rmq}
\end{prop}
%


%
%
{
We are now able to state our first bounds on the complexity of the discretization procedure. This general and technical result is divided into two cases $(i)$ and $(ii)$ leading respectively to a complexity of order $\varepsilon^{-4}$ and $\varepsilon^{-3}$ (the second case requires additional assumptions on $\dob$). Then, each case is divided into two parts $(a)$ and $(b)$. In Statement $(a)$, the result gives the dependency on the whole set of parameters whereas $(b)$ only focuses on the dimension $d$ (in order to provide a bound which is less precise but more simple to read). 
}

\begin{thm}\label{cor:thmPoisson}
Assume  {$\AL$}, $\Hcu$ and that $\gamma\le \gamma_0:= \frac{1}{8}((d (L\vee \troiss))^{-1}\wedge \frac{1}{8})$. Let $f$ be a ${\cal C}^2$ function from $\ER^d$ to $\ER^{d'}$ with $d'\le d$, $\flip\le 1$ and $\dflip\le 1$.
Let $\varepsilon>0$ and $\mte>0$.\\

\begin{itemize}
\item $(i.a)$  If $\gamma_\varepsilon =  \frac{\deux^{2(1+\mte)}}{d L^2 \dw^{2r(1+\mte)}} \varepsilon^2 \wedge 
\frac{d}{\troiss \dw^{1-q-2r \mte}}  \wedge
\frac{d}{\troiss 
(1+q)^{\frac{1-q}{1+q}} \troiss^{\frac{1-q}{1+q}} \dw^{-2r \mte}|x_0-x^\star|^{2 \frac{1-q}{1+q}}}
 \wedge
\frac{\dw^{2r(1+\mte)+q-1} d L^2}{\troiss \deux^{2(1+\mte)}}$
and

$$
N_{\varepsilon} \ge \dw^{\frac{1+3r}{2}(1+\mte)} \deux^{-\frac{3}{2}(1+\mte)} \gamma_{\varepsilon}^{-1} \epsilon^{-1} \vee d' \dw^{2r(1+\mte)} \deux^{-2(1+\mte)} \gamma_{\varepsilon}^{-1} \varepsilon^{-2}
$$
%
%

 
\noindent
  then,  a constant $c_{\mte,r,q}$  exists independent from $\deux$, $\troiss$, $L$, $d$, $\varepsilon$ such that:
\begin{equation} \label{eq:MSEleeps}
 {\cal E}_\varepsilon:=\E_{x_0}\left[\left|\frac{1}{N_\varepsilon}\sum_{k=1}^{N_\varepsilon} \bX_{k\gamma}^{\gamma}-\pi(\id)\right|^2\right]\le c_{\mte,r,q}{\varepsilon^2}\quad\textnormal{{with $\pi\,\propto \, e^{-\dob}$}}.
 \end{equation}

\item $(i.b)$ If $\varepsilon \in (0,1\wedge (d' d^{-\frac{1-r}{2(1+q-r)}})]$, $\gamma_\varepsilon=\varepsilon^{2}d^{-(1+\frac{2r}{1+q-r}+\mte)}$ and $N_\varepsilon=\varepsilon^{-4} d' d^{1+\frac{4r}{1+q-r}+\mte}$, then
 ${\cal E}_\varepsilon\le c\varepsilon^2$ where $ c$  only depends on $\mte,r,q,\deux,\troiss$ and $L$ (but not  on $d$ and $\varepsilon$).\\
\item  $(ii.a)$ Assume furthermore that $\dob$ is ${\cal C}^3$ with $D^2 W$ $\tilde{L}$-Lipschitz for $\nsp{.}$ and that
$$\|\vec{\Delta}(\nabdob)\|_{2,\infty}^2:=\sup_{x\in\ER^d} |\vec{\Delta}(\nabdob)(x)|^2<+\infty,$$
where $\vec{\Delta}(\nabdob)=(\Delta \partial_{1} \dob,\ldots,\Delta \partial_{1} \dob)$. Then, if $\gamma_\varepsilon= c_{2,1} \varepsilon$ with 
$$c_{2,1}=  \deux^{1+\mte}\min\left(\frac{1}{L \troiss^{\frac{1}{2}} \dw^{\frac{1-q+2r}{2}(1+\mte)}},\frac{ \deux^{1+\mte}}{L\tilde{L} d \Upsilon^{2r(1+\mte)}},\frac{1}{\|\vec{\Delta}(\nabdob)\|_{2,\infty} \dw^{r(1+\mte)}}\right),$$
and 
$$c_{1,2}=\max\left( \frac{d' \Upsilon^{2r(1+\mte)}}{\deux^{2(1+\mte)}},\varepsilon \frac{\Upsilon^{\frac{1+3r}{2}(1+\mte)}}{\deux^{\frac{3}{2}+\mte}}  \right),$$
and $ N_\varepsilon=\frac{c_{1,2}}{c_{2,1}} \varepsilon^{-3}$, then the conclusion of $(i.a)$ holds true.\\
\item 
$(ii.b)$ Under the assumptions of $(ii.a)$, assume that a constant $C$ independent on $d$ exists such that $\|\vec{\Delta}(\nabdob)\|_{2,\infty}^2 \leq C d^{2\rho} $ with $\rho\ge0$. 
If $\varepsilon\in (0,1\wedge (d' d^{-\frac{1-r}{2(1+q-r)}})]$,
 $\gamma_\varepsilon=\varepsilon d^{-\max(1+\frac{2r}{1+q-r},  \rho+\frac{r}{1+q-r})-\mte}$ and $N_\varepsilon=\varepsilon^{-3}d' d^{\max(1+\frac{4r}{1+q-r},\rho+\frac{3r}{1+q-r})+\mte}$, then the conclusion of $(i.b)$ holds true with a positive constant $c$ which only depends on $\mte,r,q,\deux,\troiss$, $L$ and $\tilde{L}$ (but not  on $d$ and $\varepsilon$).\\

\end{itemize}

%

%
%
%
\end{thm} 
\begin{rmq} \label{rmqlaplacien} {Since $\|\vec{\Delta}(\nabdob)\|_{2,\infty}^2={\sup_{x\in\ER^d}\sum_{i=1}^{d} (\sum_{j=1}^d \partial_{i,j,j}^3 \dob(x)})^2$,
the condition  $\|\vec{\Delta}(\nabdob)\|_{2,\infty}^2\le C d^{2\rho} $ }reasonably holds with $\rho \le 3/2$ (if each $\partial_{i,j,j}^3 \dob$ is bounded by a universal constant) and may hold with lower $\rho$ when the coordinates of the dynamical system are not ``fully connected'' (for instance if $\dob_i$ depends only on $x_i$, we get $\rho=3/2$).   {When $\rho\le 1+\frac{r}{1+q-r}$ (which is always the case when $r\ge 1/2$),  we obtain the same dependency in $d$ in $(ii.b)$ as in $(i.b)$ but with $\varepsilon^{-3}$ instead of $\varepsilon^{-4}$.}
\end{rmq}

\subsubsection{Application to Bayesian learning}
A consequence of Theorem \ref{cor:thmPoisson} is stated below when we observe a $n$ i.i.d. sample $\xi^{n}$ and when $\varepsilon_n$ refers to the statistical accuracy defined in \eqref{def:epsn} given by $\epsi_n= (d/n)^{\frac{1}{2\ic}}$ (we omit the logarithmic term) and when we implicitely assume that $d$ is smaller than $n$, \textit{i.e.} that the sequence $\varepsilon_n \rightarrow 0$ as $n \rightarrow + \infty$ (see Remark \ref{rem:d_vs_n}).

%





\begin{thm}\label{theo:learning_weak_convex} Assume $\IWC$,   $\AL$ and $\UPI$.
 Suppose that for any $\xi\in\ER^q$, $U(\xi,.)$ satisfies $\Hcu$ with  $\deux$, {$\troiss$, $r$ and $q$} independent of $\xi$. Assume furthermore that $r=q$. For the following tunings of $\gamma_{\varepsilon_n}$ and $N_{\varepsilon_n}$ given by:

\begin{itemize}
\item[$(i)$] If  $r \ge 1$, then $\gamma_{\varepsilon_n} = d^{-(2r+1-\ic^{-1})}  n^{-2r-\ic^{-1}},$
and $N_{\varepsilon_n}  = n^{2 \ic^{-1}+4r} d^{1+4r-2\ic^{-1}}.$

\item[$(ii)$] If  $ r=q=0$, then
$\gamma_{\varepsilon_n} =n^{-2}$
and
$
{N_{\varepsilon_n}} =  n^{\frac{1}{2}(1+\ic^{-1})} d^{\frac{1}{2}(1-\ic^{-1})}.
$
\item[$(iii)$] For any $r \in (0,1)$, we have
$
\gamma_{\varepsilon_n}= n^{-(2r+\ic^{-1})}d^{-1-2r+\ic^{-1}} \wedge n^{-\frac{2}{1+r}} d^{\frac{2r}{1+r}},
$
and 
$$
N_{\varepsilon_n} = \gamma_{\varepsilon_n}^{-1} \left[ n^{\frac{1}{2}\ic^{-1}-\frac{3}{2}(1-r)} d^{\frac{1+3r}{2} -\frac{1}{2} \ic^{-1}}\vee d^{2r-\ic^{-1}} n^{\ic^{-1}-2(1-r)}\right].
$$
\end{itemize}
we can find a constant $C$ that depends on $q,r,\deux,\troiss$ but not on $d$ and $n$ such that:
$$
\mathbb{E}[\|\tgno-\theta^\star\|^2] \leq C \varepsilon_n^2
$$
\end{thm}

\subsection{Discretization in the strongly convex setting} 
 This paragraph concerns the $L$-smooth  and strongly convex case, which is defined as follows. 
\begin{ass}[Assumption $\SC$]
We assume that for any $\xi \in \R^q$, $U(\xi,.)$ is $\rho$-strongly convex: \textit{i.e.}, for any $\xi \in \R^q$:
$$
\forall z \in \R^d \quad  \forall \theta \in \R^d: \qquad ^t z \nabla^2_{\theta} U(\xi,\theta) z \ge \rho |z|^2.
$$
\end{ass}
\begin{thm}\label{cor:thmPoissonSC}
Assume  {$\AL$} with $L\ge 1$ and $\SC$.  Let $\Yg$ denote the Euler scheme with constant step-size $\gamma$ initialized at $x_0$. 
 Let $f$ be a ${\cal C}^2$ function from $\ER^d$ to $\ER^{d'}$ with $d'\le d$,  $\flip\le 1$ and $\dflip\le 1$ and consider $\varepsilon>0$

\noindent  (i)  If $\gamma_\varepsilon=  \frac{\rho^2}{L^2 d } \varepsilon^2 \wedge \frac{\rho}{L^2}  \wedge \frac{d}{L^2 |x_0-x^\star|^2}$ and  
$
N_\varepsilon \ge d' \rho^{-2}  \gamma_{\varepsilon}^{-1} \varepsilon^{-2} \vee  \sqrt{d} \rho^{-3/2} \varepsilon^{-1} \gamma_{\varepsilon}^{-1}
$
, then:
\begin{equation*}
 \E_{x_0}\left[\left|\frac{1}{N_\varepsilon}\sum_{k=1}^{N_\varepsilon} \bX_{k\gamma}^{\gamma}-\pi(\id)\right|^2\right]\lesssim_{uc} {\varepsilon^2}.
 \end{equation*}

\noindent
 (ii) Assume furthermore that $\dob$ is ${\cal C}^3$ with $D^2 W$ $\tilde{L}$-Lipschitz for $\nsp{.}$ and that
$\|\vec{\Delta}(\nabdob)\|_{2,\infty}^2<+\infty.$
If $\gamma_\varepsilon=  \mathfrak{b}_2^{-\frac{1}{2}} \varepsilon$ with:
$$ \mathfrak{b}_2=
d^2+\frac{L^4 d}{\rho^3}+
L^2 {\left( \frac{1 }{\rho^2}+ \frac{\tilde{L}^2}{\rho^4}\right)}
+
\frac{ \|\vec{\Delta}(\nabdob)\|_{2,\infty}^2 }{\rho^2  } 
+{2   \frac{L^4 }{\rho^2}  |x_0-x^\star|^2},$$
and $ N_\varepsilon=\varepsilon^{-3} \frac{\sqrt{\mathfrak{b}_2} d'}{\rho^2} \vee \gamma_{\varepsilon}^{-1} \vee d \gamma_{\varepsilon}^{-1} \rho^{-1}$, then:
\begin{equation*}
 \E_{x_0}\left[\left|\frac{1}{N_\varepsilon}\sum_{k=1}^{N_\varepsilon} \bX_{k\gamma}^{\gamma}-\pi(\id)\right|^2\right]\lesssim_{uc} {\varepsilon^2}.
 \end{equation*}

\end{thm} 

\begin{rmq}\label{rmqdurmusmus}
{Note that in $(i)$, when $\varepsilon$ is small enough and $x_0$ sufficiently close to $x^\star$, we have  $\gamma_\varepsilon=  \frac{\rho^2}{L^2 d }\varepsilon^2$ . In this case, for a $\mathcal{C}^2$ function, the computational cost induced by $(i)$ is of the order $\mathcal{O}\left( \frac{d d' L^2}{\rho^4} \varepsilon^{-4}\right)$.}
This result slightly improves the bound derived in \cite{durmus} with the same dependency in $d$ and $\varepsilon$ and a better one in terms of $\rho$ and $L$, with a completely different method of proof. We also observe that we recover in this case the limiting case $r=q=0$ of Theorem \ref{cor:thmPoisson} $i)$ and $ii)$ with a sharp dependency in terms of $\rho$ and $L$.

If now the function $f$ is more regular and three times differentiable and if $\|\vec{\Delta}(\nabdob)\|_{2,\infty}^2 \lesssim_{uc} \rho^{-2}(L \tilde{L}d)^{{2}}$ (see Remark \ref{rmqlaplacien} for a discussion on $\|\vec{\Delta}(\nabdob)\|_{2,\infty}^2$), then the computational cost becomes $\mathcal{O}_{id}\left( \frac{d L^2}{\rho^4} \varepsilon^{-3}\right)$. Up to our knowledge, this result is new for deriving an $\varepsilon$ approximation M.S.E. cost.
\end{rmq}

The next theorem gives a quantitative setup to attain an $\epsi_n^2$ accuracy for a discrete LMC procedure with a constant step-size $\gamma$ where  $\epsi_n^2 = (d/n)^{\ic^{-1}}$ is defined by \eqref{def:epsn} in the situation where  $\varepsilon_n \rightarrow 0$ as $n \rightarrow + \infty$.


\begin{thm}\label{theo:learning_strong_convex}
Assume that $\IWC$, $\AL$, $\UPI$ and $\SC$ hold and that $|\theta_0-\theta^\star|^2 \lesssim_{uc} d$. For the following tunings of parameters:
 \begin{itemize}
\item $(i.a)$ if $\gamma_{\varepsilon_n} =  n^{-2}$ and $N_{\varepsilon_n} =  n^{\frac{1}{2}(1+\ic^{-1})} d^{\frac{1}{2}(1-{\alpha_c}^{-1})}$ operations.
\item $(i.b)$ if $|\hat{\theta}^{\gamma}_{n,0} - \theta^\star|^2 \leq d n^{-2}$, $\gamma =  n^{-1}$ and $N = n \vee d$ operations.
\end{itemize} 
then a constant $C$ that depends on $\ic,L,\rho$ but not on $d$ and $n$ exists such that:
$$\mathbf{E}\left[\left|\tgno-\theta^\star\right|^2\right] \leq C \epsi_n^2.$$
\end{thm}
\noindent This result deserves several comments.
\begin{itemize}
\item 
This theorem indicates that the step size should be chosen as  $\gamma \propto n^{-2}$ (first case) or $\gamma\, \propto \, n^{-1} $ (second case), which becomes smaller when $n$ increases. This   is due to the sharper statistical accuracy we can expect with the posterior mean when we have a large amount of observations.


\item A stricking point derived from the previous result is that the tuning
$(\gamma,N) = (n^{-2},n)$ is ``universal'', which means that regardless the value of $\ic$, the computation of the posterior mean with an explicit Euler scheme with step-size $n^{-2}$ and $nd$ iterates leads to an optimal statistical result.
If we now consider that each iteration of the LMC has a cost $d$, the global cost of the procedure is then $ n d$.

 \item We also observe that our results of Theorem \ref{theo:learning_weak_convex} and of Theorem \ref{theo:learning_strong_convex}
 match, \textit{i.e.}, the cost of learning in strongly convex situation is identical to the one obtained under $\Hcu$ when $r=q=0$.

\item We  observe that with Assumption $\SC$, besides the obvious curse of dimensionality for large $d$ concerning the \textit{statistical} accuracy of $\widehat{\theta}_n$, there is no burst of the 
\textit{computational} cost, which remains polynomial in terms of $n$ and $d$. 
\item Finally, it is worth saying that in the previous statement, we did not use the $\mathcal{C}^3$ assumption to assess the cost of the Bayesian learning, essentially because the potential improvement is strongly model dependent owing to the size of $\|\vec{\Delta}(\nabdob)\|_{2,\infty}^2$. It could be certainly helpful in some very specific situations.
\end{itemize}


%
\subsection{Discussion on the discretization results \label{sec:sub:discussion}}

\textit{Polynomial cost.}
Overall, the leading take-home message when considering the discrete approximation and the concrete estimator is that  in both strongly and weakly convex case, we obtain a complexity that evolves as a polynomial of $n$ and $d$, the complexity being much lower in the strongly convex case. 
The worst situation is attained when $r$ is close to $1$, and $\ic = 1$, and we obtain in that case a computational cost of the order $n^{4} d^{3}$  (by applying Theorem \ref{theo:learning_weak_convex} $(ii)$).
All the more, we observe that in our results, the complexity of Bayesian learning is seriously damaged with the loss of strong convexity, both in terms of $n$ and $d$, and this loss is parametrically described with the help of Assumption $\Hcu$.

\textit{Lack of log-concavity.}
 Behind our polynomial cost result, the log-concave assumption plays a central role: withouth a such assumption, sampling has to be considered with specific models and specific algorithms.
 To bypass the absence of log-concavity, a popular strategy relies on the existence of a functional inequality satisfied by the target measure. For example, \cite{Jordan-Flammarion} (see also \cite{Improved_LSI}) assumes the existence of a  log-Sobolev inequality (LSI) and obtain a total variation mixing time of the order $K d \varepsilon^{-2}$. A such LSI is then shown to be verified with the help of a perturbation approach (see \cite{Holley}) for strongly log-concave distributions outside a ball of radius $R$, but the dependency of $K$ may be exponentially large with $R$, so that the overall cost depends on a balance between $d$ and $R$.
   As a new example of the role of functional inequalities, in \cite{Ho3}, the authors obtain a polynomial computational cost with a reflected Metropolis-Hastings random walk, that is well suited for symmetric mixture model (which is not log-concave) and is far different from a pure Langevin strategy.
   
   \noindent Interestingly, our approach relies on the Poisson equation (that could be studied in non-convex situations) whereas  \cite{Ho3}  relies on functional inequalities such as the Poincar\'e one.

\textit{Existing results on MSE.}
As indicated above,  we obtained the complexity to compute an 
$\epsi$-approximation of the posterior mean with the M.S.E. loss, \textit{i.e.} the complexity to guarantee that the M.S.E. becomes smaller than $\epsi^2$. This is the purpose of Theorem \ref{cor:thmPoisson} and Theorem 
\ref{cor:thmPoissonSC}
  respectively in the weakly and strongly convex cases.
   The related orders of complexity  are given in Table \ref{table:complexity}, where we also draw some comparisons with some state of the art results related to the complexity of Bayesian sampling. Up to our knowledge, the only paper that derives an ad-hoc study of the M.S.E. criterion is \cite{durmus}, that obtains a computational cost of the order $d \varepsilon^{-4}$ and $d \varepsilon^{-3}$ in the strongly convex case. If our results are rather similar to those of \cite{durmus} in the strongly convex case (same dependency in terms of $d$ and $\varepsilon$), we largely improve the state of the art in the weakly convex situation.

\textit{Existing results with other criteria.}
   Other results are related to some different criteria and for our purpose, these results are stated up to a normalization factor: actually, in the recent papers \cite{Durmus3,Dalalyan2} (see also \cite{Jordan-acceleration}) that we compare with, the complexity is defined in a slightly different way, omitting the Monte-Carlo factor. More precisely, oppositely to our paper based on a Cesaro average (involving only one path), these papers use a classical Monte-Carlo approach to approximate $\pi(f)$ (for a given function $f$) by $N_{MC}^{-1}(Z_1+\ldots+Z_{N_{MC}})$ where
  $(Z_j)_{1 \le j \le N_{MC}}$ denotes an $i.i.d.$ sequence of  $N_{MC}$ i.i.d. random variables. Then, for a given $\varepsilon$, these papers define the complexity as the number $n_\varepsilon$ of iterations of the Euler-scheme to compute $Z_1$. In order to draw some fair comparisons, we need to consider the ``true'' complexity, $i.e.$ to multiply their complexity $n_\varepsilon$ by $MC(\varepsilon)={\rm Var}(Z_1)\varepsilon^{-2}$, $i.e.$ by the number of Monte-Carlo simulations that are necessary to obtain a Monte-Carlo M.S.E  lower than $\varepsilon^2$. Furthermore, since the involved function is $\id$, it is reasonable to assume that  ${\rm Var}(Z_1)\propto d$, so that we assume that the true complexity of the compared papers is $n_\varepsilon d \varepsilon^{-2}$.
%
%
 
 Finally, former works  state results with different distances: Total Variation, Kullback-Leibler, $W_1$ or $W_2$. We only consider $W_1$ or $W_2$ results in Table \ref{table:complexity}, which {seem to be} the only ones that can apply to the non-bounded (Lipschitz) function $\id$, the KL divergence and TV distance being too weak   to draw some conclusions on the posterior mean approximation. In particular, we do not report in Table \ref{table:complexity} the TV results presented in \cite{Durmus3} or KL results in \cite{Dalalyan}. First, we emphasize that we obtain in our work the best dependency (in terms of $\epsi$) in the weakly-convex situation, thanks to our parametric $\Hcu$ assumption. It would be very tempting to understand the behaviour of KLMC and $\alpha$-KLMC within this framework. 
  Second, we recover the good dependency on $d$ and $\epsi$ of the LMC in the strongly convex situation. At last, LMC is outperformed by KLMC in this same setting as reported in \cite{Dalalyan-Bernoulli} when $\epsi$ is small enough.

%
%
%
%
%

%
%

 \small{
\small{
\begin{table*}\centering
\ra{1.3}
\begin{tabular}{@{}lrrrcrrrcrrr@{}}\toprule
 & \multicolumn{3}{c}{$\SC$ and $\AL$} & \phantom{abc}& \multicolumn{3}{c}{$\Hcu$ (or weakly convex) and $\AL$} &
\phantom{acc} \\
\cmidrule{2-4} \cmidrule{6-8} 
& $\gamma$
 & $N(\epsi)$   &&&  $\gamma$ 
 & $N(\epsi)$   \\ \midrule
This work - LMC - $\mathcal{C}^2$ & $  \mathcal{O}( \frac{\rho^2}{L^2 d} \epsi^{{2}})$   & $   \mathcal{O}( \frac{d L^2}{\rho^4 } \epsi^{{-4}})$  &  & &$  \mathcal{O}(d^{-(1+ \frac{2r}{1+q-r})}  \epsi^{2})$   &$  \cal{O}_{id}(d^{1+ \frac{4r}{1+q-r}} \epsi^{-4})$   \\
This work - LMC - $\mathcal{C}^3$ & $  \mathcal{O}( \frac{\rho}{L^2 d} \epsi^{{}})$    & $   \mathcal{O}( \frac{d L^2}{\rho^4 } \epsi^{{-3}})$ & &  & $  \mathcal{O}(d^{1+ \frac{2r}{1+q-r}} \epsi)$   &$  \cal{O}_{id}(d^{1+ \frac{4r}{1+q-r}} \epsi^{-3})$   \\
  \cite{durmus} - LMC - $\mathcal{C}^2$ & $  \cal{O}(\frac{\rho^3}{d L^4}\epsi^{2})$    &$ \cal{O}(\frac{d L^4}{\rho^4} \epsi^{-4})$  \\
    \cite{durmus} - LMC - $\mathcal{C}^3$ & $  {\cal{O}(\frac{\rho}{d L^2}\epsi)}$    &$ {\cal{O}(\frac{d L^2}{\rho^4} \epsi^{-3})}$  \\
\cite{Durmus3} - LMC  - $\mathcal{C}^2$ & $  {\cal{O}(\frac{\rho}{d L}\epsi^2)}$  
 &$ {\cal{O}(\frac{d^2 L}{\rho^2} \epsi^{-4})}$  \\
  \cite{Dalalyan-Bernoulli} -  KLMC  $\mathcal{C}^2$ & $  \cal{O} (\frac{\rho}{L \sqrt{d}}\epsi)$ 
  & $\cal{O}_{id}(\frac{d^2 L}{\rho} \epsi^{-3})$ \\
  \cite{Dalalyan2} - $\alpha$-KLMC  - $\mathcal{C}^2$ &  &   &  && $  \cal{O}_{id}(d^{-1} \epsi^4)$ 
  & $\cal{O}_{id}(d^3 \epsi^{-7})$\\
  \cite{Dalalyan2} - $\alpha$-KLMC  - $\mathcal{C}^3$   &  &   &  && $  \cal{O}_{id}(d^{-1/2} \epsi^3)$ 
   & $\cal{O}_{id}(d^3 \epsi^{-6})$\\
\bottomrule
\end{tabular}
\caption{Complexity $N(\epsi)$ of an $\epsi$-approximation of the Mean-Squared Error with a constant step-size $\gamma$ of several methods. Left: strongly convex case (with dependency on $d$, $\rho$ and $L$), Right: weakly convex case with dependency on $d$.
LMC: overdamped Langevin Monte Carlo, KLMC: Kinetic Langevin Monte Carlo, $\alpha$-KLMC: $\alpha$ penalized KLMC. 
\label{table:complexity}}
\end{table*}  }

 
%
%
%

\section{Minimax Bayesian posterior mean consistency}\label{sec:bayes}
This paragraph is dedicated to the proof of Theorem \ref{theo:consistency} (Bayesian consistency rate) and of Theorem \ref{theo:lower_bound} (minimax rate).

%
%

\subsection{Poincare inequality and consequences}
 We   state a  famous result for the family $(\mathbb{P}_{\theta})_{\theta \in \mathbb{R}^d}$ of  Bobkov and Ledoux (see \textit{e.g.} \cite{BL}), borrowed in \cite{Ledoux}\footnote{In \cite{BL}, the authors assume that the function $f$ is   bounded. 
{However, when} the concentration function $\frac{\delta^2}{4C k}\wedge\frac{\delta}{2\sqrt{C k}}$ {goes} to $\infty$ when $\delta\to\infty$, {the boundedness assumption can be removed (see Proposition 1.7 in \cite{Ledoux} for details)}.
}. 



\begin{prop}\label{prop:BK}
Assume $\UPI$, then for any differentiable k-Lipschitz real function $f$:
$$
\forall \theta \in \theta \quad \forall n \in \mathbb{N}^* \qquad 
\mathbb P_{\theta}\left(\left| \frac 1n\sum_{i=1}^nf(\xi_i)-\pi_\theta(f)\right| \geq\delta\right)\leq 2e^{-n\frac{\delta^2}{4k^2 \CPU}\wedge\frac{\delta}{2 k \sqrt{ \CPU}}}.$$
\end{prop}

 
 We will apply this result for $f=\Psi$ involved in  $\IWC$ and with $f=\nabla U_{\theta}$. In particular, using Proposition \ref{prop:BK}, we obtain the following result (see the proof in appendix A).


\begin{cor}\label{lem:BK}

 Let  $\IWC$ holds and denote by $\Psi$ the corresponding 1-Lipschitz function. Then,
\begin{itemize}\item[$i)$]
$$\forall \theta \in \mathbb{R}^{{d}} \qquad \mathbb P_{\theta}\left(\left|\frac 1n\sum_{i=1}^n\Psi({\xi}_i)-\mathbb E_\theta \Psi(\xi_1)\right|\geq\delta\right)\leq 2 e^{-n\frac{\delta^2}{4 \CPU}\wedge\frac{\delta}{2\sqrt{\CPU}}}.$$

\item[$ii)$]
$$
\forall \theta \in \mathbb{R}^{{d}} \qquad 
\mathbb P_{\theta}\left(\left|\frac 1n\sum_{i=1}^n\nabla_{\theta} U(\xi_i,\theta)\right|\geq\delta\right)\leq 2 d e^{-n\frac{\delta^2}{4 L^2 \CPU  d}\wedge\frac{\delta}{2 L \sqrt{\CPU  d}}}.$$
\end{itemize}
\end{cor}

Corollary \ref{lem:BK} will be an essential ingredient for the construction of some efficient statistical tests in the family of probability distributions $(\Pth)_{\theta \in \R^d}$. This key corollary is used in Section \ref{sec:tests}.
\subsection{Consistency rate of the posterior mean}\label{sec:consistency}
To study the behavior of $(\widetilde{\theta}_n)_{n \ge 0}$
 introduced in Equation \eqref{eq:posterior_mean},
 we adopt the presentation of \cite{CasvdV} and in particular the link between the posterior mean and the posterior distribution. As noticed in  \cite{CasvdV}, there is an important need to upper bound the tail of the posterior distribution (far from $\theta^{\star}$). 
  To this end, for a non-negative sequence $(\epsi_n)_{n\geq 1}$ fixed later on, we introduce the separation radius:
\begin{equation}\label{eq:ran}
r_{n}=\epsi_n+r,\end{equation}
where $r$ will vary from $0$ to $+ \infty$. 
\subsubsection{Statistical tests} \label{sec:tests}

Statistical tests have a long standing history in Bayesian literature (see \textit{e.g.} \cite{LeCam,GhosalGhoshvdVaart}) to obtain consistency results as well as rates of convergence of Bayes procedures.
We introduce an appropriate family of tests $(\phi_n^r)_{n\geq 1}$ parametrized by $r>0$ (see Equation \eqref{eq:ran}) and
 defined for $n\in\mathbb N^\star$ by:
 \begin{equation}\label{def:tests}
\phi_n^r\left(\mathbf{\xi^{n}}\right) = \mathbf{1}_{\left| \frac{1}{n} \sum_{i=1}^n \Psi(\xi_i) - \pi_{\theta^\star}(\Psi) \right| \geq \frac{c(\ran)}{2}}. 
\end{equation}
It is expected that $\phi_n^r$ is equal to $0$ with an overwhelming probability under the null hypothesis $\mathbb{P}_{\theta^{\star}}$ whereas $\phi_n^r$ it is equal to $1$ w.o.p.  under $\mathbb{P}_{\theta}$ when $|\theta-\theta^{\star}|$ is large enough thanks to $\IWC$.
We prove the following estimations in appendix A of the first and second type error of  $(\phi_n^r)_{n\geq1}$.

\begin{prop}\label{prop:test} The sequence of tests $(\phi_n^r)_{n\geq1}$ satisfies
\begin{itemize}
\item[$i)$] 
$
\mathbb{P}_{\theta^{\star}} \left( \phi_n^r\left(\mathbf{\xi^{n}}\right)
=1\right) \leq 2  e^{-n \frac{c(\ran)^2}{ 16 \CPU} \wedge \frac{c(\ran)}{4 \sqrt{\CPU}}},
$
\item[$ii)$] 
$
 \sup_{\theta \, : \, |\theta-\theta^{\star}| \geq \ran} \mathbb{P}_{\theta} \left(\phi_n^r\left(\mathbf{\xi^{n}}\right)
=0 \right) \leq 
 2   e^{-n \frac{c(\ran)^2}{16 \CPU} \wedge \frac{c(\ran)}{4 \sqrt{\CPU}}}.
$
\end{itemize}

\end{prop}

The next result is a technical estimation related to the denominator (normalizing constant) involved in the posterior distribution distribution. A simple application of Corollary \ref{lem:BK} yields the next result.

\begin{lem}\label{lem:petite_boule}
For any $t>0$, define the sets
$\mathcal A_n^t =\left\{\left| \frac 1n\sum_{i=1}^n\nabla_{\theta} U(\xi_i,\theta^\star)\right| \leq t\right\},$ then:
 $$
\forall n \in \mathbb{N}^\star \quad \forall t>0 \qquad  
 \mathbb{P}_{\theta^{\star}}(\{\mathcal{A}_n^{t}\}^{c}) \leq 2 d e^{-n \left( \frac{t^2}{4 L^2 \CPU d} \wedge \frac{t}{2 L \sqrt{\CPU {d}  }}\right)}.$$
\end{lem}


\subsubsection{Proof of the posterior mean consistency \label{sec:proof-posterior}}

\begin{proof}[Proof of Theorem \ref{theo:consistency}] 
Our proof adopts the strategy  of \cite{CasvdV}.

\noindent \underline{Step 1: Decomposition of the quadratic risk.}
We remark that  for all $n\in\mathbb N^\star$:
\begin{align}
\mathbb{E}_{\theta^{\star}}\left[\left|\widetilde{\theta}_n-\theta^{\star}\right|^p\right] & =  \mathbb{E}_{\theta^{\star}} \left[ \left| \int_{\mathbb R^d} (\theta - \theta^{\star}) \text{d}\pi_n(\theta) \right|^p \right] \nonumber\\
& \leq   \mathbb{E}_{\theta^{\star}} \left[  \int_{\mathbb R^d}  \left|\theta - \theta^{\star}  \right|^p \text{d}\pi_n(\theta)\right] \nonumber\\
& =   p \mathbb{E}_{\theta^{\star}} \left[  \int_{0}^\infty t^{p-1} \pi_n(|\theta-\theta^{\star}| \geq t) \text{d}t \right]\nonumber\\
& = p \mathbb{E}_{\theta^{\star}} \left[  \int_{0}^{\epsi_n} t^{p-1} \pi_n(|\theta-\theta^{\star}| \geq t) \text{d}t+\int_{\epsi_n}^\infty t^{p-1} \pi_n(|\theta-\theta^{\star}| \geq t) \text{d}t \right]\nonumber\\
&\leq   \epsi_n^p + p \int_{0}^{+ \infty}  \ran^{p-1}  \mathbb{E}_{\theta^{\star}} \left[ \pi_n(|\theta-\theta^{\star}| \geq \ran)\right] \text{d}r, \label{eq:bornequad}
\end{align}
where we used the Jensen inequality in the second line, an integration by part in the third line, a direct integration $\int_0^{\epsi_n} p t^{p-1} \text{d} t =  \epsi_n^p$ in the last line associated with the Fubini relationship.

\noindent
\underline{Step 2: Use of the tests $(\phi_n^r)_{n\geq1}$.}
We now use the tests $(\phi_n^r)_{n\geq1}$ and the sets $(\mathcal A_n^t)$, we can write:
\begin{align}
\mathbb{E}_{\theta^{\star}} \left[ \pi_n(|\theta-\theta^{\star}| \geq \ran  )
\right] = & \,\mathbb{E}_{\theta^{\star}} \left[ \phi_n^r(\mathbf{\xi^{n}}) \pi_n(|\theta-\theta^{\star}| \geq \ran  )
\right]\nonumber  \\
&+ \mathbb{E}_{\theta^{\star}} \left[ (1- \phi_n^r(\mathbf{\xi^{n}})) \pi_n(|\theta-\theta^{\star}| \geq \ran  ) \mathbf{1}_{\mathcal{A}_n^t}
\right]\nonumber \\ 
& + \mathbb{E}_{\theta^{\star}} \left[ (1- \phi_n^r(\mathbf{\xi^{n}})) \pi_n(|\theta-\theta^{\star}| \geq \ran  ) \mathbf{1}_{\{\mathcal{A}_n^t\}^c}
\right] .\label{eq:introtest}
\end{align}
From this expression we can deduce the following inequality:
\begin{align*}
\mathbb{E}_{\theta^{\star}} \left[ \pi_n(|\theta-\theta^{\star}| \geq \ran  )
\right]\leq&\, \mathbb{E}_{\theta^{\star}} \left[ \phi_n^r(\mathbf{\xi^{n}})\right] +  \mathbb{E}_{\theta^{\star}} \left[ (1- \phi_n^r(\mathbf{\xi^{n}})) \pi_n(|\theta-\theta^{\star}| \geq \ran  ) \mathbf{1}_{\mathcal{A}_n^{t}}
\right] \\& + \mathbb{P}_{\theta^{\star}}(\{\mathcal{A}_n^{t}\}^c).
\end{align*}

\noindent
$\bullet$
Study of $\mathbb{E}_{\theta^{\star}} \left[ (1- \phi_n^r(\mathbf{\xi^{n}})) \pi_n(|\theta-\theta^{\star}| \geq \ran  ) \mathbf{1}_{\mathcal{A}_n^t}\right]$. We write that:
\begin{align}
\displaystyle\pi_n(|\theta-\theta^{\star}| \geq \ran  )  &= \int_{\theta \, : \, |\theta-\theta^{\star}| \geq \ran } \text{d}\pi_n(\theta)  = \frac{\displaystyle\int_{\theta \, : \, |\theta-\theta^{\star}| \geq \ran }  \frac{e^{-W_n(\mathbf{\xi^n},\theta)}}{e^{-W_n(\mathbf{\xi^n},\theta^{\star})}}\text{d}\pi_0(\theta)}{\displaystyle\int_{\R^d }  \frac{e^{-W_n(\mathbf{\xi^n},\theta)}}{e^{-W_n(\mathbf{\xi^n},\theta^{\star})}}\text{d}\pi_0(\theta)}.\label{eq:numerateur_denominateur}
\end{align}
At this stage, we  control the denominator and the numerator separately.  Let us denote by $Z_t = \pi_0(B(\theta^{\star},t))$ the prior mass of the Euclidean ball centered at $\theta^{\star}$ and of radius $t$. We have \begin{align*}
\log \left( \int_{\R^d} \frac{e^{-W_n(\mathbf{\xi^n},\theta)}}{e^{-W_n(\mathbf{\xi^n},\theta^{\star})}}d\pi_0(\theta) \right) & \geq  
\log \left( \int_{B(\theta^{\star},t)} \frac{e^{-W_n(\mathbf{\xi^n},\theta)}}{e^{-W_n(\mathbf{\xi^n},\theta^{\star})}}\text{d} \pi_0(\theta) \right)
\\ 
& \geq \log \left( \int_{B(\theta^{\star},t)} \frac{e^{-W_n(\mathbf{\xi^n},\theta)}}{e^{-W_n(\mathbf{\xi^n},\theta^{\star})}} \frac{\text{d} \pi_0(\theta)}{Z_t} \right) + \log Z_t \\
& \geq 
\int_{B(\theta^{\star},t)}  \sum_{i=1}^n [U(\xi_i,\theta^{\star})-U(\xi_i,\theta)] \frac{\text{d} \pi_0(\theta)}{Z_t}  + \log Z_t,
\end{align*}
where we used the Jensen inequality in the last line with the concave function $\log$ and the normalized measure $\text{d} \pi_0 Z_t^{-1}$ over $B(\theta^{\star},t)$.
Using that 
$\nabla_\theta U(\xi,.)$ is L-Lipschitz we get
$$ 
\forall x \in \mathbb{R}^d \qquad 
U(\xi,\theta_1)-U(\xi,\theta_2) \leq \langle \theta_1-\theta_2, \nabla_{\theta} U(\xi,\theta_2)\rangle + \frac{L}{2} \|\theta_1-\theta_2\|^2,$$
for all $(\theta_1,\theta_2)\in \mathbb{R}^d$, which implies that:
$$
\forall i \in \{1,\ldots,n\} \quad \forall \theta \in \R^q \qquad 
\left|U(\xi_i,\theta)-U(\xi_i,\theta^{\star}) - \langle \theta-\theta^{\star},\nabla U(\xi_i,\theta^{\star})\rangle\right| \leq \frac{L}{2} |\theta-\theta^{\star}|^2.
$$
 Using a sum over $i$ and the triangle inequality, we then deduce that:
$$
\forall \theta \in \R^d \quad \sum_{i=1}^n U(\xi_i,\theta^\star)-U(\xi_i,\theta) \geq - \left| \langle \theta-\theta^{\star},\sum_{i=1}^n \nabla_{\theta} U(\xi_i,\theta^{\star} )\rangle \right| - n \frac{L}{2} |\theta-\theta^{\star}|^2.
$$
The Cauchy-Schwarz inequality yields:
$$
\left| \langle \theta-\theta^{\star},\sum_{i=1}^n \nabla_\theta U(\xi_i,\theta^{\star}) \rangle \right| \leq |\theta-\theta^{\star} | \, \left| \sum_{i=1}^n \nabla_\theta U(\xi_i,\theta^{\star})\right|.
$$
An integration over $B(\theta^{\star},t)$ with the normalized measure $\pi_0 Z_t^{-1}$  leads to:
\begin{align*}
\log \left( \int_{\R^d} \frac{e^{-W_n(\mathbf{\xi^n},\theta)}}{e^{-W_n(\mathbf{\xi^n},\theta^{\star})}}\text{d}\pi_0(\theta) \right) & \geq - n \frac{L}{2}  t^2 \frac{ \pi_0(B(\theta^{\star},t)) }{Z_t} - t \left\| \sum_{i=1}^n \nabla_{\theta} U(\xi_i,\theta^{\star})\right\| + \log(Z_t) \nonumber\\
& \geq - n \frac{L}{2}  t^2  - t \left\| \sum_{i=1}^n \nabla_\theta U(\xi_i,\theta^{\star})\right\| +  \log (Z_t).
\end{align*}
To lower bound the denominator, we use the set $\mathcal A_n^t$ and we have
\begin{align}
\log \left( \int_{\R^d} \frac{e^{-W_n(\mathbf{\xi^n},\theta)}}{e^{-W_n(\mathbf{\xi^n},\theta^{\star})}}\text{d}\pi_0(\theta) \right)\mathbf 1_{\mathcal A_n^t} & \geq\left( - n \frac{L}{2}  t^2  - t \left\| \sum_{i=1}^n \nabla_\theta U(\xi_i,\theta^{\star})\right\| + \log(Z_t)\right)\mathbf 1_{\mathcal A_n^t} \nonumber\\
& \geq\left( - nt^2 \left(\frac{L}{2}+1\right)   +  \log (Z_t)\right)\mathbf 1_{\mathcal A_n^t}.\label{eq:minoration_intermediaire}
\end{align}
Using \eqref{eq:minoration_intermediaire} with \eqref{eq:numerateur_denominateur} and the Jensen Inequality, we have
\begin{align*}
\lefteqn{\mathbb{E}_{\theta^{\star}} \left( (1-\phi_n^r(\mathbf{\xi^{n}})) \pi_n(|\theta-\theta^{\star}| \geq \ran) \mathbf{1}_{\mathcal{A}_n^{t}} \right)} \\
& \leq \frac{ \mathbb{E}_{\theta^{\star}} \left[ (1-\phi_n^r(\mathbf{\xi^{n}})) \int_{\theta: |\theta-\theta^{\star}| \geq \ran} \frac{e^{-W_n(\mathbf{\xi^n},\theta)}}{e^{-W_n(\mathbf{\xi^n},\theta^{\star})}} \text{d}\pi_0(\theta) \right]}{Z_t e^{- n \left(\frac{L}{2}+1\right) t^2 }} \\
& \leq \int_{\theta: |\theta-\theta^{\star}| \geq \ran}
\mathbb{E}_{\theta^{\star}} \left[ (1-\phi_n^r(\xi^n))\frac{e^{-W_n(\mathbf{\xi^n},\theta)}}{e^{-W_n(\mathbf{\xi^n},\theta^{\star})}} \right]  \text{d}\pi_0(\theta) e^{n t^2  \left(  \frac{L}{2}  + 1 \right) }(Z_t )^{-1} \\
& \leq(Z_t )^{-1} e^{n t^2 \left(  \frac{L}{2}  + 1 \right)}\sup_{\{\theta \, : |\theta-\theta^{\star}| \geq \ran\}}\mathbb{E}_{\theta} \left[ 1-\phi_n^r(\mathbf{\xi^{n}})  \right] \\
& \leq 2   e^{n t^2 \left(  \frac{L}{2}  + 1 \right)-\log \pi_0(B(\theta^{\star},t))} e^{-n \frac{c(\ran)^2}{16 \CPU} \wedge \frac{c(\ran)}{4 \sqrt{\CPU}}},
\end{align*}
where in the penultimate line we used a change of measure from $\mathbb{P}_{\theta^{\star}}$ to $\mathbb{P}_{\theta}$.

\noindent
$\bullet$
Study of $\mathbb{E}_{\theta^{\star}} \left[ \phi_n^r(\mathbf{\xi^{n}})\right]$.
Using the  first type error given by $i)$ of Proposition \ref{prop:test}, we have:
$$
\mathbb{E}_{\theta^{\star}}[\phi_n^r(\mathbf{\xi^{n}})] 
 \leq 2  e^{-n \frac{c(\ran)^2}{16 \CPU} \wedge \frac{c(\ran)}{4 \sqrt{\CPU}}}.
$$

\noindent
$\bullet$
Study of $\mathbb{E}_{\theta^{\star}} \left[ (1- \phi_n^r(\mathbf{\xi^{n}})) \pi_n(|\theta-\theta^{\star}| \geq \ran  ) \mathbf{1}_{\{\mathcal{A}_n^t\}^c}\right]$. We introduce the dependency in $r$ and $n$ to upper bound $\mathbb{P}_{\theta^{\star}}(\{\mathcal{A}_n^{t}\}^c)$ with $t $ and apply Lemma \ref{lem:petite_boule}. We have:
 $$\mathbb{P}_{\theta^{\star}}(\{\mathcal{A}_n^{t}\}^{c}) \leq 2 d e^{-n \left( \frac{t^2}{4 L^2 \CPU d} \wedge \frac{t}{2 L \sqrt{\CPU  d}}\right)}.$$
We then obtain that:
\begin{align}
\lefteqn{
\mathbb{E}_{\theta^{\star}} \left[ \pi_n(|\theta-\theta^{\star}| \geq \ran  )
\right] } \nonumber \\
& \leq 4   e^{n t^2 \left(  \frac{L}{2}  + 1 \right)-\log \pi_0(B(\theta^{\star},t))} e^{-n \frac{c(\ran)^2}{16 \CPU}\wedge \frac{c(\ran)}{4 \sqrt{\CPU}}}+2 d e^{-n \left( \frac{t^2}{4 L^2 \CPU d} \wedge \frac{t}{2 L \sqrt{\CPU  d}}\right)}. \label{eq:intermediaire}
\end{align}
\noindent
\underline{Step 3: Small ball calibration and prior mass}
 We   now adjust the different parameters in order to obtain the best   rate for $(\widetilde{\theta}_n)_{n\geq0}$. 
  We  choose $t$ that depends on $r$ and $n$, \textit{i.e.} $t=t_{r,n}$ according to:
$$
t_{r,n} =  \frac{c(\ran) \wedge \sqrt{c(\ran)}}{A},
$$
with $A$ chosen sufficiently large such that $A = \sqrt{32 (L+1) \CPU \vee \sqrt{\CPU}}$, so that:
$$
n t_{r,n}^2 \left(  \frac{L}{2}  + 1 \right) - n \frac{c(\ran)^2}{16 \CPU}\wedge \frac{c(\ran)}{4 \sqrt{\CPU}} \leq 
- n \frac{c(\ran)^2}{32 \CPU}\wedge \frac{c(\ran)}{8 \sqrt{\CPU}}.
$$
The previous inequality may be verified by considering the value of $c(\ran)$ and its position in comparison to $1$ and to $2 \sqrt{\CPU}$.
In the meantime, we get that:

$$
-\log \pi_0(B(\theta^{\star},t_{r,n})) \leq - \log \pi_0\left( B(\theta^{\star},A^{-1}[c(\epsi_n) \wedge \sqrt{c(\epsi_n)}])\right).
$$
We introduce $\epsi_n$ as:
$$
\epsi_n =  \left(L^2 \CPU d \frac{\log n}{n}\right)^{1/2 \ic},
$$
and consider $n$ large enough such that $ \epsi_n \leq 1$ and $b_1 \varepsilon_n^{\ic} < 1$ (where $b_1$ is given in Assumption $\IWC$), then:
 $$
-\log \pi_0(B(\theta^{\star},t_{r,n})) \leq - \log \pi_0\left( B(\theta^{\star},A^{-1}c(\epsi_n)) \right).
$$
Since $\pi_0 = e^{-V_0}$ with $V_0$ a $\mathcal{C}_1^1$ function, we then deduce that: 
$$
\forall \delta>0, \quad 
\forall \theta \in B(\theta^{\star},\delta): \qquad |V_0(\theta)-V_0(\theta^{\star})| \leq \delta \|\nabla V(\theta^{\star})\| + \frac{1}{2} \delta^2,
$$
which implies:
\begin{align*}
\lefteqn{- \log \pi_0(B(\theta^{\star},t_{r,n}))} \\
 &\leq - \log \left( \int_{B(\theta^{\star},A^{-1}c(\epsi_n))} e^{V_0(\theta^{\star}) - A^{-1}c(\epsi_n) \|\nabla V(\theta^{\star})\| - A^{-2}c(\epsi_n)^2/2} \text{d} \lambda_d(\theta) \right) \\
& = -V_0(\theta^{\star}) + A^{-1}c(\epsi_n) \|\nabla V(\theta^{\star})\| + \frac{A^{-2}c(\epsi_n)^2}{2} + d \log (A c(\epsi_n)^{-1})- \log \lambda_d (B(0,1))\\
& \leq -V_0(\theta^{\star}) + A^{-1} c(a \epsi_2) \|\nabla V(\theta^{\star})\| + \frac{A^{-2} c(a \epsi_2)^2}{2} + d \log (c(\epsi_n)^{-1}).
\end{align*} 
Using the behaviour of $c$ near $0$ (see $\IWC$), a constant $C_{\theta^{\star}}$ exists such that:
$$
- \log \pi_0(B(\theta^{\star},t_{r,n})) \leq \log(C_{\theta^{\star}}) + d \ic \log(\epsi_n^{-1}).
$$
We then obtain  with $K=32 \vee 4 A^2$ that a universal $q$ exists such that  \eqref{eq:intermediaire} may be upper bounded as:
\begin{align*}
\mathbb{E}_{\theta^{\star}} \left[ \pi_n(|\theta-\theta^{\star}| \geq \ran  )
\right] & \le q d e^{-\frac{n}{K} \left[ \frac{c(r_{n})^2}{L^2 \CPU d} \wedge \frac{c(r_{n})}{L \sqrt{\CPU d}} \right]+ d \ic \log(\epsi_n^{-1})}.
\end{align*}
Finally, Equation \eqref{eq:ic} yields:
\begin{align}\label{eq:deviation}
\mathbb{E}_{\theta^{\star}} \left[ \pi_n(|\theta-\theta^{\star}| \geq \ran)
\right] & \le q d e^{-\frac{n}{K} b_1^2 \frac{\{r_{n}\}^{2 \ic}}{L^2 \CPU d}} \mathbf{1}_{r_{n} \leq 1}
+q d e^{-\frac{n}{K} b_2 \frac{ \log(r_{n})+1}{L \sqrt{\CPU d}}} \mathbf{1}_{r_{n} \geq 1}.
\end{align}


\noindent
\underline{Step 4: Convergence rate}
 We use \eqref{eq:bornequad} and \eqref{eq:deviation} and obtain that:

\begin{eqnarray*}
\mathbb{E}_{\theta^{\star}}[\|\widetilde{\theta}_n-\theta^{\star}\|^p] &\lesssim_{uc} & \epsi_n^p   + d \int_{0}^{+ \infty} r_{n}^{p-1} \left(
 e^{-\frac{n}{K} b_1^2 \frac{\{r_{n}\}^{2 \ic}}{L^2 \CPU d}} \mathbf{1}_{r_{n} \leq 1} +  e^{-\frac{n}{K}  b_2 \frac{ \log(r_{n})+1}{L \sqrt{\CPU d}}} \mathbf{1}_{r_{n} \geq 1}\right)\text{d}r\\
& \luc &  \epsi_n^p + d \left[ \int_{\epsi_n}^{1} r^{p-1} 
  e^{- \frac{n}{K}  b_1^2 \frac{r^{2\ic}}{L^2 \CPU d}} \text{d}r+
\int_{1}^{+\infty} r^{p-1} e^{- \frac{n}{K}  b_2 \frac{\log(r)+1}{L \sqrt{\CPU d}}} \text{d}r\right].
\end{eqnarray*}
Using the value of $\epsi_n$ we introduced in Step 2,
we then observe that:
$$
\int_{\epsi_n}^{+ \infty} r^{p-1}  e^{- \frac{n}{K}   b_1^2 \frac{r^{2\ic}}{L^2 \CPU d}} \text{d}r \leq  \left( \frac{K L^2 \CPU d}{ b_1^2 n}\right)^{p/2\ic}  (2 \ic)^{-1}\int_{b_1^2 K \log n}^{+ \infty} v^{p/2\ic-1}e^{-v} \text{d}v 
=  \cal{O}_{uc}(\epsi_n^p).
$$
The second integral may be made exponentially small (in terms of $n$). We then observe that the leading contribution of the $L^p$ loss is then brought by $\epsi_n^p$.
\end{proof}


\subsection{ Minimax lower bound (proof of Theorem \ref{theo:lower_bound})}\label{sec:lower_bound}
 
Below, we establish a lower bound of estimation that matches with the Bayesian consistency rate we derive in Theorem \ref{theo:consistency} in terms of $n$ and $d$.

\begin{proof}[Proof of Theorem \ref{theo:lower_bound}]
 For this purpose, we introduce $\varphi_{\ic}$ defined by:
$$
\varphi_{\ic}(x) = sgn(x) |x|^{\ic} \mathbf{1}_{|x| \leq 1} + x   \mathbf{1}_{|x| > 1},
$$
and the family of multivariate Gaussian distributions $\mathbb{P}_{\theta} = \mathcal{N}(\delta \mu_{\ic}(\theta),I_d)$ with  $\delta>0$ defined later on and 
$$
\forall i \in \{1,\ldots,d\} \qquad 
\mu_{\ic}(\theta)_i = \varphi_{\ic}(\theta_i).
$$
Considering all positions of $x$ and $y$, we observe that a constant $c$ exists such that:
$$
\forall (x,y) \in \mathbb{R}^2 \qquad 
|\varphi_{\ic}(x)-\varphi_{\ic}(y)| \ge c |x-y|^{\ic}
$$

As a translation model, it is immediate to verify that:
\begin{align*}
W_1(\mathbb{P}_{\theta_1},\mathbb{P}_{\theta_2})^2& = \delta^2 |\mu_{\ic}(\theta_1)- \mu_{\ic}(\theta_2)|_2^2 \\
& =  \delta^2  \sum_{i=1}^d |\varphi_{\ic}(\theta_1^i)-\varphi_{\ic}(\theta_2^i)|^2\\
& \ge \delta^2  c^2 d \left(\frac{1}{d}\sum_{i=1}^d | \theta_1^i - \theta_2^i |^{2 \ic}\right) \\
& \ge d \delta^2   c^2 \left(\frac{1}{d}\|\theta_1-\theta_2\|^{2}\right)^{ \ic} = c^2 \delta^2   d^{1-\ic} \|\theta_1-\theta_2\|_2^{2 \ic}
\end{align*}
where we applied the Jensen inequality in the previous line with the convex function $r\longmapsto r^{\ic}$ with $\ic \ge 1$.
Therefore, we observe that when $\delta=d^{-(1-\ic)/2}$, the family $\mathbb{P}_{\theta}$ satisfies $\IWC$.
In the meantime, straightforward computations show that $\UPI$ and $\AL$ also hold for this statistical model so that this statistical model belongs to the $\mathcal{F}_{\ic}^L$ class used in the statement of the result.\\

In this statistical model, we then derive a lower bound of estimation with the $L^2$ risk applying the Fano Lemma associated with the Varshamov-Gilbert Lemma. We introduce the Hamming distance $\rho(\omega,\omega')$ defined by: 
$$
\forall (\omega,\omega') \in \{0,1\}^d \times \{0,1\}^d \qquad 
\rho(\omega,\omega') = \sum_{j=1}^d \mathbf{1}_{\omega^j \neq \omega'^j}.
$$
The Varshamov-Gilbert yields the existence of $M=\lfloor e^{d/32}\rfloor$ points in $\{0,1\}^d$ denoted by $(\omega_1,\ldots,\omega_M)$ such that 
$$
\forall j \neq k \qquad \rho(\omega_j,\omega_k) \ge \frac{d}{4}.
$$
With this net over $\{0,1\}^d$, we introduce the net $\theta_1,\ldots,\theta_M$ defined by
$$
\theta_i = \beta \omega_i,
$$
for some $\beta>0$ chosen later on.
 We then verify that:
$$
\forall j \neq k \qquad 
\|\theta_j-\theta_k\|_2 \ge \beta  \frac{\sqrt{d}}{2}.
$$
In the meantime,  the Kullback divergence for Gaussian distributions leads to:
$$
KL(\mathbb{P}_{\theta_i},\mathbb{P}_{\theta_j}) = \frac{n}{2} \delta^2 |\mu_{\ic}(\theta_i) - \mu_{\ic}(\theta_j)|_2^2 \leq \frac{n}{2} d^{\ic} \beta^{2 \ic}
$$
where we used that the maximal value of the Hamming distance is $d$ and $d \delta^2 = d^{\ic}$. We then apply the Fano Lemma and observe that:
$$
\inf_{\hat{\theta}_n} \sup_{\theta} \mathbb{P}_{\theta}\left( |\hat{\theta}_n-\theta|_2 \ge \frac{\beta \sqrt{d}}{8} \right) \ge c >0
$$
as soon as:
$$
\frac{\frac{n}{2} d^{\ic} \beta^{2 \ic} + \log(2)}{d/32} < 1.
$$
We then choose $\beta$ as large as possible, \textit{i.e.} we choose $\beta$ such that:
$$
\beta = b \left(\frac{d^{1-\ic}}{n}\right)^{\frac{1}{2 \ic}} = b \left(\frac{d}{n}\right)^{\frac{1}{2 \ic}} d^{-1/2}.
$$
This entails a minimax lower bound for the $L^2$ rate of the order:
$$
r_{n,d}(\ic) \gtrsim \sqrt{d} \beta \simeq \left(\frac{d}{n}\right)^{\frac{1}{2\ic}}.
$$
 \end{proof}

\section{Discretization of the Langevin procedure - Weakly convex case }\label{sec:disc_c}

The weakly convex (\textit{i.e.} not uniformly strongly convex) case is tackled {with a completely different approach with the help of Assumption $\Hcu$. Actually, in the weakly convex case, a series of properties disappear. For instance, one can not  easily control the pathwise distance between the process and its discretization. The problem is then  significantly harder and we choose here to make use of the inversion of the Poisson equation, which leads to a relatively tractable formulation of  the error between the discretized Cesaro average and the invariant distribution (applied to the identity function). } 
%
%
%
%
%
%
In particular, this ``Poisson equation approach'' is in the continuity of  \cite{lamberton_pages,pages_menozzi}) and has a long-standing history in the study of central limit theorem for Markov chains. 
We refer to \cite{Meyer,Neveu,Revuz,Meyn_Poisson} for seminal contributions on additive functionals of Markov chains.
We first stay at an informal level in this paragraph for the sake of readability.
We sketch the general idea behind the use of this equation with an   Euler scheme.

Again, we first state some general results with a diffusion process $(X_t)_{t \geq 0}$ solution of:
\begin{equation}\label{eq:EDS} 
dX_t=- \nabdob(X_t) dt+ \sqrt{2}  dB_t.\end{equation}

\subsection{How to use the Poisson equation $f-\pi(f)= \cal{L}  g$?}

This approach is based on the inversion of the operator ${\cal L}$ of the diffusion. For a given function $f$, we recall that the solution of the Poisson equation is the function $g$ such that $\pi(g)=0$ and that satisfies:
$$ f-\pi(f)={\cal L} g,$$
where $\pi$ denotes the invariant distribution of the diffusion (see below for background on existence and uniqueness of the solution). 
We consider $g$ the solution of the Poisson equation.

$\bullet$
For such a solution, a first important ingredient is based on the following remark: if $(X_t)_{t\ge0}$ is a Markov process with generator ${\cal L}$ and $g$ belongs to the (extended) domain, then the Ito formula yields:
$$
g(X_t)=  g(X_0) + \int_{0}^t {\cal L} g(X_s) \text{d}s + \mathcal{M}_t^g,
$$
so that
\begin{equation}\label{eq:extmart}
\int_0^t f(X_s)-\pi(f) \text{d}s= \int_0^t {\cal L} g(X_s) \text{d}s- (g(X_t)- g(X_0))
\end{equation}
is a local martingale (and certainly a true martingale under appropriate conditions). Thus, the control of the distance between 
$(\frac{1}{t}\int_0^t f(\xi_s))_{t\ge0}$ and $\pi(f)$ can be tackled from a martingale point of view. 

$\bullet$
The second main interest of this approach is the possibility to specify that our estimator involves $f=\id$, which is an important ingredient of the approximation of $\pi(f)$. Such a precision is  untractable when we handle distances between probability distributions.\\

{We first state that the Poisson equation is well-posed in our setting and recall a classical formulation of this solution. {The proof is postponed to the Appendix B.} Note that this result is only stated under the assumptions of our main theorems but may be certainly extended to a more general setting (see \cite[Corollary 3.2]{CCG} for a more general result).}
\begin{prop}[Poisson equation]\label{theo:poisson} 
{Assume $\Hcu$ and suppose that $\dob$ is ${\cal C}^3$ with bounded third derivatives. 
Then, Equation \eqref{eq:EDS}  admits a unique invariant distribution and for any ${\cal C}^2$-function $f$ with bounded derivatives,  the problem  $\LL g = f-\pi(f)$ is well-posed on the set of ${\cal C}^2$-functions such that $\pi(g)=0$ and the unique solution is given by:
$$
g(x) = \int_{0}^{+ \infty} [\pi(f)-P_s f(x)] {\rm d}s.
$$}
\end{prop}

Note that in what follows, we will solve this equation $d$ times for a multivariate function $f=(f^1,\ldots,f^{{d'}})$ ({and will mainly consider the case  $f=\id$ for applications to Bayesian estimation}).
\subsection{Poisson equation and discretization}\label{sec:poissonbisbis}

In the  discretized case, the aim is then to mimick the martingale property of Equation \eqref{eq:extmart} but some additional error terms appear with the discretization approximation. Such ideas have been strongly studied in \cite{lamberton_pages,pages_menozzi} but since the solution to the Poisson equation is not explicit (in general), 
the previous works have usually made \textit{ad hoc} assumptions on the function $g$ and its derivatives. For our purpose, we identify the key properties satisfied by the solution when $f=\id$ in terms of the dimensional dependence.

We observe that $(\Y_{t_k})_{k \ge 1}$, 
computed through the recursion
$$
\Y_{t_k} = - \gamma \nabla W(\Y_{t_{k-1}}) + \sqrt{2 \gamma}  \Ur_k,
$$
where $(\Ur_k)_{k \ge 0}$ is an i.i.d. sequence of standard $d$-dimensional Gaussian random variables, 
is a sequence of discrete time observations of the continuous time process $(\Y_t)_{t \ge 0}$ defined by:
\begin{equation}\label{barxt}
\forall t \in [t_k,t_{k+1}] \qquad 
d \Y_{t} = -   \nabla W(\bX_{t_k}) dt + \sqrt{2}  dB_t,
\end{equation}
with $\sqrt{\gamma} \Ur_k = B_{t_k}-B_{t_{k-1}}$.

Considering a \textit{multivariate} function $g=(g^1,\ldots, g^{{d'}}):\R^d\rightarrow\R^{{d'}}$, 
we denote by $Dg=[\nabla g^1{;} \ldots{;} \nabla g^{{d'}}]$ {its Jacobian matrix} {which maps from $\ER^d$ to the space of $d'\times d$-matrices.} Similarly, $\Delta g$ refers to the vector built with $(\Delta g^1, \ldots, \Delta g^{{d'}})$.
For  $s >0$, we define $\un{s}$ the largest grid point in $(t_k)_{k \ge 0}$ below $s$:
\begin{equation}\label{def:s}
\un{s} := \sup \{t_k \, : t_k \le s \}.
\end{equation}
We then observe that:
$$g(\bX_t)=g(x)+\int_0^t \bar{\cal L} g(\bX_s,\bX_{\un{s}}) \text{d}s+ {\cal M}_t^{(g)},$$
where  $\un{s}$ is  defined in \eqref{def:s},  $\bar{\LL}$ is given by:
\begin{equation}\label{def:LL_b}
 \bar{\LL} g(x,\underline{x})=-  D g(x) \nabdob(\un{x}) + \Delta g(x),
 \end{equation}
and ${\cal M}^{(g)}$ is {the  $\R^d$-valued local martingale} defined by:
\begin{equation}\label{eq:martingale_g}
 {\cal M}_t^{(g)}=  \sqrt{2} \int_0^t   D g(\bX_s) \text{d}B_s.
 \end{equation} 
Similarly, the definition of $\LL$ shall be extended to multivariate functions by
\begin{equation}\label{def:LL}
\LL g(x)={(\LL g^1(x),\ldots\LL g^{{d'}}(x))} = - D g(x) \nabdob(x) + \Delta g(x).
\end{equation}



\noindent We first state some useful technical results for the proof of Theorem \ref{cor:thmPoisson} {whose proofs 
 ar postponed to  Appendix B}

\begin{lem}\label{lem:gradWmaindoc} Assume $\Hcu$. Then, $\forall x\in\ER^d$,
$$ \frac{\deux}{1-r}\left( \dob^{1-r}(x)-\dob^{1-r}(x^\star)\right)\le   |\nabla \dob(x)|^2\le \frac{\troiss}{1-r}\left( \dob^{1-q}(x)-\dob^{1-q}(x^\star)\right).$$
Furthermore, $\forall x\in\ER^d$,
\begin{equation}\label{eq:asympdob}
\dob^{1+r}(x)-\dob^{1+r}(x^\star)\ge (1+r)\deux|x-x^\star|^2\;\textnormal{and}\;\dob^{1+q}(x)-\dob^{1+q}(x^\star)\le \frac{\troiss(1+q)}{1-q}|x-x^\star|^2.
\end{equation}
\end{lem}


{In this second technical result, we obtain some crucial bounds related to the solution of the Poisson equation under Assumption $\Hcu$.}
\begin{prop}\label{prop:fundamentalbounds2} Assume $\Hcu$ with $\cun>0$, $\deux>0$ and $r\in[0,1]$. {Let $f:\ER^d\rightarrow\ER^{d'}$ be a Lispchitz ${\cal C}^2$-function with $\flip\le 1$ and $\dflip\le 1$}.    Then, $g{:\ER^d\rightarrow\ER^{d'}}$ is  a ${\cal C}^2$-function and for every $\mte\in(0,1)$, a constant  $c_\mte$ exists (which only depends on $\mte$ {and not on the other parameters}), such that for any $x$:
\begin{enumerate}

\item[$i-a)$]

{$\nsp{Dg(x)}{\le} c_\mte \deux^{-1-\mte} \left(\dob^{r(1+\mte)}(x)+\Upsilon^{r(1+\mte)}\right).$}

\item[$i-b)$]
{
If $\dob(x_0)\lesssim_{uc} \dw$,
$\sup_{t\ge0}\E_{x_0}[\nsp{Dg(\bX_t)}^2]{\le} c_\mte   \deux^{-2(1+\mte)} \Upsilon^{2r(1+\mte)}.$}
{
\item[$ii-a)$] If $\dob(x^\star)\lesssim_{uc}\Upsilon$,
$| g(x)-g(x^\star)|^2{\le}
c_\mte \deux^{-3-2\mte} \left({\dob}^{(1+3r)(1+\mte)}(x)+ \Upsilon^{(1+3r)(1+\mte)}\right).
 $
\item[$ii-b)$] If $\dob(x_0)\lesssim_{uc} \Upsilon$, then
$\sup_{t\ge0} \E_{x_0}[| g(\bX_t)-g(x_0)|^2]{\le} c_\mte \deux^{-3-2\mte}  \Upsilon^{(1+3r)(1+\mte)}.$}
{
\item[$iii-a)$] Assume that $\hesdob$ is $\tilde{L}$-Lipschitz for the norm $\nsp{.}$. Then, for any $x, y\in\ER^d$ with $x\neq y$,
$$\frac{\nsp{Dg(y)-Dg(x)}}{|y-x|}{\le}  c_\mte \deux^{-2(1+\mte)}\tilde{L}\left(\dob^{2r(1+\mte)}(x)+\dob^{2r(1+\mte)}(y)+\Upsilon^{2r(1+\mte)}\right).$$
\item[$iii-b)$] Set $\tX_{{t}}=\bX_{\un{t}}-(t-\un{t}) \nabdob(\bX_{\un{t}})$. If $\dob(x_0)\lesssim_{uc} \dw$,
$$\E_{x_0}\left[|(Dg(\bX_{{t}})-Dg(\tX_{{t}}))(\nabla W(\bX_{\un{t}}+\Delta_{\un{t}t})-\nabla W(\bX_{\un{t}}))|^2\right]{\le} \mathfrak{c}_\mte (L \tilde{L} \deux^{-2(1+\mte)}(t-\un{t}))^2 d^2 \Upsilon^{4r(1+\mte)}.$$
}
\end{enumerate}
\end{prop}

\begin{proof}[Proof of Theorem  \ref{cor:thmPoisson}]

The plan of the proof is the same for $i)$ and $ii)$ and is decomposed into three steps. Steps 1 and 2  are common whereas the last one is treated separately.

\noindent \textit{\underline{Step 1: Decomposition of $\pi_N(f)-\pi(f)$.}}
We observe that:
\begin{align*}
\pi_N(f)-\pi(f) & := \frac{1}{t_N} \sum_{k=1}^N \gamma_k f(\bX_{t_{k-1}}) -\pi(f)
 = \frac{1}{t_N} \sum_{k=1}^N  \int_{t_{k-1}}^{t_k} f(\bX_{\un{s}}) \text{d}s -\pi(f)\\
 &= \frac{1}{t_N} \int_{0}^{t_N} [f(\bX_{\un{s}}) - \pi(f)] \text{d}s\nonumber\\
& = \frac{1}{t_N} \int_{0}^{t_N} [f(\bX_{s}) - \pi(f)] \text{d}s  + \frac{1}{t_N} \int_{0}^{t_N} [f(\bX_{\un{s}}) - f(\bX_{s})] \text{d}s
\end{align*}
Now, we may use the Poisson equation $f-\pi(f)={\cal L} g$ and deduce that:
\begin{align*}
&\pi_N(f)-\pi(f)= \frac{1}{t_N} \int_{0}^{t_N} {\cal L}  g(\bX_{s}) \text{d}s+ \frac{1}{t_N} \int_{0}^{t_N} [{f}(\bX_{\un{s}})- {f}(\bX_{s})] \text{d}s\\
& = 
\frac{1}{t_N} \int_{0}^{t_N} \bar{{\cal L}}  g(\bX_s,\bX_{\un{s}}) \text{d}s+ \frac{1}{t_N} \int_{0}^{t_N} [{\cal L} g(\bX_{s})-\bar{{\cal L}}  g(\bX_s,\bX_{\un{s}})] \text{d}s  + \frac{1}{t_N} \int_{0}^{t_N} [{f}(\bX_{\un{s}})- {f}(\bX_{s})] \text{d}s.
\end{align*}
{where $\bar{{\cal L}}$ has been defined in \eqref{def:LL_b}.}
To handle the first term of the right-hand side, we use the  Ito formula to obtain:
$$
g(\bX_{t_N}) = g(x_{0})+
\int_{0}^{t_N} \bar{{\cal L}}  g(\bX_s,\bX_{\un{s}}) \text{d}s+ \mathcal{M}_{t_N}^{(g)}\quad\textnormal{with}\quad \mathcal{M}_{t_N}^{(g)}=2\int_0^{t_N} Dg(\bX_s) \text{d}B_s.
$$ 
By \eqref{def:LL_b}, we remark that
$$
{\cal L} g(x) - \bar{{\cal L}}g(x,\un{x}) =  D g(x) [\nabla W(\un{x})-\nabla W(x)].
$$
We then obtain that:
\begin{align}
\pi_N(f)-\pi(f) = 
& = \overbrace{\frac{g(\bX_{t_N}) - g(x_{0})}{t_N}}^{:=A_{t_N}^{(0)}} - \overbrace{\frac{\mathcal{M}_{t_N}^{(g)}}{t_N}}^{:=A_{t_N}^{(1)}}  + \overbrace{\frac{1}{t_N} \int_{0}^{t_N} Dg(\bX_{{s}})[\nabla W(\bX_{\un{s}})-\nabla W(\bX_{s})] \text{d}s}^{:=A_{t_N}^{(2)}} \nonumber\\
 &+\underbrace{ \frac{1}{t_N} \int_{0}^{t_N} [ f(\bX_{\un{s}}) - f(\bX_{s})] \text{d}s}_{:=A_{t_N}^{(3)}}.
\end{align}

\noindent The rest of the proof consists in studying the mean-squared error related to each term of the above righ-hand side and to deduce the result the upper-bound for  $\E|\pi_N(f)-\pi(f)|^2$.\\

\noindent \underline{\textit{Step 2: Mean squared error related to $A_{t_N}^{(0)}$, $A_{t_N}^{(1)}$ and $A_{t_N}^{(3)}$}:}\\

\noindent $\bullet$ By Proposition \ref{prop:fundamentalbounds2} $ii)-b)$, {there exists a constant $c_\mte$ depending only on $\mte$ (and which may change from line to line) such that},
{
$$ \E_{x_0}\left|A_{t_N}^{(0)}\right|^2=\E_{x_0}\left| \frac{g(\bX_{t_N}) - g(\bX_{0})}{t_N} \right|^2
\le c_\mte \deux^{-3-2\mte}  \frac{\Upsilon^{(1+3r)(1+\mte)}}{t_N^2}.$$
}

\noindent $\bullet$ Let us consider the martingale term $A_{t_N}^{(1)}$:
$$ \E_{{x_0}}[|{\cal M}_{t_N}^{(g)}|^2]= 2 \int_0^{t_N}   \E_{x_0}\|D g(\bX_s)\|_F^2 \text{d}s,$$
where $\|.\|_F$ refers to the Frobenius norm. Then, since for a {$d'\times d$-matrix $A$, $\|A\|_F^2\le {\rm Tr}(A A^T)\le d' \nsp{A}^2$},
Proposition \ref{prop:fundamentalbounds2} implies:
{
$$\E_{{x_0}}\left|A_{t_N}^{(1)}\right|^2 \le c_\mte   \deux^{-2(1+\mte)} \frac{d' \Upsilon^{2r(1+\mte)}}{t_N}  .$$}
\noindent $\bullet$ Let us now consider $A_{t_N}^{(3)}$.  On $[t_{k-1},t_k)$:
\begin{equation}
 \bX_s-\bX_{\un{s}} =-(s-\un{s}) \nabla W(\bX_{\un{s}}) + \Delta_{\un{s}s}. \label{eq:dif_bxs_bx}
 \end{equation}
 with $\Delta_{\un{s} s}=\sqrt{2}(B_s-B_{t_{k-1}})$.
 Thus, by Taylor formula, there exists  $\xi(s)\in [\bX_{\un{s}},\bX_{\un{s}}+ \Delta_{\un{s}s}]$ such that
 $$f(\bX_s)-f(\bX_{\un{s}})= {f(\bX_s)-f(\bX_s- (s-\un{s}) \nabla W(\bX_{\un{s}})}+ Df(\xi_2(s)) \Delta_{\un{s}s}.$$
Thus, 
$$\int_{t_{k-1}}^{t_k}
  f(\bX_s)-f(\bX_{\un{s}})\text{d}s=\underbrace{\int_{t_{k-1}}^{t_k} Df(\bar{X}_{\un{s}}) \Delta_{\un{s}s}\text{d}s}_{\Delta N_k}+R(s)$$
where 
$$|R(s)|\le \flip (s-\un{s})  |\nabdob(\bX_{\un{s}}) |+ \dflip |\Delta_{\un{s}s}|^2.$$

  On one hand, since $(\Delta N_k)_{k\ge1}$ is a sequence of martingale increments (and $\flip\le 1$),  
\begin{align*}
\E_{{x_0}}\left[\left(\frac{1}{t_N}\sum_{k=1}^N\Delta N_k\right)^2\right]&=\frac{1}{t_N^2}\sum_{k=1}^N\E_{{x_0}}|\Delta N_k|^2
\le \frac{\flip^2}{t_N^2} \sum_{k=1}^N \E \left|
\int_{t_{k-1}}^{t_k} (B_s-B_{t_{k-1}}) \text{d}s\right|^2\\
& \le  \frac{ N}{t_N^2} 
\E \left|
\int_{0}^{\gamma} B_s \text{d}s\right|^2 \le \frac{1}{N\gamma^2} \E \left|
\int_{0}^{\gamma} (\gamma-s) \text{d}B_s \right|^2 \le  \frac{ d \gamma}{3N}.
\end{align*}
On the other hand, the Jensen inequality (combined with the fact $\flip\le 1$ and $\dflip\le 1$) yields:



\begin{align*}
\E_{{x_0}}\left[\left(\frac{1}{t_N}\int_0^{t_N} R(s) ds\right)^2\right]
&\le \frac{2}{t_N}\int_0^{t_N} \left((s-\un{s})^2 \E\left|\nabdob(\bX_{\un{s}}) \right|^2+2\E |\Delta_{\un{s}s}|^4\right) ds\\
&\le \frac{2}{3}\gamma^2 \sup_{s\ge1} \E\left|\nabdob(\bX_{\un{s}}) \right|^2+ \frac{10 d^2}{t_N}\int_0^{t_N}(s-\un{s})^2 ds\\
&\lesssim_{uc} \frac{\troiss \gamma^2}{1-q} (\dob^{1-q}(x_0)+\dw^{1-q})+ d^2\gamma^2,
\end{align*}
where in the last inequality, we used Lemma \ref{lem:gradWmaindoc} and Proposition \ref{lem:expbounds22} $ii)$.  
We deduce from what precedes that:
{
\begin{align*}
\E_{{x_0}}\left|A_{t_N}^{(3)}\right|^2\lesssim_{uc} \frac{d\gamma^2}{t_N}+\frac{\troiss \gamma^2}{1-q}(\dw^{1-q}+W(x_0)^{1-q})+d^2\gamma^2.
\end{align*}
}



\noindent \underline{\textit{Step 3: Mean squared error related to $A_{t_N}^{(2)}$}} The study of this term is isolated not only because its study is more involved, but also because this term differentiates the bound of $i)$ and $ii)$.

\noindent We separate the drift and the diffusion components and recall that $\Delta_{\un{s}s}=\sqrt{2}(B_s-B_{\un{s}})$. We have:
 
 \begin{align*}
A_{t_N}^{(2)} =& -\overbrace{\frac{1}{t_N} \int_{0}^{t_N} Dg(\bX_{{s}})[\nabla W(\bX_{s})-\nabla W(\bX_{\un{s}}+\Delta_{\un{s}s})] {\rm d}s}^{:=A_{t_N}^{(2,1)}}\\
&- \underbrace{\frac{1}{t_N} \int_{0}^{t_N} Dg(\bX_{{s}})[\nabla W(\bX_{\un{s}}+\Delta_{\un{s}s})-\nabla W(\bX_{\un{s}})]{\rm d}s}_{:=A_{t_N}^{(2,2)}}.
\end{align*}
%
{Since $\nabla W$ is $L$-Lipschitz, Equation \eqref{eq:dif_bxs_bx} yields 
$|\nabla W(\bX_{s})-\nabla W(\bX_{\un{s}}+\Delta_{\un{s}s})|\le  L|s-\un{s}|. |\nabla W(\bX_{\un{s}})|$}. Then, Jensen and Cauchy-Schwarz inequalities imply that:
$$\E_{{x_0}}[|A_{t_N}^{(2,1)}|^2]\le \frac{L^2}{t_N} \int_{0}^{t_N}(s-\un{s})^2 \E_{{x_0}}[\nsp{Dg(\bX_{{s}})}^2]\E_{{x_0}}[|\nabdob (\bX_{\un{s}})|^2] {\rm d}s.$$
Using $\Hcu$, Proposition \ref{lem:expbounds22} $ii)$, Proposition \ref{prop:fundamentalbounds2} $i-b)$, we have:
{
$$\sup_{s\ge0}\E_{x_0}[\nsp{Dg(\bX_{{s}})}^2]\E_{x_0}[|\nabdob (\bX_{\un{s}})|^2] \lesssim_{uc} c_\mte   \deux^{-2(1+\mte)} \dw^{2r(1+\mte)} \troiss (\dw^{1-q} + W(x_0)^{1-q})
 ,$$}
so that:
{
$$\E_{{x_0}}[|A_{t_N}^{(2,1)}|^2]\lesssim_{uc}  L^2  \deux^{-2(1+\mte)}\troiss   (\dw^{(1-q+2r)(1+\mte)}+\dw^{2r(1+\mte)} W(x_0)^{1-q}) \gamma^2.$$
}
We finally separate the study of $ A_{t_N}^{(2,2)}$ into two cases, respectively for $i)$ and $ii)$.\\

\noindent \underline{\textit{Step 4a: End of Proof of Theorem \ref{cor:thmPoisson} $(i.a)$}:}
The Cauchy-Schwarz inequality yields
\begin{equation}\label{eq:firstview}
\E_{{x_0}}[|A_{t_N}^{(2,2)}|^2]\le \frac{L^2}{t_N} \int_{0}^{t_N} \E_{{x_0}}[\nsp{Dg(\bX_{{s}})}^2]\E_{{x_0}}[|\Delta_{\un{s}s}|^2] {\rm{d}}s 
\le \frac{2L^2}{t_N} \int_{0}^{t_N} \E_{{x_0}}[\nsp{Dg(\bX_{{s}})}^2] d (s-\un{s}) {\rm{d}}s
.
\end{equation}
Again Proposition \ref{prop:fundamentalbounds2} $i-a)$ implies that:
$$\E_{{x_0}}[|A_{t_N}^{(2,2)}|^2]\lesssim_{uc}  c_\mte\frac{L^2}{t_N} \int_0^{t_N}   \deux^{-2(1+\mte)} \dw^{2r(1+\mte)} d (s-\un{s}) {{\rm{d}}s}\lesssim_{uc}  c_\mte   \deux^{-2(1+\mte)} L^2  {\dw^{2r(1+\mte)} d} {\gamma}.$$
{Collecting the bounds obtained for $A_{t_N}^{(0)}$, $A_{t_N}^{(1)}$, $A_{t_N}^{(2,1)}$, $A_{t_N}^{(2,2)}$ and $A_{t_N}^{(3)}$,
we get:
\begin{align}
 &\E\Big|\frac{1}{t_N} \sum_{k=1}^N \gamma_k {f}(\bX_{t_{k-1}}) -\pi(f)\Big|^2\lesssim_{uc} \max\left( \deux^{-3-2\mte}  \frac{\Upsilon^{(1+3r)(1+\mte)}}{t_N^2},\deux^{-2(1+\mte)} \frac{d' \Upsilon^{2r(1+\mte)}}{t_N}, \frac{d \gamma^2}{t_N}\right) \label{eq:first3452}\\
 &+
\max\left( \troiss \gamma^2 (\dw^{1-q}+W^{1-q}(x_0)), L^2  \deux^{-2(1+\mte)}\troiss   {\dw^{2r(1+\mte)}}(\dw^{1-q}+W(x_0)^{1-q})) \gamma^2,
 \deux^{-2(1+\mte)} L^2  {\dw^{2r(1+\mte)} d} {\gamma}\right).\label{eq:first3453}
\end{align}
}

In order to deduce Theorem \ref{cor:thmPoisson} $i.a)$, we need to calibrate $\gamma$ and $N$ in order that whole these terms be dominated by 
$\varepsilon^2$ up to a constant depending only on $r$, $\mte$ and $q$. 
{The largest dependency in terms of $\gamma$ is induced by the last term in \eqref{eq:first3453} and we choose:
\begin{equation} 
\label{eq:gamma_principale}
\gamma \leq \frac{\deux^{2(1+\mte)}}{d L^2 \dw^{2r(1+\mte)}} \varepsilon^2.
\end{equation}
The other terms that involve $\gamma$ are upper bounded as follows:
$$
\frac{\troiss \gamma^2}{1-q}(\dw^{1-q}+W(x_0)^{1-q}) \leq \varepsilon^2 
\frac{\deux^{2(1+\mte)}}{d L^2 \dw^{2r(1+\mte)}} \frac{ \gamma \troiss (\dw^{1-q}+W(x_0)^{1-q})}{1-q},
$$
so that $\gamma$ must satisfy:
$$
\gamma \leq \frac{\dw^{2r(1+\mte)} d L^2}{\troiss \deux^{2(1+\mte)}}(\dw^{q-1} \wedge W(x_0)^{q-1}).
$$
Using Lemma \ref{lem:gradWmaindoc}, we have
$$
W(x_0)^{q-1} \geq \frac{1}{ (1+q)^{\frac{1-q}{1+q}}`\deux^{\frac{1-q}{1+q}} |x_0-x^\star|^{2 \frac{1-q}{1+q}}}.
$$
\noindent
In the same way, the middle term in \eqref{eq:first3453} induces that:
$$
\gamma \leq \frac{d}{\troiss \dw^{1-q-2r \mte}} \wedge \frac{d}{\troiss 
(1+q)^{\frac{1-q}{1+q}}`\deux^{\frac{1-q}{1+q}} \dw^{-2r \mte}|x_0-x^\star|^{2 \frac{1-q}{1+q}}
}.
$$
}

Let us now consider \eqref{eq:first3452} to calibrate $N_\varepsilon$. We identify again the largest element (in terms of $t_N$) and choose $t_N$ such that:
$$
t_N \ge \varepsilon^{-2} \frac{d' \dw^{2r(1+\mte)}}{\deux^{2(1+\mte)}},
$$
so that the   dependency on $\varepsilon$ is translated by:
$$
N_{\varepsilon} \ge \varepsilon^{-4} \frac{d' d L^2 \dw^{4 r(1+\mte)}}{\deux^{4(1+\mte)}} \vee \varepsilon^{-2}  \frac{d'  \dw^{(2 r)(1+\mte)}(\dw^{1-q} + W(x_0)^{1-q})\troiss }{d \deux^{2(1+\mte)}} \vee \varepsilon^{-2} \frac{\dw^{1-q} \troiss}{d L^2}.
$$
The two other terms of \eqref{eq:first3452} must also be upper bounded by $\varepsilon^2$. Using the constraint on $\gamma$ induced by \eqref{eq:gamma_principale} and $t_N = N_{\varepsilon} \gamma$, we verify that:
$$
 \frac{d \gamma^2}{t_N} \leq \varepsilon^2 \frac{\deux^{2(1+\mte)}}{L^2 \dw^{2r(1+\mte)} N_{\varepsilon}},
$$
so that $N_{\varepsilon}$ must satisfy:
$$
N_{\varepsilon} \ge \frac{\deux^{2(1+\mte)}}{L^2 \dw^{2r(1+\mte)} }.
$$
Finally, the last term involved in \eqref{eq:first3452} is automatically upper bounded by $\varepsilon^2$ since $t_N \ge 1$ induces the constraint
$$ t_N \ge \varepsilon^2 \frac{1}{d' \deux \dw^{(1+r)(1+\mte)}}.$$ \\

For Theorem \ref{cor:thmPoisson} $(i.b)$, we remark that 
$\dw \leq c (\log d) d^{\frac{1}{1+q-r}}$ where $c$ depends on $\mte,r,q,\deux,\troiss$ and $L$ but not on $\varepsilon$ and $d$. Plugging this estimate of $\dw$ into \eqref{eq:first3452} and \eqref{eq:first3453} leads to 
\begin{align*}
 \E\Big|\frac{1}{t_N} \sum_{k=1}^N \gamma_k {f}(\bX_{t_{k-1}}) -\pi(f)\Big|^2\leq c  d^{\tilde{\mte}} \max\left( \frac{d^{\frac{1+3r}{1+q-r}}}{t_N^2},\frac{d' d^{\frac{2r}{1+q-r}}}{t_N},d^{\frac{1-q+2r}{1+q-r}}\gamma^2,  d^{1+\frac{2r}{1+q-r}} {\gamma}\right),
\end{align*}
where $\tilde{\mte}$ denotes an arbitrary small positive number. The result easily follows using that the first term is smaller than the second one under the condition  $\varepsilon\le d' d^{-\frac{1-r}{2(1+q-r)}}$.\\

%
%
%
%
%

%
  

\noindent \underline{\textit{Step 4b: End of Proof of Theorem \ref{cor:thmPoisson} $(ii.a)$}:}
 The proof is given for any choice of $q$ and $r$ but the statement is valid only when $r=q$.
It is possible to exploit the centering of $\Delta_{\un{s}s}$. We decompose into two parts and we decompose $A_{t_N}^{(2,2)}$ as follows:
\begin{equation}\label{def:atn22}
\begin{split}
A_{t_N}^{(2,2)}&=\overbrace{\frac{1}{t_N} \int_{0}^{t_N} (Dg(\bX_{{s}})-Dg(\tX_{\un{s}}))[\nabla W(\bX_{\un{s}}+\Delta_{\un{s}s})-\nabla W(\bX_{\un{s}})] ds}^{:=\circled{1}}\\
&+\underbrace{\frac{1}{t_N} \int_{0}^{t_N} Dg(\tX_{\un{s}})[\nabla W(\bX_{\un{s}}+\Delta_{\un{s}s})-\nabla W(\bX_{\un{s}})] ds,}_{:=\circled{2}}
\end{split}
\end{equation}
where $\tX_{\un{s}}=\bX_{\un{s}}-(s-\un{s})\nabdob(\bX_{\un{s}})$. For the first term, we use  the Jensen inequality and Proposition \ref{prop:fundamentalbounds2} $iii-b)$. We obtain that:
\begin{equation}\label{eq:ocircqun}
\E_{{x_0}}[|\circled{1}|^2]\lesssim_{uc}  \mathfrak{c}_\mte (L \tilde{L} \deux^{-2(1+\mte)})^2 d^2 \Upsilon^{4r(1+\mte)} \frac{1}{t_N} \int_{0}^{t_N}(s-\un{s})^2
\text{d}s = \mathfrak{c}_\mte  (L \tilde{L} \deux^{-2(1+\mte)})^2 d^2 \Upsilon^{4r(1+\mte)}\frac{\gamma^2}{3}
\end{equation}

Let us finally consider $\circled{2}$. 
We introduce the martingale $(M_t)_{t \ge 0}$ defined as:
$$
M_t = \int_{0}^t \hesdob(\bX_{\un{u}} + \Delta_{\un{u}u}) \text{d}B_u.
$$
Using the {It\^o} formula, we obtain that:

\begin{align*}
\nabdob(\bX_{\un{s}}+\Delta_{\un{s}s})-\nabdob (\bX_{\un{s}}) &= \int_{\un{s}}^{s} \hesdob(\bX_{\un{s}} + \Delta_{\un{s}u}) \text{d}B_u + \int_{\un{s}}^s \vec{\Delta}(\nabdob)(\bX_{\un{s}}+ \Delta_{\un{s}u}) \text{d}u \\
& = M_s-M_{\un{s}} + \int_{\un{s}}^s \vec{\Delta}(\nabdob)(\bX_{\un{s}}+ \Delta_{\un{s}u}) \text{d}u.
\end{align*}
where for a vector field $\phi:\ER^d\mapsto\ER^d$, $\vec{\Delta} \phi=(\Delta \phi_1,\ldots,\Delta \phi_d)^T$ with $\Delta$ standing for the Laplacian operator.
We use this decomposition into $\circled{2}$ and observe that:
\begin{align*}
\circled{2} &=\frac{1}{t_N} \int_{0}^{t_N} Dg(\tilde{X}_{\un{s}}) \left( M_s-M_{\un{s}}\right) \text{d}s +\frac{1}{t_N} \int_{0}^{t_N} Dg(\tilde{X}_{\un{s}})
 \int_{\un{s}}^{s} \vec{\Delta}(\nabdob)(\bX_{\un{s}} + 
\Delta_{\un{s}u}) \text{d}u \text{d}s.
\end{align*}
The Young inequality yields:
\begin{equation}\label{def2a2b}
\E[|\circled{2}|^2] \leq 2 \underbrace{\E \left[ \left|\frac{1}{t_N} \int_{0}^{t_N} Dg(\tilde{X}_{\un{s}}) \left( M_s-M_{\un{s}}\right) \text{d}s  \right|^2\right]}_{:=\circled{2a}}+
2 \underbrace{\E \left[ \left|\frac{1}{t_N} \int_{0}^{t_N} Dg(\tilde{X}_{\un{s}}) \int_{\un{s}}^{s} \vec{\Delta}(\nabdob)(\bX_{\un{s}} + 
\Delta_{\un{s}u}) \text{d}u\text{d}s \right|^2\right]}_{:=\circled{2b}}
\end{equation}
%
%
%
We first deal with \circled{2a} and use the martingale decomposition, we obtain that:
\begin{equation}
\circled{2a} = 
\E\Big[ \Big| \frac{1}{t_N} \sum_{k=1}^{N}  \underbrace{\int_{t_{k-1}}^{t_k}
Dg(\tilde{X}_{t_{k-1}}) (M_s-M_{t_{k-1}}) \text{d}s}_{:=\Delta N_k}  \Big|^2 \Big] = \frac{1}{t_N^2}  \sum_{k=1}^N \E[\E[|\Delta N_k|^2 \, \vert \mathcal{F}_{t_{k-1}}]],\label{eq:decomposition_2a}
\end{equation}
where we used that $\E[\Delta N_k \, \vert \mathcal{F}_{t_{k-1}}]=0$.  Standard computations show that:
\begin{align*}
\E[|\Delta N_k|^2 \, \vert \mathcal{F}_{t_{k-1}}]&=
\E \left[ \int_{t_{k-1}}^{t_k} \int_{t_{k-1}}^{t_k}  
\{Dg(\tilde{X}_{t_{k-1}}) (M_v-M_{t_{k-1}})\}^T Dg(\tilde{X}_{t_{k-1}}) (M_u-M_{t_{k-1}}) \text{d}u \text{d}v  \, \vert \mathcal{F}_{t_{k-1}}\right]\\
& =  \int_{t_{k-1}}^{t_k} \int_{t_{k-1}}^{t_k}  
{\rm Tr}\left( \E \left[
Dg(\tilde{X}_{t_{k-1}})^T Dg(\tilde{X}_{t_{k-1}}) (M_u-M_{t_{k-1}}) (M_v-M_{t_{k-1}})^T\, \vert \mathcal{F}_{t_{k-1}}\right] \right) \text{d}u \text{d}v \\
 & = 2 \int_{t_{k-1}}^{t_k} \int_{t_{k-1}}^{u}  
{\rm Tr}\left(
Dg(\tilde{X}_{t_{k-1}})^T Dg(\tilde{X}_{t_{k-1}}) 
 \E \left[(M_u-M_{t_{k-1}}) (M_v-M_{t_{k-1}})^T\, \vert \mathcal{F}_{t_{k-1}}\right] \right)  \text{d}v \text{d}u.
\end{align*}
The last conditional expectation deserves a specific study: for any pair $(i,j)\in \{1,\ldots,d\}^2$ and any $t_{k-1}\leq v \leq u \leq t_{k}$,  we have:
\begin{align*}
\E  \big(\big[ \big((M_u-M_{t_{k-1}})& (M_v-M_{t_{k-1}})^T\big)_{i,j} \, \vert \mathcal{F}_{t_{k-1}} \big] \big) = 
\E \left(\left[ (M_u^{(i)}-M^{(i)}_{t_{k-1}}) (M^{(j)}_v-M^{(j)}_{t_{k-1}}) \, \vert \mathcal{F}_{t_{k-1}} \right] \right)\\
 &=  \int_{t_{k-1}}^{ v} \sum_{\ell=1}^d \E[ \nabla^2 W^{i,\ell}(\bar{X}_{t_{k-1}}+\Delta_{t_{k-1},s})\nabla^2 W^{j,\ell}(\bar{X}_{t_{k-1}}+\Delta_{t_{k-1},s}) \, \vert \mathcal{F}_{t_{k-1}}]\text{d}s,
\end{align*}
where the previous equality comes from the independent Brownian increments coordinates per coordinates and from $u \wedge v=v$. This induces the matricial equality:
$$
\E \left[  (M_u-M_{t_{k-1}}) (M_v-M_{t_{k-1}})^T  \, \vert \mathcal{F}_{t_{k-1}} \right] =
\int_{t_{k-1}}^{ v}\E[ \nabla^2 W(\bar{X}_{t_{k-1}}+\Delta_{t_{k-1},s})\nabla^2 W(\bar{X}_{t_{k-1}}+\Delta_{t_{k-1},s})^T \, \vert \mathcal{F}_{t_{k-1}}]\text{d}s.
$$
Noting that 
$$
{\rm Tr}\left(
Dg(\tilde{X}_{t_{k-1}})^T Dg(\tilde{X}_{t_{k-1}}) \right) 
=\|  Dg(\tilde{X}_{t_{k-1}})\|_{F}^2,
$$
we then deduce that:
\begin{equation}
\E[|\Delta N_k|^2 \, \vert \mathcal{F}_{t_{k-1}}]\nonumber\\
  = 2 \E \left[ \int_{t_{k-1}}^{t_k} \int_{t_{k-1}}^{u}  \int_{t_{k-1}}^{v}  \|  Dg(\tilde{X}_{t_{k-1}}) 
   \nabla^2 W(\bar{X}_{t_{k-1}}+\Delta_{t_{k-1},s})\|_{F}^2 \text{d}s \text{d}v \text{d}u
 \, \vert \mathcal{F}_{t_{k-1}}
 \right].\label{eq:Delta_Nk_conditionnel}
\end{equation}
Plugging Equation \eqref{eq:Delta_Nk_conditionnel} into Equation \eqref{eq:decomposition_2a}, we deduce that:
\begin{equation}\label{def2aa}
\circled{2a} = \frac{2}{t_N^2} \sum_{k=1}^N   \int_{t_{k-1}}^{t_k} \int_{t_{k-1}}^{u}  \int_{t_{k-1}}^{v}  \E \left[\|  Dg(\tilde{X}_{t_{k-1}}) 
   \nabla^2 W(\bar{X}_{t_{k-1}}+\Delta_{t_{k-1},s})\|_{F}^2\right] \text{d}s \text{d}v \text{d}u.
\end{equation}

We then use the relationship between the spectral and the Frobenius norm and the fact that the spectral norm is {an operator} norm, we observe that for any $k$ and $s \in [t_{k-1},t_k]$:
\begin{align}
\|Dg(\tX_{t_{k-1}})\hesdob(\bX_{t_{k-1}}+\Delta_{t_{k-1},s})\|_F^2 & \leq
d' \nsp{Dg(\tX_{t_{k-1}})\hesdob(\bX_{t_{k-1}}+\Delta_{t_{k-1},s})}^2\nonumber \\
& \leq d' \nsp{\hesdob(\bX_{t_{k-1}}+\Delta_{t_{k-1},s})}^2  \nsp{Dg(\tX_{t_{k-1}})}^2\nonumber\\
& \leq d' L^2 \ \nsp{Dg(\tX_{t_{k-1}})}^2,\label{defdfll}
\end{align}
where we use that $\nabla W$ is a L-Lipschitz function. We then use Proposition \ref{prop:fundamentalbounds2} $i-b)$, a slight adaptation of Proposition \ref{lem:expbounds22} to get the same bounds for $\E[\nsp{Dg(\tX_{\un{t}})}^2]$ as for $\E[\nsp{Dg(\bX_{\un{t}})}^2]$. We conclude that:

\begin{equation}
\circled{2a}  \lesssim_{uc} \frac{c_{\mte} L^2 \deux^{-2r(1+\mte)} d\dw^{2r(1+\mte)}  }{t_N^2} \sum_{k=1}^{N} \int_{t_{k-1}}^{t_k} \int_{t_{k-1}}^u \int_{t_{k-1}}^v \text{d}s \text{d}v \text{d}u =  c_{\mte} L^2 \deux^{-2(1+\mte)} d'\dw^{2r(1+\mte)}\frac{\gamma}{N}
\label{eq:terme1}
\end{equation}

Let us finally consider $\circled{2b}$. By Jensen inequality, 
\begin{align}
\circled{2b} & \leq \E\left[\frac{1}{t_N} \int_{0}^{t_{N}}  \left|  Dg(\tilde{X}_{\un{s}}) 
\int_{\un{s}}^s  \vec{\Delta}(\nabdob)(\bX_{\un{s}}+ \Delta_{\un{s}u}) \text{d}u \right|^2 \text{d}s\right]\nonumber\\
& \leq \E\left[\frac{1}{t_N} 
\int_{0}^{t_{N}} \nsp{Dg(\tilde{X}_{\un{s}})}^2 \|\vec{\Delta}(\nabdob)\|_{2,\infty}^2 (s-\un{s})^2 
\text{d}s\right] \nonumber\\
& \leq \frac{ \|\vec{\Delta}(\nabdob)\|_{2,\infty}^2 }{t_N} \int_{0}^{t_N}
\E[ \nsp{Dg(\tilde{X}_{\un{s}})}^2](s-\un{s})^2 \text{d}s\label{eq:terme2exact}\\
& \lesssim_{uc}  \|\vec{\Delta}(\nabdob)\|_{2,\infty}^2 \gamma^2 \deux^{-2(1+\mte)} \dw^{2r(1+\mte)},\label{eq:terme2}
\end{align}
where the last inequality follows from Proposition \ref{prop:fundamentalbounds2}. \\

\noindent By \eqref{eq:ocircqun}, \eqref{eq:terme1} and \eqref{eq:terme2}, we are now able to give a new bound for $A_{t_N}^{(2,2)}$:

\begin{align*}
 \E\left|A_{t_N}^{(2,2)}\right|^2&\lesssim_{uc}  \mathfrak{c}_\mte  \deux^{-2(1+\mte)} \gamma^2 \max\left(  \deux^{-2(1+\mte)}(L\tilde{L})^2 d^2 \Upsilon^{4r(1+\mte)},\|\vec{\Delta}(\nabdob)\|_{2,\infty}^2  \dw^{2r(1+\mte)}\right)\\
 &+c_{\mte}  \deux^{-2(1+\mte)} d'\dw^{2r(1+\mte)}\frac{(L\gamma)^2}{t_N}.
\end{align*}



{Collecting the bounds obtained for $A_{t_N}^{(0)}$, $A_{t_N}^{(1)}$, $A_{t_N}^{(2,1)}$, $A_{t_N}^{(2,2)}$ and $A_{t_N}^{(3)}$,
we get (using that $L\gamma\le 1$):
\begin{align}
& \E\Big|\frac{1}{t_N} \sum_{k=1}^N \gamma_k {\id}(\bX_{t_{k-1}}) -\pi(\id)\Big|^2\lesssim_{uc} c_\mte \max\left( \deux^{-3-2\mte}  \frac{\Upsilon^{(1+3r)(1+\mte)}}{t_N^2},\deux^{-2(1+\mte)} \frac{d' \Upsilon^{2r(1+\mte)}}{t_N} \right) \label{eq:first3454}\\
 &+
c_\mte  \deux^{-2(1+\mte)} \gamma^2 \max\left( L^2\troiss  {\dw^{(1-q+2r)(1+\mte)}} , (\deux^{-(1+\mte)}L\tilde{L} d)^2 \Upsilon^{4r(1+\mte)},\|\vec{\Delta}(\nabdob)\|_{2,\infty}^2  \dw^{2r(1+\mte)}\right).\label{eq:first3455}
\end{align}
}
In order to ensure that $\eqref{eq:first3455}\le c_{\mte,r,q}\varepsilon^2$, we fix $\gamma_\varepsilon= c_{2,1} \varepsilon$ with 
$$ c_{2,1}=  \deux^{1+\mte}\min\left(L^{-1}\troiss^{-\frac{1}{2}}  {\dw^{-\frac{1-q+2r}{2}(1+\mte)}} , \deux^{1+\mte}(L\tilde{L})^{-1} d^{-1} \Upsilon^{-2r(1+\mte)},\|\vec{\Delta}(\nabdob)\|_{2,\infty}^{-1}  \dw^{-r(1+\mte)}\right).$$
Then, the condition on $N_\varepsilon$ comes from \eqref{eq:first3454} with exactly the same form as in $(i.b)$ (since the right-hand side of \eqref{eq:first3454} is the same as in 
\eqref{eq:first3452}). This concludes the proof of $(ii.a)$.\\

{For $(ii.b)$, we use that $\dw\le c (\log d) d^{\frac{1}{1+q-r}}$ where $c$ does not depend on $d$ so that the previous bound leads to:  }
\begin{align*}
 \E\Big|\frac{1}{t_N} \sum_{k=1}^N \gamma_k {\id}(\bX_{t_{k-1}}) -\pi(\id)\Big|^2\leq c  d^{\tilde{\mte}} \max\left( \frac{d^{\frac{1+3r}{1+q-r}}}{t_N^2},\frac{d'd^{\frac{2r}{1+q-r}}}{t_N},\gamma^2 d^{2+\frac{4r}{1+q-r}},  \gamma^2 d^{2\rho+\frac{2r}{1+q-r}}\right),
\end{align*}
where $\tilde{\mte}$ denotes an arbitrary small positive number. 
Setting $\gamma_\varepsilon=\varepsilon d^{-\max(1+\frac{2r}{1+q-r},\rho+\frac{r}{1+q-r})-\tilde{\mte}}$ allows to control the two last terms. Then, one sets $N_\varepsilon=d'd^{\frac{2r}{1+q-r}}\gamma_{\varepsilon}^{-1}\varepsilon^{-2}$ in order to control the second term and finally checks that the first one is also controlled by $c\varepsilon^2$ under this condition  and the fact that $\varepsilon\le d' d^{-\frac{1-r}{2(1+q-r)}}$ (where $c$ does not depend on $d$). 
\end{proof}

 \subsection{Bayesian learning with discrete LMC - weakly convex case - Theorem \ref{theo:learning_weak_convex}}
\label{sec:learning_weak}


\begin{proof}[Proof of Theorem \ref{theo:learning_weak_convex}]

We know that for any $\xi$, $U(\xi,.)$ satifies $\AL$. It implies that $W_n$ is $nL$-smooth.
 Since  $U(\xi,.)$ satisfies $\Hcu$, a direct computation shows that $W_n$ satisfies $\Hcu$ with $q=0$ and $\troiss=nL$.
In the meantime, Assumption $\Hcu$ on each $U(\xi,.)$ implies that:
\begin{align*}
\underline{\lambda}(\nabla^2 W_n) &= \underline{\lambda} \left( \sum_{i=1}^n {\nabla^2 }U(\xi_i,.)\right) \\
& \ge \sum_{i=1}^n \underline{\lambda}(\nabla^2 U(\xi_i,.))\\
& \ge \sum_{i=1}^n \deux U(\xi_i,.)^{-r} = n \left(\frac{1}{n} \sum_{i=1}^n \deux U(\xi_i,.)^{-r}\right)\\ 
& \ge \{\deux n^{1-r}\} W_n^{-r}(\mathbf{\xi^n},x),
\end{align*}
where we applied the Jensen inequality to the convex function $u \mapsto u^{-r}$. Thus $\Hcu$ holds with the pair $(\tilde{\deux}=\deux n^{1-r},r)$ and  $(\tilde{\troiss}=nL,0)$. By Proposition \ref{lem:expbounds22}, we thus deduce that
$$\sup_{t\ge0} \E_x[\dob_n ^p (X_{t})]+\sup_{t\ge0} \E_x[\dob_n ^p (\bX_{t})]\le c_p \left(  \dob_n^p (x)+\dw^p_n\right),$$
with 
$$
\dw_n =c_r (n L) (\deux n^{1-r})^{\frac{-1}{1-r}} d  \log(1+dnL) =c_{r,L,\deux,\mte} d n^{\mte}
$$
where $\mte$ is an arbitrary small number and $c_{r,L,\deux,\mte}$ is a constant depending only on $r$, $L$, $\deux$ and $\mte$.
We are now able to apply Theorem \ref{cor:thmPoisson}(i.a) with $\Upsilon=\Upsilon_n$. We observe that $\gamma_{\varepsilon_n}$ must be chosen such that
$$
\gamma_{\varepsilon_n} = \frac{n^{2(1-r)}}{d n^2 d^{2r(1+\mte)}}\varepsilon_n^2 \wedge \frac{d}{n d^{1-q-2 r \mte}} \wedge \frac{d}{n d^{-2 r \mte} (d n)^{\frac{1-q}{1+q}}} \wedge \frac{d^{2r+q-1} d n^2}{n n^{2(1-r)}}
$$
For $n$ large enough, the minimum is then related to
$$
\gamma_{\varepsilon_n} =n^{-(2r+\alpha^{-1})} d^{\alpha^{-1}-1-2r} \wedge \frac{d^{q+2 r \mte}}{n} \wedge \frac{d^{\frac{2q}{1+q}+2 r \mte}}{n^{\frac{2}{1+q}}} \wedge \frac{d^{2r+q}}{n^{1-2r}}.
$$ 
We then verify that when $n \ge d$, we have:
$$
\forall q \in [0,1] \qquad 
\frac{d^{\frac{2q}{1+q}+2 r \mte}}{n^{\frac{2}{1+q}}} \leq \frac{d^{q+2 r \mte}}{n} \leq \frac{d^{2r+q}}{n^{1-2r}},
$$
so that
$$
\gamma_{\varepsilon_n} = n^{-(2r+\alpha^{-1})} d^{\alpha^{-1}-1-2r}
\wedge \frac{d^{\frac{2q}{1+q}+2 r \mte}}{n^{\frac{2}{1+q}}}.
$$
 We observe that 
$$
N_{\varepsilon_n} = \gamma_{\varepsilon_n}^{-1} \left[ 
\underbrace{d^{\frac{1+3r}{2}} n^{-\frac{3}{2}(1-r)} \left(\frac{n}{d}\right)^{\frac{1}{2\ic}}}_{:=A} \vee \underbrace{d^{2r} n^{-2(1-r)} \left(\frac{n}{d}\right)^{\frac{1}{\ic}}}_{:=B}
\right] 
$$
The study is then made explicit by considering some specific cases.
\begin{itemize}
\item If $r \ge 1$, we observe that $B \ge A$ for all values of $t \in [0,1)$. In that case we also verify that 
$
\gamma_{\varepsilon_n}^{-1} = n^{2r+\alpha^{-1}} d^{2r+1-\alpha^{-1}}.
$
The overall cost of computation is then given by
$$
N_{\varepsilon_n} = n^{2 \ic^{-1}+4r} d^{1+4r-2\ic^{-1}}.
$$
\item If $r=q=0$, since $d<n$, we verify that:
$$
\gamma_{\varepsilon_n}^{-1} = n^{\ic^{-1}}d^{1-\ic^{-1}} \vee n^2 = n^{2}
$$
and 
$$
N_{\varepsilon_n} = n^{\frac{1}{2}(1+\ic^{-1})} d^{\frac{1}{2}(1-\ic^{-1})}
$$

\end{itemize}
\end{proof}

\section{Discretization of the Langevin procedure - Strongly convex case}\label{sec:discretisation}

We adapt the proof of Theorem \ref{cor:thmPoisson} to the strongly convex case and to this end, start with an adaptation of Proposition \ref{prop:fundamentalbounds2}:
\begin{prop}\label{prop:bisbis} Assume $\SC$. {Let $f:\ER^d\rightarrow\ER^{d'}$ be a Lispchitz ${\cal C}^2$-function with $\flip\le 1$ and $\dflip\le 1$}.    Then, $g$ is  a ${\cal C}^2$-function and for every $\mte\in(0,1)$, a constant  $c_\mte$ exists (which only depends on $\mte$ and not on $d$), such that for any $x$:
\begin{enumerate}

\item[$i)$]

{$\nsp{Dg(x)}\le \frac{1}{\rho}.$}
\item[$ii)$] Let $\gamma \in (0,\rho/(2L^2)]$. Then,
$$\sup_{t\ge0} \E_{x_0}[| g(\bX_t)-g(x_0)|^2]\le \frac{4}{\rho^2} \left(|x_0-x^\star|^2+ \frac{d}{\rho}\right).$$
\item[$iii)$] Assume that $\hesdob$ is $\tilde{L}$-Lipschitz for the norm $\nsp{.}$. Then, {$x\mapsto Dg(x)$ is Lipschitz and its related Lipschitz constant 
$[Dg]_{1,\star}$ satisfies:}
$$[Dg]_{1,\star}=\sup_{x\neq y} \frac{\nsp{Dg(y)-Dg(x)}}{|y-x|}\le 
\left( {\frac{2 \tilde{L} }{\rho^2}} + {\frac{1}{2 \rho}}\right).$$
\end{enumerate}
\end{prop}

\begin{proof}[Proof of Theorem  \ref{cor:thmPoissonSC}]
(i) We follow the proof of Theorem  \ref{cor:thmPoisson}(i). Keeping the notations of the proof of Theorem  \ref{cor:thmPoisson},  we first deduce from $(ii)$ that
$$\E_{x_0}\left|A_{t_N}^{(0)}\right|^2\le \frac{4}{\rho^2 t_N^2} \left(|x_0-x^\star|^2+ \frac{d}{\rho}\right)\le \frac{8d}{\rho^3 t_N^2}$$
since $|x_0-x^\star|^2\le d/\rho$.


Second, by the deterministic bound $i)$ of Proposition \ref{prop:bisbis}: $\nsp{Dg(x)} \le \rho^{-1}$  and the arguments given in the proof of Theorem \ref{cor:thmPoisson}, we have:
$$\E_{{x_0}}\left|A_{t_N}^{(1)}\right|^2\le \frac{d'}{\rho^2 t_N}.$$
For $A_{t_N}^{(3)}$, we remark that the only modification comes from the control of 
$\sup \E|\nabdob(\bar{X}_{\un{s}})|^2$.
Since $\nabdob$ is $L$-Lipschitz and $\nabdob(x^\star)=0$,  
$$\E|\nabdob(\bar{X}_{\un{s}})|^2\le L^2 \E[|\bar{X}_{\un{s}}-x^\star|^2].$$
By Lemma 5.1 of \cite{egea-panloup}, we deduce that 
$$\E|\nabdob(\bar{X}_{\un{s}})|^2\le L^2\left(|x_0-x^\star|^2+ {\frac{2d}{\rho}}\right)\le \frac{8L^2d}{\rho},$$
since $|x_0-x^\star|^2\le d/\rho$.
\begin{align*}
\E_{{x_0}}\left|A_{t_N}^{(3)}\right|^2\le \frac{d\gamma^2}{3 t_N}+
{L^2 \gamma^2 \left(|x_0-x^\star|^2+\frac{d}{\rho}\right)}
+10d^2\gamma^2.
\end{align*}

For $A_{t_N}^{(2,1)}$, the fact that $\nsp{Dg(x)}\le \rho^{-1}$  yields:
%
%
\begin{align*}
\E_{{x_0}}[|A_{t_N}^{(2,1)}|^2]&\le \frac{L^2}{\rho^2 t_N} \int_{0}^{ t_N}(s-\un{s})^2 \E_{{x_0}}[|\nabdob (\bX_{\un{s}})|^2] {\rm d}s\\
&\le  \frac{L^2}{\rho^2t_N}\times  { L^2\left(|x_0-x^\star|^2+ {\frac{2d}{\rho}}\right)}\times N \gamma^3 = { 
 \frac{L^4}{\rho^2}\times  \left(|x_0-x^\star|^2+ \frac{2d}{\rho} \right)\times \gamma^2 
}
\end{align*}

Finally, for $A_{t_N}^{(2,2)}$, we deduce from \eqref{eq:firstview} that
\begin{equation*}
\E_{{x_0}}[|A_{t_N}^{(2,2)}|^2]\le \frac{L^2}{\rho^2 t_N} \int_{0}^{t_N} E_{{x_0}}[|\Delta_{\un{s}s}|^2] {\rm{d}}s 
\le \frac{2L^2 d}{\rho^2 t_N} \int_{0}^{t_N} (s-\un{s}) {\rm{d}}s\le \frac{2L^2 d}{\rho^2}\gamma 
.
\end{equation*}
Thus, collecting the previous bounds and $L^2 d / \rho \leq L^4 d / \rho^3$:
\begin{align*}
\E\Big|\frac{1}{t_N} \sum_{k=1}^N \gamma_k {f}(\bX_{t_{k-1}}) -\pi(f)\Big|^2 & \le \frac{d'}{\rho^2 t_N}+\frac{8d}{\rho^3 t_N^2}+\frac{d\gamma^2}{3 t_N}+10d^2\gamma^2+ \frac{8 L^4 d}{\rho^3} \gamma^2+\frac{2L^2 d}{\rho^2}\gamma\\
& + {2 \frac{L^4 }{\rho^2}  \gamma^2 |x_0-x^\star|^2}.
\end{align*}
{The larger term above regarding the size of $\gamma$ is $L^2 \rho^{-2 }d  \gamma$ and we are led to choose $\gamma \leq \frac{\rho^2}{2 L^2 d }\varepsilon^2$. We now verify the size of the other terms.}
Obviously we have:
$$
{d^2 \gamma^2 \leq d^2 \frac{\rho^4}{L^4 d^2} \varepsilon^4  \leq \varepsilon^4}
$$
since $\rho \leq L$.
We  observe that:
$$
{\frac{L^4 d }{\rho^3} \gamma^2 \leq \frac{L^4 d}{\rho^3} \frac{\rho^2 }{L^2 d} \varepsilon^2 \gamma = \varepsilon^2  \left(\frac{L^2}{\rho} \gamma\right)}
$$
At last, the terms involving the initial condition may be upper bounded in the same way:
$$
{2 \frac{L^4 }{\rho^2} \gamma^2 |x_0-x^\star|^2 \leq \varepsilon^2 \times \frac{\rho^2}{2 L^2 d} \left(\gamma |x_0-x^\star|^2 \times 2 \frac{L^4 }{\rho^2} \right) \leq 
\varepsilon^2 \times \frac{\gamma |x_0-x^\star|^2 L^2}{d}.
}
$$
{Using again that $\rho \leq L$, we then deduce that we have to choose $\gamma$ such that:
$$
\gamma \leq \frac{\rho^2}{L^2 d } \varepsilon^2 \wedge \frac{\rho}{L^2}  \wedge \frac{d}{L^2 |x_0-x^\star|^2}.
$$
}
{ Now, fix $N \ge\varepsilon^{-4} \frac{L^2}{\rho^4} d' d $ in such a way that $\frac{d'}{\rho^2 t_N}\le 2\varepsilon^2$.
Then,
$$\frac{8d}{\rho^3 t_N^2}= \frac{d'}{\rho^2 t_N} \times \frac{8 d}{\rho d' t_N}
\le \varepsilon^2 \times \frac{16 d }{\rho  d' t_N} \le 16 \varepsilon^2,$$
under the condition $N \ge \frac{d}{d' \rho} \gamma^{-1}$.
 Finally, since $\gamma\le \frac{\varepsilon^2}{d}$ and $N \ge 1$ it is obvious that $\frac{d\gamma^2}{3 t_N}\le \varepsilon^2/2$. The result follows.}\\

\noindent $(ii)$ As in $(i)$, we follow the proof of Theorem  \ref{cor:thmPoisson}(ii). With the notations of this proof, we check that the only difference between $(i)$ and $(ii)$ comes from the control of $A_{t_N}^{(2,2)}$. Actually, in this part, owing to a further expansion, $\E_{{x_0}}[|A_{t_N}^{(2,2)}|^2]$ is controlled by $\E_{{x_0}}[|\circled{1}|^2]+\E[|\circled{2}|^2]$ where $\circled{1}$ and $\circled{2}$ are defined in \eqref{def:atn22}.
On the one hand, by Proposition \ref{prop:bisbis}(iii), the Jensen inequality and $(a+b)^2 \leq 2 a^2+2b^2$: 
\begin{align*}
\E_{{x_0}}[|\circled{1}|^2]&\le L^2{ \left( \frac{1 }{2 \rho^2}+ \frac{8 \tilde{L}^2}{\rho^4}\right)} \frac{1}{t_N}\int_0^{t_N} \E[ |\Delta_{\un{s}s}|^4] ds\\
&\le  L^2 {\left( \frac{1 }{2 \rho^2}+ \frac{8 \tilde{L}^2}{\rho^4}\right)}(4d^2-d)\frac{1}{t_N} \int_0^{t_N} (s-\un{s})^2 ds\le  \frac{4\gamma^2}{3 } L^2 {\left( \frac{1 }{2 \rho^2}+ \frac{8 \tilde{L}^2}{\rho^4}\right)}.
\end{align*}
For $\E[|\circled{2}|^2]$, we recall that this term is divided into two parts denoted by $\circled{2a}$ and $\circled{2b}$ (see \eqref{def2a2b}. 
On the one hand,  by \eqref{def2aa}, \eqref{defdfll} and Proposition \ref{prop:bisbis}(i),
\begin{equation*}
\circled{2a} \le  \frac{2 d' L^2 }{\rho^2 t_N^2} \sum_{k=1}^N   \int_{t_{k-1}}^{t_k} \int_{t_{k-1}}^{u}  \int_{t_{k-1}}^{v}  \text{d}s \text{d}v \text{d}u\le 
\frac{ d' L^2 \gamma^2}{3\rho^2 t_N}.
\end{equation*}
Finally, for $\circled{2b}$, we derive from \eqref{eq:terme2exact} and Proposition \ref{prop:bisbis}(i) that, 
$$\circled{2b}\le \frac{ \|\vec{\Delta}(\nabdob)\|_{2,\infty}^2 }{\rho^2 t_N} \int_{0}^{t_N}
(s-\un{s})^2 \text{d}s\le \frac{ \|\vec{\Delta}(\nabdob)\|_{2,\infty}^2 \gamma^2}{3 \rho^2 }.$$
Finally, we get
$$\E_{x_0}[|A_{t_N}^{(2,2)}|^2]\le \gamma^2 \left(
L^2 {\left( \frac{d'}{2 \rho^2}+ \frac{8 \tilde{L}^2}{\rho^4}\right)}
+\frac{ \|\vec{\Delta}(\nabdob)\|_{2,\infty}^2 }{3 \rho^2 }\right),$$
where we used that {$\frac{ d' L^2 }{3 \rho^2 t_N}\le
 L^2 \left( \frac{d' }{2 \rho^2}+ \frac{8 \tilde{L}^2}{\rho^4}\right)$} since $t_N \ge 1$.
 Collecting the whole bounds and replacing the one of $i)$ by the one we just obtained, we deduce that:
\begin{align*}
&\E\Big|\frac{1}{t_N} \sum_{k=1}^N \gamma_k {f}(\bX_{t_{k-1}}) -\pi(f)\Big|^2\le \frac{\frac{d'}{\rho^2}+\frac{8d}{\rho^3 t_N}}{t_N}+
{\mathfrak{b}_2}\gamma^2 \\
&\textnormal{where} \quad\mathfrak{b}_2 \lesssim_{uc} d^2+\frac{L^4 d}{\rho^3}+
L^2 {\left( \frac{d' }{\rho^2}+ \frac{\tilde{L}^2}{\rho^4}\right)}
+
\frac{ \|\vec{\Delta}(\nabdob)\|_{2,\infty}^2 }{\rho^2  } 
+{\left(L^2  + \frac{L^4 }{\rho^2} \right)  |x_0-x^\star|^2} 
,
\end{align*}
%
where we used that $\frac{d\gamma^2}{3 t_N}\le \gamma^2 (L^4 d)/\rho^3$ since $\rho\le L$ and that $L\ge 1$.
For a given $\varepsilon$, we thus fix $\gamma_\varepsilon= \mathfrak{b}_2^{-\frac{1}{2}} \varepsilon$ so that ${\mathfrak{b}_2}\gamma^2\le \varepsilon^2$. Then, we fix  
$ N_\varepsilon=\varepsilon^{-3} \frac{\sqrt{\mathfrak{b}_2} d'}{\rho^2} \vee \gamma_{\varepsilon}^{-1} \vee d \gamma_{\varepsilon}^{-1} \rho^{-1}$, that implies that   $\frac{d'}{\rho^2 t_N}\le \varepsilon^2$. 
It remains to observe that $8 d / (\rho t_N) \leq 8$, the result then follows.
\end{proof}

The purpose of this paragraph is finally to prove the complexity bound stated in the strongly convex situation.
 
\begin{proof}[Proof of Theorem \ref{theo:learning_strong_convex}]
We first consider $(i.a)$.
We shall apply the above results  in our Bayesian framework. Given the set of $n$ observations $(\xi_1,\ldots,\xi_n)$, we have:
$$W_n(\mathbf{\xi^n},x) = \sum_{i=1}^n U(\xi_i,x).
$$
Assuming that $U(\xi,.)$ satisfies $\AL$, the triangle inequality shows that $W_n(\mathbf{\xi^n},.)$ satisfies  $\AnL$ regarless the value of $\mathbf{\xi^n}$.
{In} the meantime, assuming that $U(\xi,.)$ satisfies $\SC$ yields 
$  W_n(\mathbf{\xi^n},.)$ is $n \rho$-strongly convex \textit{i.e.} satisfies $(\mathbf{ SC}_{n\rho})$. The result in then straightforward observing that the condition on $\gamma$ is driven by $\gamma_{\varepsilon_n} \leq \frac{d}{n^2 L^2 |\theta_0-\theta^\star|^2} = \cal{O}_{id}(n^{-2})$ and pluging it in the value of $N_{\varepsilon_n}$.

Concerning $(i.b)$, we shall verify that the constraint on $\gamma$ is brought by $A_{t_N}^{(2,1)}$ that should verify $n d \gamma^2 \leq \varepsilon_n^2$. Then the constraint on $N$ is then deduced in the same way.
\end{proof}

\appendix
\section{Concentration results for posterior consistency rate \label{sec:app:concentration}}
In this Appendix, we provide the technical proofs related to the posterior mean concentration, stated in Section \ref{sec:bayes} that essentially relies on $\UPI$.

 Note that, among other settings,
 Assumption $\UPI$ holds true  for any location model since, in this case, the Poincar\'e constant of each distribution $\pi_{\theta}$ is independent of  $\theta$, as indicated by the next proposition:
\begin{prop}\label{prop:poincare_pi_theta}
If $U(\xi,\theta)=U(\xi-\theta)$, for all $\xi$ and $\theta$, then $C_P(\pi_\theta)=C_P(\pi)$ where $C_P(\pi)$ stands for the Poincar\'e constant related to $\pi\, \propto \, e^{-U}$.
\end{prop} 
We also point out that this assumption may be verified in {a more general setting} using the upper bound of \cite{KLS} (see the statement in Theorem \ref{theo:log-concave_poincare}).
Finally, the Poincar\'e inequality is satisfied as soon as the log-concave distribution has a second order moment. We observe here that $\CPU$ may include a dimensional effect even though it is clear that it is not the case for strongly log-concave probability distributions with the help of the Bakry-Emery result (see \cite{BakryEmery}). If we believe in the \textit{Kannan-Lovász-Simonovits conjecture} \cite{KLS}, then the constant $\CPU$ may be considered in our model as independent from the dimension $d$, which entails a correct minimax dependency of the Bayesian strategy with respect to $d$.

\begin{proof}[Proof of Proposition \ref{prop:poincare_pi_theta}]
We consider $\theta \in \R^d$ and $f\in W^{1,2}(\pi_{\theta})$ where $W^{1,2}$ stands for the usual Sobolev space (also denoted $H^1$).  For any $\epsi> 0$, a density argument  proves that a $f_\epsi \in \mathcal{C}_{K}^\infty(\R^d,\R) $, \textit{e.g.}, a compactly supported and infinitely differentiable function, exists such that $\pi_{\theta}((f-f_\epsi)^2) \leq \epsi^2$ and $\pi_{\theta}(|\nabla f-\nabla f_\epsi|^2) \leq \epsi^2$ (see Theorem 9.2 of \cite{MR2759829}).

 We shall remark that if $f^{-\theta}_\epsi: x \longmapsto f_{\epsi}(x+\theta)$, then
$\pi_{\theta}(f_\epsi) = \pi(f^{-\theta}_\epsi).$
The function  $f^{-\theta}_{\epsi}$ is infinitely differentiable and compactly supported, we shall apply the Poincaré inequality with the measure $\pi$:
\begin{equation}\label{eq:poincare_pi}
Var_{\pi}(f^{-\theta}_{\epsi}) \leq C_P(\pi) \pi(|\nabla f^{-\theta}_{\epsi}|^2).
\end{equation}
Now,   a straigthforward change of variable yields:
$$
Var_{\pi_{\theta}}(f_\epsi) = Var_{\pi}(f^{-\theta}_\epsi) \quad \text{ and } \quad  \pi(|\nabla f^{-\theta}_\epsi|^2) = \pi_{\theta}(|\nabla f_\epsi|^2).
$$
We then deduce from the previous equalities and from \eqref{eq:poincare_pi} that
\begin{equation}\label{eq:poincare_feps}
Var_{\pi_{\theta}}(f_\epsi)  \leq  C_P(\pi)  \pi_{\theta}(|\nabla f_\epsi|^2).
\end{equation}
Now, we end the proof with a density argument: the Cauchy-Schwarz inequality shows that
$$
|\pi_{\theta}(f)-\pi_{\theta}(f_\epsi)|\leq \sqrt{\pi_{\theta}((f-f_\epsi)^2)} \leq \epsi, \quad |\pi_{\theta}(f^2)-\pi_{\theta}(f_\epsi^2)| \leq 2 [\pi_{\theta}(f^2)+\pi_{\theta}(f_\epsi^2)] \epsi.
$$
Finally, we can prove that
$
|Var_{\pi_{\theta}}(f_\epsi) - Var_{\pi_{\theta}}(f)| \leq 5 \epsi [\pi_{\theta}(f^2)+ \pi_{\theta}(f_\epsi^2)],
$
and in the meantime $|\pi_{\theta}(|\nabla f_\epsi|^2)-\pi_{\theta}(|\nabla f|^2)| \leq \epsi \sqrt{2 \pi_{\theta}(|\nabla f|^2 + |\nabla f_\epsi|^2})$. We use these last upper bounds in \eqref{eq:poincare_feps}, we obtain since $\epsi$ may be chosen arbitrarily small, that:
$$
Var_{\pi_{\theta}}(f) \leq  C_P(\pi) \pi_{\theta}(|\nabla f|^2),
$$
which ends the proof of the proposition.
\end{proof} 

\begin{proof}[Proof of Proposition \ref{prop:BK}]
The proof is straightforward as soon as we remark that Assumption $\UPI$  implies that each $\pi_\theta$ satisfies a Poincaré inequality with constant $\CPU$. 
Since $f$ satisfies $\|\nabla f\|_{\infty} \leq k$ and for any $\theta$, $U_\theta$ is a convex coercive function, then $U_\theta(x)$ has a linear growth for large values of $x$. Hence, $f \in L^2(\pi_\theta)$ and we 
then  simply apply 
the concentration inequality stated in Corollary 3.2 of \cite{BL}. 
This ends the proof of the proposition.
\end{proof}

\begin{proof}[Proof of Corollary \ref{lem:BK}]
The proof of $i)$ is a straightforward application of Proposition \ref{prop:BK} with $f = \Psi$, which is a $1$-Lipschitz function.  
The proof of $ii)$ is similar. {Since $Z_\theta=1$ for any $\theta$, } a direct integration yields:
$${\mathbb{E}}_{\theta}[ \nabla_{\theta} U(\xi,\theta)] = \int_{\R^d} \nabla_{\theta} U(\xi,\theta) e^{-U(\xi,\theta)} \text{d}\xi=\nabla_{\theta} \int_{\R^d}  e^{-U(\xi,\theta)} \text{d}\xi =\nabla_{\theta} Z_\theta= 0.
$$
We then use a union bound deduced by the triangle inequality: if $Z_i = \nabla_{\theta} U(\xi_i,\theta)$, then:
 \begin{align}
\left\{ \left|\frac{1}{n} \displaystyle\sum_{i=1}^n Z_i\right|  \geq \delta \right\} 
& \subset \bigcup_{j=1}^d  \left\{  \frac{1}{n} \left| \sum_{i=1}^n Z^j_i \right|  \geq \frac{\delta}{\sqrt{d}} \right\}. \label{eq:union}
\end{align}
We now apply Proposition \ref{prop:BK} to each term in the union bound and we get that:
$$
\mathbb{P}_{\theta} \left( \left| \frac{1}{n}  \sum_{i=1}^n Z_i^j \right| \geq \frac{\delta}{\sqrt{d}} \right) \leq 2 e^{- n \frac{\delta^2}{4d L^2 \CPU} \wedge \frac{\delta}{2 L \sqrt{d \CPU}}}.
$$
The bound being independent of $j$, the result {follows} by summing over $j$ in  \eqref{eq:union}
\end{proof}

Finally, we derive the proof of the upper bounds related to the first and second type error of the family of tests $(\phi_n^r)$.

\begin{proof}[Proof of Proposition \ref{prop:test}]
The first upper bound $i)$ follows directly from Lemma 3.2 applied with $\theta^{\star}$. For the second estimation, we consider $\theta$ such that $|\theta-\theta^{\star}| \geq \ran$ and we get:
 \begin{align*}
 \lefteqn{
 \left\{ 
 \left| \frac{\sum_{i=1}^n \Psi(\xi_i)}{n} -\mathbb{E}_{\theta^{\star}}[\Psi(X)] \right| \leq \frac{c(\ran)}{2}\right\}}\\
 & = 
  \left\{ 
 \left| \frac{\sum_{i=1}^n \Psi(\xi_i)}{n} -\mathbb{E}_{\theta}[\Psi(X)] + \mathbb{E}_{\theta}[\Psi(X)] - \mathbb{E}_{\theta^{\star}}[\Psi(X)] \right| \leq \frac{c(\ran)}{2}\right\} \\
 & \subset \left\{ \left|\mathbb{E}_{\theta}[\Psi(X)] - \mathbb{E}_{\theta^{\star}}[\Psi(X)] \right|
- \left| \frac{\sum_{i=1}^n \Psi(\xi_i)}{n} -\mathbb{E}_{\theta}[\Psi(X)] \right|
 \leq \frac{c(\ran)}{2}\right\}\\
 & \subset  \left\{\left| \frac{\sum_{i=1}^n \Psi(\xi_i)}{n} -\mathbb{E}_{\theta}[\Psi(X)] \right|
 \geq c(|\theta-\theta^{\star}|) - \frac{c(\ran)}{2}\right\}\\
 & \subset  \left\{\left| \frac{\sum_{i=1}^n \xi_i}{n} -\mathbb{E}_{\theta}[X] \right|
 \geq   \frac{c(\ran)}{2}\right\}. 
 \end{align*}
 In the previous lines, we used the triangle inequality $|a+b| \geq |a|-|b|$ in the third line, the identifiability property $\IWC$ and the fact that $c$ is an increasing map.
 Applying again $i)$, Corollary 3.2, we obtain an upper bound of the probability of deviations uniform regarding the condition $|\theta -\theta^{\star}| \geq \ran$. Taking the supremum over $\theta$, we then obtain the proof of $ii)$.
\end{proof}

\section{Discretization tools - Weakly convex case \label{app:disc}}

In the first part of this section, we prove some bounds on the moments of the Euler scheme. More precisely, in this weakly convex setting, we first establish some  bounds on the exponential moments and then derive some bounds on the moments by Jensen-type arguments. Note that this part is unfortunately highly technical (and much more involved than in the strongly convex case).\\

\begin{rmq} In the weakly convex setting, it is usual to first consider exponential moments. In some sense, introducing exponential Lyapunov functions is a way to compensate the lack of mean-reverting in this case.
\end{rmq}

\noindent For the proofs, we introduce the following notations.  For any positive $c$ and any $\gamma_0>0$, let  ${\cal C}^K_{c,\gamma_0}$ be defined by:
\begin{equation}\label{eq:compactc}
\begin{split}
&{\cal C}^K_{c,\gamma_0}:=
\{x\in\mathbb{R}^d,  |\nabdob (x)|^2-c d\bar{\lambda}_{{\rm loc}} (\gamma_0,x,K) \le 1\} \quad \text{with,}\\
  \; &\bar{\lambda}_{{\rm loc}} (\gamma_0,x,K):=\sup_{ |u-x| < 
  \gamma_0|\nabdob(x)|+\sqrt{K}} \bar{\lambda}_{\hesdob}(u).
  \end{split}
\end{equation}
We observe that   $\Hcu$ combined with Lemma \ref{lem:gradW} entails the compactness of ${\cal C}^K_{c,\gamma_0}$ for any positive $c$ and $\gamma_0$.
 For every positive $\aaa$, we set:
\begin{equation}\label{def:betalambda}
\beta(\aaa,c,\gamma_0,K):=\sup_{x\in {\cal C}^K_{c,\gamma_0}} e^{\aaa \dob(x)} (1+c d L). 
\end{equation}
In what follows, the size of $\beta(\aaa,c,\gamma_0,K)$  will be of first importance. This is why before going further, we study these quantities under assumption $\Hcu$.
\subsection{Preliminary bounds under $\Hcu$}
\begin{lem}\label{lem:gradW} Assume $\Hcu$. Then, 
$$ \frac{\deux}{1-r}\left( \dob^{1-r}(x)-\dob^{1-r}(x^\star)\right)\le   |\nabla \dob(x)|^2\le \frac{\troiss}{1-q}\left( \dob^{1-q}(x)-\dob^{1-q}(x^\star)\right).$$
Furthermore, 
\begin{equation}\label{eq:asympdob}
\dob^{1+r}(x)-\dob^{1+r}(x^\star)\ge (1+r)\deux|x-x^\star|^2\;\textnormal{and}\;\dob^{1+q}(x)-\dob^{1+q}(x^\star)\le \frac{\troiss(1+q)}{1-q}|x-x^\star|^2.
\end{equation}
\end{lem}
\begin{proof} Since $\nabla \dob(x^\star)=0$ and since $\dob$ is convex, the differential equation $\dot{x}(t)=-\nabla \dob(x(t))$ with $x(0)=x$ satisfies:
$\lim_{t\rightarrow+\infty} x(t)=x^\star$.  Thus, setting $\Upsilon_1(t)= |\nabla \dob(x(t))|^2$, we have
$$  |\nabla \dob(x)|^2=-\int_0^{+\infty} \Upsilon_1'(s) ds=\int_0^{+\infty} \langle D^2 \dob(x(s)) \nabla \dob(x(s)), \nabla \dob(x(s))\rangle ds.$$
Thus, using Assumption $\Hcu$, 
$$ \deux \int_0^{+\infty} \dob^{-r}(x(s))|\nabla \dob(x(s))|^2 ds\le   |\nabla \dob(x)|^2\le \troiss \int_0^{+\infty} \dob^{-q}(x(s))|\nabla \dob(x(s))|^2 ds.$$
Setting $\Upsilon_2(t)= \frac{1}{1-\rho}\dob^{1-\rho}(x(t))$ with $\rho\in[0,1)$, we remark that $\Upsilon'_2(t)= -\dob^{-\rho}(x(t))|\nabla \dob(x(t))|^2$ so that
(with $\rho=r$ and $\rho=q$), we get
$$ \frac{\deux}{1-r}\left( \dob^{1-r}(x)-\dob^{1-r}(x^\star)\right)\le   |\nabla \dob(x)|^2\le \frac{\troiss}{1-r}\left( \dob^{1-q}(x)-\dob^{1-q}(x^\star)\right).$$
Finally, in order to prove \eqref{eq:asympdob}, one checks that
$$\underline{\lambda}_{D^2 \dob^{1+r}}(x)\ge (1+r) \dob^r(x)\underline{\lambda}_{D^2 \dob}(x)\le (1+r)\deux$$
and that
$$\bar{\lambda}_{D^2 \dob^{1+q}}(x)\le (1+q)\dob^q(x)\bar{\lambda}_{D^2 \dob}(x)+q(1+q)\dob^{q-1}(x)|\nabdob(x)|^2\le \troiss(1+q)\left(1+\frac{q}{1-q}\right), $$
which in turn easily implies \eqref{eq:asympdob}. 
\end{proof}
\begin{lem}\label{lem:gradW2} Assume $\Hcu$. Assume that $L\gamma_0\le 1/4$. For any $c>0$ and $K>0$, ${\cal C}^K_{c,\gamma_0}$
is a compact set.  Furthermore, setting 
$$\cpar(c,K)=(4 c\troiss \deux^{-1})^{\frac{1}{1+q-r}}+d^{-{\frac{1}{1+q-r}}}\left(
\left(\frac{\troiss}{1-r}\right)^{\frac{1}{1+q}}\vee 4 L\sqrt{K}\vee \left(\frac{2(1-r)}{\deux}\right)^{\frac{1}{1-r}}
\vee(4^{\frac{1}{1-r}}\dob(x^\star))\right),$$
we have
 $$\beta(\aaa,c,\gamma_0,K) {\le} (1+cd L) 
e^{\aaa  \cpar(c,K) d^{\frac{1}{1+q-r}}},$$
where
$\beta(\aaa,c,\gamma_0,K)$ is defined by \eqref{def:betalambda}.
\end{lem}
{
\begin{rmq} Keeping all the constants is unfortunately necessary for the objectives of this paper (but sometimes hard to read). But, if for a moment, we forget the constants $L$, $K$, $c$,  $\deux$ and $\troiss$, the result reads $\beta(\aaa,c,\gamma_0,K) \le C_1 d e^{ C_2 d^{\frac{1}{1+q-r}}}$ where $C_1$ and $C_2$ are some positive constants which do not depend on $d$.
\end{rmq} 
}
\begin{proof} By Lemma \ref{lem:gradW} and $\Hcu$, 
\begin{equation}
\label{eq:minoration_croissance} |\nabdob (x)|^2-c d\bar{\lambda}_{{\rm loc}} (\gamma_0,x)\ge \frac{\deux}{1-r}\left( \dob^{1-r}(x)-\dob^{1-r}(x^\star)\right)- c\troiss d \sup_{ |u-x| < 
  \gamma_0|\nabdob(x)|+\sqrt{K}} \dob^{-q}(u).
  \end{equation}
Since $\dob$ is $L$-Lipschitz, for any $u\in B(x, \gamma_0|\nabdob(x)|+\sqrt{K})$,
\begin{equation}\label{mino:Llipschitz}
\dob(u)\ge  \dob(x)-L(\gamma_0|\nabdob(x)|+\sqrt{K})
\end{equation}
so that when the right-hand member is positive,
$$\dob^{-q}(u)\le \dob(x)^{-q} \left(1-\frac{L(\gamma_0|\nabdob(x)|+\sqrt{K})}{\dob(x)}\right)^{-q} .$$
By Lemma \ref{lem:gradW}, $ |\nabdob(x)| \dob^{-1}(x)\le \sqrt{\frac{\troiss}{1-r}} \dob^{-\frac{q+1}{2}}(x).$ Hence, 
if 
\begin{equation}\label{eq:cond1compact}
\dob(x)\ge \left( 4 L\sqrt{K}\right)\vee \left((4L\gamma_0)^{\frac{2}{1+q}} \left(\frac{\troiss}{1-r}\right)^{\frac{1}{1+q}}\right),
\end{equation}
then, $\frac{L(\gamma_0|\nabdob(x)|+\sqrt{K})}{\dob(x)}\le 1/2$ and thus,  $(1-\frac{L(\gamma_0|\nabdob(x)|+\sqrt{K})}{\dob(x)})^{-q} \ge {2^{q}}$. 
Thus, under \eqref{eq:cond1compact},
$$\sup_{ |u-x| < 
  \gamma_0|\nabdob(x)|+\sqrt{K}} \dob^{-q}(u)\le 2^q \dob(x)^{-q}.$$
We deduce that under \eqref{eq:cond1compact},
Equation \eqref{eq:minoration_croissance} yields:
\begin{align*}
 |\nabdob (x)|^2-c d\bar{\lambda}_{{\rm loc}} (\gamma_0,x,K)&\ge 
 \frac{\deux}{1-r}\left( \dob^{1-r}(x)-\dob^{1-r}(x^\star)\right)- c d 2^q \troiss W^{-q}(x) \\
 & \ge \frac{\deux}{1-r} W^{1-r}(x) \left[ 1-\frac{W^{1-r}(x^\star)}{W^{1-r}(x)} -  \frac{2^q c d \troiss}{\deux W^{1+q-r}(x)} \right].
\end{align*}
We now fix some conditions on $\dob(x)$ in order that 
$ 1-\frac{W^{1-r}(x^\star)}{W^{1-r}(x)} -  \frac{2^q c d \troiss}{\deux W^{1+q-r}(x)} \ge \frac{1}{2}$ and $\frac{\deux}{1-r} W^{1-r}(x)\ge 2$.
We thus assume that
\begin{equation}\label{eq:cond2compact}
\dob(x)\ge \left(\frac{2(1-r)}{\deux}\right)^{\frac{1}{1-r}}\vee (4^{\frac{1}{1-r}}\dob(x^\star))\vee (2^{q+2} c\troiss \deux^{-1} d)^{\frac{1}{1+q-r}}).
\end{equation}
Finally, if \eqref{eq:cond1compact} and \eqref{eq:cond2compact} are satisfied, then 
 \begin{equation*}
 |\nabdob (x)|^2-c d\bar{\lambda}_{{\rm loc}} (\gamma_0,x,K)\ge 1.
 \end{equation*}
 But the definition of $\cpar$, we remark that \eqref{eq:cond1compact} and \eqref{eq:cond2compact} are satisfied if 
 $\dob(x)\ge \cpar(c,K)d^{\frac{1}{1+q-r}}$ so that 
$${\cal C}^K_{c,\gamma_0}\subset\{x \in\ER^d, \dob(x)\le  \cpar(c,K)d^{\frac{1}{1+q-r}}\}.$$
The bound on $\beta(\lambda,c,\gamma_0,K)$   easily follows.

\end{proof}

%
%

\subsection{Exponential bounds for the continuous-time process}
\begin{lem}\label{lem:expbounds1cont} { Assume that ${\cal C}^0_{c,0}$ defined by \eqref{eq:compactc} is a compact set.Then, for any $\aaa\in(0,1)$,
$$ \E_x [e^{\aaa \dob(X_t)}]\le e^{\aaa \dob (x)} e^{- \tilde{\aaa} t}+\beta(\aaa, (1-\aaa)^{-1},0,0),$$
where $\tilde{\aaa}=\aaa(1-\aaa)^{-1}$ and $\beta(\aaa, (1-\aaa)^{-1},0,0)$ is defined by \eqref{def:betalambda}.
In particular, under $\Hcu$,}
$$\sup_{t\ge0} \E_x [e^{\aaa \dob(X_t)}]\le e^{\aaa \dob (x)} + (1+(1-\aaa)^{-1}d L) 
e^{  \cpar(1,0) d^{\frac{1}{1+q-r}}}.$$
\end{lem}
\begin{proof}
We apply the Ito formula to $f_\aaa$ defined by $f_\aaa(x)=e^{\aaa \dob(x)}$ 
 ($\aaa$ has to be chosen):
 $$ f_\aaa(X_t)= f_\aaa(x)+ \int_0^t \mathcal{L} f_\aaa(X_s) ds+2 \underbrace{\int_0^t \langle \nabla f_\aaa(X_s), dB_s}_{M_t}\rangle.$$
 Then, one can check that
\begin{align*}
 \frac{ \mathcal{L} f_\aaa(x)}{f_\aaa(x)}&=-\aaa |\nabdob(x)|^2+ \aaa^2 |\nabdob(x)|^2+\aaa \Delta \dob(x)\\
 &\le  {\aaa}{(1-\aaa)}\left(- |\nabdob(x)|^2+ d(1-\aaa)^{-1} \bar{\lambda}_{\hesdob}(x)\right).
 \end{align*}
 Setting $\tilde{\aaa}=\aaa(1-\aaa)^{-1}$ and $c=(1-\aaa)^{-1}$, we deduce from $\Hcu$ and $\AL$ that:
 \begin{align}
 \mathcal{L} f_\aaa(x)& \le  - \tilde{\aaa} f_\aaa(x) 1_{\{x \in \{{\cal C}^K_{c,\gamma_0}\}^c\}}+\tilde{\aaa} f_\aaa(x)\left(- |\nabdob(x)|^2+ d(1-\aaa) \bar{\lambda}_{\hesdob}(x)\right) 1_{\{x \in {\cal C}^K_{c,\gamma_0}\}}\label{eq:lfpluspetitquef}\\
 &\le - \tilde{\aaa} f_\aaa(x)+\tilde{\aaa} f_\aaa(x)\left(1+ d (1-\aaa)^{-1} \bar{\lambda}_{\hesdob}(x) \right) 1_{\{x \in {\cal C}^K_{c,\gamma_0}\}}\\
 &\le  -\tilde{\aaa} f_\aaa(x)+\tilde{\aaa} \beta(\aaa,  (1-\aaa)^{-1} ,0,0).\nonumber
 \end{align}
It follows that $(M_t)_{t\ge0}$ is a true martingale and the Gronwall lemma leads to:
$$ \E_x[f_\aaa(X_t)]\le f_\aaa(x) e^{-\tilde{\aaa} t}+\beta(\aaa,  (1-\aaa)^{-1} ,0,0)  \int_0^s \tilde{\aaa}e^{\tilde{\aaa}(s-t)} ds\le f_\aaa(x)e^{-\tilde{\aaa} t}+{\beta(\aaa,  (1-\aaa)^{-1} ,0,0)}.$$
\end{proof}
%
%
%
\subsection{Exponential bounds for the continuous-time Euler scheme (Proposition \ref{lem:expbounds22})\label{sec:expo_continuous}}

\begin{rmq} The proof of this result is rather technical and the important thing is to pay a specific attention to the dependency  with respect to $d$ and $L$. As indicated in our statement, we are led to choose $\gamma$ {lower than} $(L d)^{-1}$.
\end{rmq}
{
\begin{prop}\label{lem:expbounds2scheme}
Assume $\Hcu$. Assume that  $\gamma\le\gamma_0:= \frac{1}{4 d L + 1}$, then:
For any $\aaa \leq 1/16$, a constant $C_\aaa$ (depending only on $\aaa$) exists such that
$$\sup_{t\ge0} \E_x[e^{ \aaa \dob(\bX_{t})}]\le {
e^{{\aaa} \dob(x)}+C_\aaa (1+5d L) 
e^{\aaa  \cpar(5,32\aaa^{-1}) d^{\frac{1}{1+q-r}}}}.
$$ 
\end{prop}
}
\begin{proof}
 We assume that $\aaa \leq \frac{1}{16}$ and that $\gamma(4dL+1) \leq 1$.
The Taylor formula yields:
\begin{align*}
\dob(\bX_{t_{k+1}})&\le \dob(\bX_{t_k})-\gamma |\nabdob(\bX_{t_k})|^2+\langle \nabdob(\bX_{t_k}), \Delta_{k+1}\rangle\\
&+
\left(\int_0^1\bar{\lambda}_{\hesdob}(\bX_{t_{k}}^{(\theta)})\text{d}\theta\right) (\gamma^2 |\nabdob(\bX_{t_k})|^2+|\Delta_{k+1}|^2),
\end{align*}
where $\Delta_{k+1}=\sqrt{2}(B_{t_{k+1}}-B_{t_k})$ and $\bX_{t_{k}}^{(\theta)}=\bX_{t_k}+\theta(-\gamma \nabdob(\bX_{t_k})+\Delta_{k+1}).$
 Setting $f_\aaa(x)=e^{\aaa \dob(x)}$ and using $\AL$, we deduce that:
\begin{equation}\label{eq:fl345}
\E[f_\aaa(\bX_{t_{k+1}})|{\cal F}_{t_k}]\le f_\aaa(\bX_{t_k}) e^{(-\aaa \gamma+ L \aaa \gamma^2) |\nabdob(\bX_{t_k})|^2}\Psi_\gamma(\bX_{t_k}),
\end{equation}
where 
$$ \Psi_\gamma: x \longmapsto \E \exp\left({\aaa \sqrt{2\gamma} \langle \nabdob(x), Z\rangle+2\aaa\gamma \left(\int_0^1\bar{\lambda}_{\hesdob}(x(\theta,\gamma,Z))\text{d}\theta\right)  |Z|^2}\right),
$$
with ${Z \sim {\cal N}(0,I_d)}$ and $x(\theta,\gamma,z):=x+\theta(-\gamma \nabdob(x)+\sqrt{{2}\gamma} z)$. 
 We decompose $ \Psi_\gamma$ into two parts:
\begin{align*}
\Psi_\gamma(x)&=\underbrace{\E\left[\exp\left(\aaa \sqrt{2\gamma} \langle \nabdob(x), Z\rangle+2\aaa\gamma \left(\int_0^1\bar{\lambda}_{\hesdob}(x(\theta,\gamma,Z))\text{d}\theta\right)  |Z|^2\right)1_{\{|Z|^2\le (2\gamma)^{-1}{K}\}}\right]}_{:=\Psi_\gamma^{(1)}(x)}\\
&+\underbrace{\E\left[\exp\left({\aaa \sqrt{2\gamma} \langle \nabdob(x), Z\rangle+2\aaa\gamma \left(\int_0^1\bar{\lambda}_{\hesdob}(x(\theta,\gamma,Z))\text{d}\theta\right)  |Z|^2}\right)1_{\{|Z|^2> (2\gamma)^{-1}{K}\}}\right]}_{:=\Psi_\gamma^{(2)}(x)},
\end{align*}
where $K>0$ will be chosen later on.\\

\noindent
$\bullet$ \underline{Upper bound of $\Psi_\gamma^{(1)}(x)$.}
If $|Z|^2 \leq  (2\gamma)^{-1}K$,  then $|x(\theta,\gamma,Z)-x| \leq \gamma |\nabla W(x)| + \sqrt{K}$, so that: 
\begin{align*}
\Psi_\gamma^{(1)}(x)&\le \E\left[\exp\left({\sqrt{2\gamma} \aaa \langle \nabdob(x), Z\rangle+ 2 \aaa \gamma \bar{\lambda}_{{\rm loc}} (\gamma,x) |Z|^2}\right)\right]\\
&\le \prod_{i=1}^d \E_{{Z_1 \sim {\cal N}(0,1)}}[e^{\sqrt{2\gamma} \aaa\partial_i\dob(x)Z_1+ 2 \aaa  \gamma \bar{\lambda}_{{\rm loc}} (\gamma,x)  Z_1^2}],
\end{align*}
where we used   $\bar{\lambda}_{{\rm loc}}(\gamma,x)  :=\sup_{u \in B(x,\gamma|\nabdob(x)|+{\sqrt{K}})}\bar{\lambda}_{\hesdob}(u)$. {We choose to alleviate this notation in the sequel of the proof by  writing only $\bar{\lambda}_{{\rm loc}}$ (instead of $\bar{\lambda}_{{\rm loc}}(\gamma,x)$).}

 Below, we will use that:
\begin{equation}\label{eq:formexacte}
\forall \alpha_1 \in \R \quad \forall \alpha_2<1/2 \qquad 
\E_{{Z_1 \sim {\cal N}(0,1)}}[e^{\alpha_1 Z_1+\alpha_2 Z_1^2}]= \frac{1}{\sqrt{1-2\alpha_2} }e^{\frac{\alpha_1^2}{2(1-2\alpha_2)}},
\end{equation}
with $\alpha_1= \sqrt{2\gamma}\aaa\partial_i\dob(x)$ and $\alpha_2=2 \aaa \gamma \bar{\lambda}_{{\rm loc}}  $. {Since $ \bar{\lambda}_{\rm loc}\le L$}, our choice of $\gamma$ and $\aaa$ leads to $\alpha_2 < 1/2$ and we  deduce from Equation \eqref{eq:formexacte} that:
$$ \Psi_\gamma^{(1)}(x)=\left(\frac{1}{1- 4 \aaa \gamma 
{\bar{\lambda}_{{\rm loc}} }}
\right)^{\frac{d}{2}} e^{\frac{\gamma \aaa^2 |\nabdob(x)|^2}{1-4 \aaa\gamma \bar{\lambda}_{{\rm loc}}  }} = \exp \left( -\frac{d}{2} \log (1-4 \aaa \gamma \bar{\lambda}_{{\rm loc}} )
+\frac{\gamma \aaa^2  |\nabdob(x)|^2}{1-4 \aaa \gamma \bar{\lambda}_{{\rm loc}}}\right).
$$

\noindent We observe that $\log(1-u) \ge - 5/4 u$ when $u \in [0,1/2]$ and apply this inequality with $u=4 \aaa \gamma \bar{\lambda}_{{\rm loc}}<1/2$. Using again that $\bar{\lambda}_{{\rm loc}}  \leq L$ in the second term, we obtain:

\begin{equation}\label{eq:psigam111}
 \Psi_\gamma^{(1)}(x) \leq \exp\left(\frac{5}{2}  \aaa \gamma d \bar{\lambda}_{{\rm loc}}   +  \frac{\gamma \aaa^2  |\nabdob(x)|^2}{1-4 \aaa \gamma L}\right).
\end{equation}



$\bullet$ \underline{Upper bound of $\Psi_\gamma^{(2)}(x)$.}
Using the Cauchy-Schwarz inequality and the exponential Markov inequality,
\begin{align*}
\Psi_\gamma^{(2)}(x)&\le \E\left[\exp\left( 2\aaa \sqrt{2\gamma} \langle \nabdob(x), Z\rangle+ 4 \aaa  \gamma L  |Z|^2 \right)\right]^{\frac{1}{2}}\left[\mathbb{P}\left(|Z|^2 \ge \tcb{K} \gamma^{-1}\right)\right]^{\frac{1}{2}}\nonumber\\
& \le 
\exp \left( -\frac{d}{4} \log (1-8 \aaa \gamma L )+
\frac{2 \aaa^2 \gamma |\nabdob(x)|^2}{ 1-8 \aaa \gamma L}\right) 
e^{-\frac{K}{8\gamma}} (\E[e^{\frac{Z_1^2}{4}}])^{\frac{d}{2}}.
 \end{align*}
 By \eqref{eq:formexacte}, we deduce that
 \begin{align*}
\Psi_\gamma^{(2)}(x)&\le \exp \left(\frac{5d}{2} \aaa \gamma  L + \frac{2 \aaa^2 \gamma  |\nabdob(x)|^2}{ 1-8 \aaa \gamma L}\right)
e^{-\frac{K}{8\gamma}+d\frac{\log 2}{4}}\\
 & \leq {\exp \left(  - \frac{{K}}{8 \gamma}\left[1 - 20 \aaa \frac{ \gamma^{2} d L}{{K}} - \frac{\gamma d \log 2 }{2{K}}\right] + \frac{2 \aaa^2 \gamma  |\nabdob(x)|^2}{ 1-8 \aaa \gamma L} \right)}.
\end{align*}

%

%
Checking that $20 \aaa \gamma^2 d L + \gamma d \log 2/2 \leq 10$ and choosing $K$ larger then $20$ yields:
\begin{equation}\label{psi2part1}
\Psi_\gamma^{(2)}(x) \le \exp \left(  - \frac{{K}}{16 \gamma} + \frac{2 \aaa^2 \gamma  |\nabdob(x)|^2}{ 1-8 \aaa \gamma L} \right).
\end{equation}

We then plug \eqref{eq:psigam111} and \eqref{psi2part1} into \eqref{eq:fl345} and obtain that:
\begin{align*}
\E&[f_\aaa(\bX_{t_{k+1}})|{\cal F}_{t_k}]\le \E[f_\aaa(\bX_{t_{k}})] e^{(- \gamma \aaa + L \aaa \gamma^2) |\nabdob(\bX_{t_k})|^2} \\
& \times \left[ 
\exp\left(
\frac{5}{2}  \aaa \gamma d \bar{\lambda}_{{\rm loc}}   +  \frac{\gamma \aaa^2  |\nabdob(x)|^2}{1-4 \aaa \gamma L}\right)+
\exp \left(  -   \frac{{K}}{16 \gamma} + \frac{2 \aaa^2 \gamma  |\nabdob(x)|^2}{ 1-8 \aaa \gamma L}\right)\right] \\
& \leq \E[f_\aaa(\bX_{t_{k}})] 
 e^{- \aaa \gamma  \left(1-   L  \gamma - \frac{2 \aaa}{1-8 \aaa \gamma L}\right) |\nabdob(\bX_{t_k})|^2 } \left( \exp\left( \frac{5}{2} \aaa   \gamma d \bar{\lambda}_{{\rm loc}} (\gamma,\bX_{t_{k}},K)\right) + \exp\left(- \frac{{K}}{16 \gamma} \right)\right) \\
 & \leq \E[f_\aaa(\bX_{t_{k}})]  \left[ \exp\left(- \frac{K}{16 \gamma}\right) + \exp\left( -\frac{ \aaa \gamma}{2} \left[   |\nabdob(\bX_{t_k})|^2 -5  d \bar{\lambda}_{{\rm loc}}(\gamma,\bX_{t_{k}},K) \right] \right)\right],
\end{align*}
since  $L \gamma + \frac{2 \aaa}{1-8 \aaa \gamma L} \leq 1/2$.
Using our assumption $\Hcu$ and the notations introduced in \eqref{eq:compactc}, we know that:
$$- |\nabdob(x)|^2+
{5} d \bar{\lambda}_{\rm loc}(\gamma,x,K))\le -1  \quad \textnormal{when $x \in \{{\cal C}^K_{5,\gamma_0}\}^c$}.
$$
We introduce  $\rho=\rho({\gamma,K})= e^{-\frac{\aaa \gamma}{2}}+e^{-\frac{K}{16 \gamma}}$ and we shall observe that a $K$ large enough exists such that for our choice of $\gamma$, $\rho< 1$. Therefore, we have:
\begin{align*}
\E[f_\aaa(\bX_{t_{k+1}})|{\cal F}_{t_k}] &\le 
f_\aaa(\bX_{t_k})\left( e^{-\frac{a\gamma}{2}}1_{\{\bX_{t_k}\in \{{\cal C}^K_{5,\gamma_0}\}^c\}}+ e^{-\frac{K}{16 \gamma}}{+}
e^{\frac{\aaa \gamma}{2} ( -|\nabdob(\bX_{t_k})|^2+
{5} d L)}
 1_{\{\bX_{t_k}\in {\cal C}^K_{5,\gamma_0}\}}
 \right) \\ 
 & \le \rho  f_\aaa(\bX_{t_k})
 +  f_\aaa(\bX_{t_k})\left[ e^{\frac{5 \aaa \gamma dL}{2}}-e^{-\frac{a\gamma}{2}} \right] 1_{\{\bX_{t_k}\in {\cal C}^K_{5,\gamma_0}\}}\\
 & \le \rho f_\aaa(\bX_{t_k})
 +  e^{-\frac{\gamma \aaa}{2}} f_\aaa(\bX_{t_k})\left[ e^{\frac{\gamma \aaa }{2}(1+5dL)}- 1 \right] 1_{\{\bX_{t_k}\in {\cal C}^K_{5,\gamma_0}\}}.
 \end{align*}

We observe that $\frac{\gamma \aaa }{2}(1+5dL) \leq 1$ and using $e^{x} \le 1+2x$ when $x \in [0,1]$, we obtain that:
$$
\E[f_\aaa(\bX_{t_{k+1}})|{\cal F}_{t_k}] \le  \rho  f_\aaa(\bX_{t_k})+ \gamma \aaa \beta(\aaa,5,\gamma_0,K),
$$
where $\beta(\aaa,c,\gamma_0,K)$ is defined by \eqref{def:betalambda}.
\noindent Thus, setting $v_k=\E[f_\aaa(\bX_{t_k})]$, we obtain that:
$$  \forall k \ge 0 \qquad 
v_{k+1}\le \rho  v_k+ \gamma \aaa \beta(\aaa,5,\gamma_0,K).$$
An induction leads to
$$ \forall k \ge 1 \qquad v_k\le \gamma \bar{\beta}_d \sum_{j=0}^{k-1} \rho^{j} + \rho^k v_0 =  \frac{\gamma \aaa \beta(\aaa,5,\gamma_0,K)}{1-\rho(\gamma,K)} +  v_0.
$$
{We finally have to lower-bound $1-\rho(\gamma,K)$:}{ using that $\exp(-x)\le (2x)^{-1}$ for $x\ge 1$ and $\exp(-x)\le 1-x$ for $x\ge0$, we have $1-\rho(\gamma,K)\le \gamma(\frac{\aaa}{2}-\frac{8}{K})\ge \gamma\frac{\aaa}{4}$ by taking $K\ge 32 \aaa^{-1}$}.
 We finally deduce that 
$$\sup_{k\ge0} \E_x[e^{{\aaa} \dob(\bX_{t_k})}]\le  e^{{\aaa} \dob(x)}+C_\aaa \beta(\aaa,5,\gamma_0,K){\le
e^{{\aaa} \dob(x)}+C_\aaa (1+5d L) 
e^{\aaa  \cpar(5,32\aaa^{-1}) d^{\frac{1}{1+q-r}}}}
,$$
{where $C_\aaa$ does not only depend on $\aaa$ and the last inequality follows from Lemma \ref{lem:gradW2}}. To extend to any time $t\ge0$, it is enough to write for any $t\in[t_k,t_{k+1}]:$
$$\E[f_\aaa(\bar{X}_t)]=\E[\E[f_\aaa(\bX_t)|{\cal F}_{t_k}]],$$
and then to adapt the beginning of the proof. The details are left to the reader.
\end{proof}

We conclude this section by a useful technical result for our purpose, which is stated as $ii)$ of Proposition \ref{lem:expbounds22} in the main part of the paper.

In what follows, we assume that $\aaa \leq \frac{1}{16}$ and that $\gamma$ satisfies $\gamma(4dL+1) \leq 1$.

\begin{prop}\label{cor:momentsEUler}
Under the Assumptions of Proposition \ref{lem:expbounds22}, assume that  $\gamma\le\gamma_0:= \frac{1}{4 dL+1}$.
Let $p>0$. Then, 
$$\sup_{t\ge0} \E_x[\dob ^p (X_{t})]+\sup_{t\ge0} \E_x[\dob ^p (\bX_{t})]\le c_p \left(  \dob^p (x)+ \Upsilon^p\right),$$
where $p\mapsto c_p$ is a locally finite positive function on $[0,+\infty)$ and where,
$$\Upsilon=c_r {(\troiss\vee L)^{\frac{1}{1+q-r}}}{\deux^{-\frac{1}{1-r}}}\log(1+dL) d^{\frac{1}{1+q-r}},$$ 
with $c_r$ depending only on $r$. In particular, for any $p\le 9$,
\begin{equation}\label{eq:universalboundmoment}
\sup_{t\ge0} \E_x[\dob ^p (X_{t})]+\sup_{t\ge0} \E_x[\dob ^p (\bX_{t})]\le  C( \dob^p (x)+ \Upsilon^p),
\end{equation}
where $C$ is a universal constant.
\end{prop}
\begin{rmq} Note that this property will play a fundamental in the proof of Proposition \ref{prop:fundamentalbounds2}. The second part is only a way to recall to the reader that the dependence in $p$ can be omitted in the proofs since we use this property for some values of $p$ which are always bounded by $9$ (More precisely, the ``worst'' value of $p$ in the proof is $p=8r(1+\mte)\vee (1+3r)(1+\mte)$ where $\mte$ is an arbitrary small positive number).
\end{rmq}
\begin{proof}[Proof of Proposition \ref{cor:momentsEUler}]
 Let us first consider the bound on the diffusion process. Owing to Jensen inequality and to the elementary inequality $(a+b)^p\le a^p+b^p$ for $p\le 1$ and positive $a$ and $b$, we can only consider the case $p\ge1$. {Let $\aaa\in(0, 1)$ and write}
$$\E_x [\dob^{p} (X_t)]=\aaa^{-p}\E_x[\log^p(e^{\aaa\dob(X_t)})]\le \aaa^{-p}\E_x[\log^p(e^{p-1+\aaa\dob(X_t)})]$$
Since $x\mapsto \log^p x$ is concave on $[e^{p-1},+\infty)$ and that by construction, $e^{p-1+\aaa\dob(X_t)}\ge e^{p-1}$, we deduce from Jensen inequality that
\begin{align*}
\E_x [\dob^{p} (X_t)]&\le \aaa^{-p} \left( (p-1)+ \log\E_x[e^{\aaa\dob(X_t)}]\right)^p\\
&\le  \aaa^{-p}\left( (p-1)+\log\left(e^{\aaa \dob (x)} + (1+(1-\aaa)^{-1}d L) 
e^{  \cpar(1,0) d^{\frac{1}{1+q-r}}}\right)\right)^p,
\end{align*}
where in the last line, we used Lemma \ref{lem:expbounds1cont}. For $a>0$, $\rho\ge 1$ and $b\ge0$, one easily checks that
$\log(e^a+\rho e^b)\le a+ b+\log(2\rho)$. Thus, 
\begin{align*}
\E_x [\dob^{p} (X_t)]&\le \aaa^{-p}\left( p-1+{\aaa \dob (x)} + \log(2(1+(1-\aaa)^{-1}d L)) 
+{  \cpar(1,0) d^{\frac{1}{1+q-r}}}\right)^p\\
&\le (3\aaa^{-1})^p\left( (p-1)^p+{\aaa^p \dob^p (x)} + \log^p(2(1+(1-\aaa)^{-1}d L)) 
+{  \cpar(1,0)^p d^{\frac{p}{1+q-r}}}\right).
\end{align*}
Taking $\aaa=1/2$ (for instance), we get
\begin{align*}
\E_x [\dob^{p} (X_t)]&\le c_p \left(\dob^p (x)+(\log(1+dL)\cpar(1,0))^p d^{\frac{p}{1+q-r}}\right).
\end{align*}
To deduce the result, one finally checks that $(\log(1+dL)\cpar(1,0)\lesssim_{uc}  \dw $.\\

\noindent 
A similar strategy based on Proposition \ref{lem:expbounds2scheme} leads to
\begin{align*}
\E_x [\dob^{p} (X_t)]&\le c_{p,\aaa} \left(\dob^p (x)+ (\log(1+5d L) 
  \cpar(5,32\aaa^{-1}))^p d^{\frac{p}{1+q-r}}\right).
\end{align*}
for any $\aaa\in(0,1/16)$. Taking $\aaa=1/16$ and checking that $\log(1+5d L) 
  \cpar(5,32\aaa^{-1})\lesssim_{uc} \dw$ leads to the result.

\end{proof}

\subsection{Analysis of the first and second variation processes\label{sec:tangent_technique}}
We introduce the \textit{first variation process} $Y^x=(Y^{x,ij})_{1\le i,j\le d}$ defined for all $(i,j)\in\{1,\ldots,d\}^2$ by
$Y_s^{x,ij}=\partial_{x_j} \{X_s^x\}^{i}$ where $\{X_s^x\}^{i}$ denotes the $i^{th}$ component of $X_s^x$. The process $Y^x$ is thus a matrix-valued process solution of  the ordinary differential equation: 
\begin{equation}\label{eq:Yt}
Y_0^{x} =I_d \qquad \text{and} \qquad \frac{d Y_s^{x}}{ds} = - \hesdob(X_t^x) Y_t^{x}.
\end{equation}



\begin{lem} \label{lem:pathconttang}
(i) { $\forall x\in\ER^d$,
$$ \nsp{Y^{x}_t}^2\le e^{-2\int_0^t \underline{\lambda}_{\hesdob}(X_s^x) {\rm{d}}s}.$$
}

(ii) { Assume that $\hesdob$ is $\tilde{L}$-Lipschitz for the norm $\nsp{.}$. Then, for any $x,y\in\ER^d$, 
$$\nsp{Y_t^y-Y_t^x}^2\le\tilde{L}^2 |x-y|^2
\int_0^t \frac{1}{\un{\lambda}_{\hesdob}(X_s^y)}e^{-\int_s^t \underline{\lambda}_{\hesdob}(X_u^y) {\rm{d}}u}e^{-2\int_0^s \underline{\lambda}_{\hesdob}(X_u^x) {\rm{d}}u} ds.
$$
}
%
\end{lem}
\begin{proof}
$(i)$ { By  \eqref{eq:Yt}, 
$$Y^{x}_t z=z-\int_0^t \hesdob(X_s) Y_s^{x}z  ds.$$
Thus, 
$$ |Y^{x}_t z|^2\le |z|^2- 2\int_0^t \underline{\lambda}_{\hesdob}(X_s^x) |Y^{x}_s z|^2 ds.$$
By a Gronwall-type argument, we deduce that
$$ |Y^{x}_t z|^2\le |z|^2 e^{-2\int_0^t \underline{\lambda}_{\hesdob}(X_s^x) {\rm{d}}s}$$
and the result follows.
}

\noindent  (ii) 
{For any $x,y\in\ER^d$,
\begin{align*}
 Y_t^{y}-Y_t^{x}&=\int_0^t -\hesdob(X_s^y)Y_s^{y}+\hesdob(X_s^x)Y_s^{x} ds\\
 &=\int_0^t -\hesdob(X_s^y)(Y_s^{y}-Y_s^x)+(\hesdob(X_s^x)-\hesdob(X_s^y))Y_s^{x} ds.
\end{align*}
}
{Setting $V_t=Y_t^{y}-Y_t^{x}$, we have for any $z\in\ER^d$ such that $|z|=1$,
\begin{align*}
\frac{d}{dt} |V_t z|^2&=-2 \langle V_t z, \hesdob(X_t^y) V_t z\rangle +2\langle V_t z,  (\hesdob(X_t^x)-\hesdob(X_t^y))Y_t^{x}z\rangle\\
&\le -2\un{\lambda}_{\hesdob}(X_t^y) |V_t z|^2+ 2| V_t z|\times \tilde{L} |X_t^x-X_t^y| e^{-\int_0^t \underline{\lambda}_{\hesdob}(X_s^x) {\rm{d}}s},
\end{align*}
where $\tilde{L}$ denotes the Lipschitz constant of $\nabla^2 W$ for the norm $\nsp{.}$. 
The inequality $2|ab|\le \lambda |a|^2+\lambda^{-1}|b|^2$ (for $\lambda>0$) then yields:
\begin{align*}
\frac{d}{dt} |V_t z|^2\le -\un{\lambda}_{\hesdob}(X_t^y) |V_t z|^2+ \frac{1}{\un{\lambda}_{\hesdob}(X_t^y)}\tilde{L}^2 |z|^2 |X_t^x-X_t^y|^2 e^{-2\int_0^t \underline{\lambda}_{\hesdob}(X_s^x) {\rm{d}}s}.
\end{align*}
Thus, by a Gronwall-type argument,  we deduce that
\begin{align*}
|V_t z|^2&\le \tilde{L}^2 |z|^2 \int_0^t \frac{|X_s^x-X_s^y|^2}{\un{\lambda}_{\hesdob}(X_s^y)}e^{-\int_s^t \underline{\lambda}_{\hesdob}(X_u^y) {\rm{d}}u}e^{-2\int_0^s \underline{\lambda}_{\hesdob}(X_u^x) {\rm{d}}u} ds.
\end{align*}
The result follows by using that $ |X_s^x-X_s^y|\le \sup_{u\in\ER^d}\nsp{Y_s^{u}}|y-x|\le |y-x|$ by $(i)$. 
}
\end{proof}
\subsection{ Solution of Poisson equation. Bounds on  the solution of the  Poisson equation and its derivatives \label{sec:tangent_technique}}
We remind that $g$ is solution of the Poisson equation {$f-\pi(f)={\cal L}g$ where $f$ is an at least ${\cal C}^2$-function from $\ER^d$ to $\ER^{d'}$ }and that, $x^\star$ is the unique minimizer of $W$.
\begin{proof}[Proof of Proposition \ref{theo:poisson}]
\underline{Uniqueness:} Consider two ${\cal C}^2$ solutions $g_1$ and $g_2$. Then, $\LL(g_1-g_2)=0$ and:
$$
\int (g_1-g_2) \LL(g_1-g_2) \text{d}\pi= - \int |\nabla(g_1-g_2)|^2 \text{d} \pi.
$$
Since the operator $\LL $ is elliptic, we know  that the density of $\pi$ is $a.s.$ positive so that $g_1-g_2$ is constant. The constraint $\pi(g_1)=\pi(g_2)=0$ implies that $g_1=g_2$.

\noindent
\underline{Existence:}
{ Let $g_t(x)=\int_0^t \nu(f) - P_s f(x) \text{d}s$. 
Following the arguments of Proposition \ref{prop:fundamentalbounds2} and its proof  below (mainly the fact that the first and second variation processes go to $0$ in $L^1$, {sufficiently fast and locally uniformly in $x$}), $g$ is well-defined, of class ${\cal C}^2$ and, 
 $(g_t)$, $D g_t$ and $D^2 g_t$ converge locally uniformly to $g$, $Dg$ and $D^2g$ respectively. In particular, $\LL g=\lim_{t\rightarrow+\infty} \LL g_t$. Now,
 using that $\LL$ is a linear operator (null on constant functions) and the Dynkin formula, we get:
$$
\LL g_t(x) =\int_{0}^t \LL (\nu(f)-P_s f)(x) \text{d}s = P_0 f(x)-P_t f(x) =  f(x) - P_t(f)(x)\xrightarrow{t\rightarrow+\infty} f(x)-\pi(f).
$$
Then,  $\LL g(x)=\lim_{t\rightarrow+\infty} \LL g_t(x)=f(x)-\pi(f)$ for every $x\in\ER^d$ (see Proposition A.8 of \cite{Pages_Panloup} for a similar but more detailed proof).}

\end{proof}


\begin{proof}[Proof of Proposition \ref{prop:fundamentalbounds2} {of the main document}]
The fact that $g$ is ${\cal C}^2$ is proved along the proof.
\noindent \underline{Proof of $i)$.}
If the conditions of the Lebesgue differentiability are met (checked later on), then:
$$ Dg(x)=\int_0^{+\infty}\E_x[ Df(X_t) Y_t] \text{d}t.$$
Thus, since $\nsp{.}$ is a norm and since $\sup_{x\in\ER^d} \nsp{Df(x)}=\flip$,
\begin{equation}\label{eq:controleDflipp}
 \nsp{Dg(x)}\le \flip \int_0^{+\infty} \E_x[\nsp{Y_t}] dt.
 \end{equation}
By Lemma \ref{lem:pathconttang}, for all $z\in\ER^d$, 
$$\E_x[|Y_t z|^2]\le |z|^2 \E_x[e^{-2\int_0^t \underline{\lambda}_{\hesdob}(X_s) \text{d}s}].$$
Thus, for any $t\ge1$, for every positive $\delta_1$, we have
\begin{align} \label{eq:trozpaze11}
\E_x|Y_t z|^2]\le |z|^2 e^{-2 t^{\delta_1}}+|z|^2\PE\left(\int_0^{t} \underline{\lambda}_{\hesdob}(X_u) \text{d}u\le t^{\delta_1}\right).
\end{align}
Then, using $\Hcu$ and the Markov inequality, we get for every positive $\mte$, for every $z\neq0$,
\begin{align} 
\frac{1}{|z|^2}\E_x[|Y_t z |^2]& \le e^{-2 t^{\delta_1}}+\PE_x\left(\int_0^{t} \dob^{-r} (X_u) \text{d}u\le {\deux^{-1}} t^{\delta_1}\right)\nonumber\\
& \le e^{-2t^{\delta_1}}+\PE_x\left(\left(\int_0^{t} \dob^{-r} (X_u) \text{d}u\right)^{-2(1+\mte)} \ge ({\deux^{-1}} t^{\delta_1})^{-2(1+\mte)}\right)\nonumber\\
&\le e^{-2t^{\delta_1}}+ ({\deux^{-1}} t^{\delta_1})^{2(1+\mte)} \E_x\left[\left(\int_0^{t} \dob^{-r} (X_u) \text{d}u\right)^{-2(1+\mte)} \right].\nonumber
\end{align}
Since $x\mapsto x^{-2(1+\mte)}$ is convex on $(0,+\infty)$, it follows from the Jensen inequality that:
\begin{align} \label{eq:trozpaze}
\frac{1}{|z|^2}\E_x[|Y_t z |^2]& \le e^{-2t^{\delta_1}}+ ({\deux^{-1}} t^{\delta_1})^{2(1+\mte)} t^{-2(1+\mte)} \sup_{u\ge0} \E_x[ \dob^{2r(1+\mte)} (X_u)].
\end{align}
Setting $\delta_1=\mte/(2(1+\mte))$, using the Jensen inequality and {the elementary inequality $(a+b)^\frac{1}{2}\le a^\frac{1}{2}+b^{\frac{1}{2}}$ for non negative $a$ and $b$}, we obtain:
\begin{align*} 
\E_x[|Y_t z |]& \le |z| e^{- t^{\frac{\mte}{2(1+\mte)}}}+|z|  \deux^{{-1-\mte}} t^{-1-\frac{\mte}{2}}\sup_{u\ge0} \E_x[ \dob^{r(1+\mte)} (X_u)].
\end{align*}
%
By Proposition \ref{cor:momentsEUler}, Inequality \eqref{eq:universalboundmoment}, we get
\begin{align*} 
\E_x[\nsp{Y_t }]=\sup_{z\neq0} \frac{\E_x[|Y_t z |]}{|z|}& \lesssim_{uc} e^{- 2t^{\frac{\mte}{2(1+\mte)}}}+   \deux^{{-1-\mte}} (1+t)^{-1-\frac{\mte}{2}} \left(\dob^{r(1+\mte)}(x)+\Upsilon^{r(1+\mte)}\right).
\end{align*}
The above property has several consequences. First, it implies that $g$ is well-defined. Actually,
\begin{align}
&|P_t f(x)-\pi(f)|\le \int |P_t f(x)-P_t f(y)|\pi({d}y)\le \flip \int \sup_{u\in[x,y]}\E_u[\nsp{Y_t}]|y-x|\pi(dy)\nonumber\\
&\le C\left(e^{- 2t^{\frac{\mte}{2(1+\mte)}}}\int|y-x|\pi(dy)+ (1+t)^{-1-\frac{\mte}{2}}\int (\dob^{r(1+\mte)}(x)+\dob^{r(1+\mte)}(y))|y-x|\pi(dy)\right)\nonumber\\
&\le C_x\max(e^{- 2t^{\frac{\mte}{2(1+\mte)}}}, (1+t)^{-1-\frac{\mte}{2}}).\label{eq:expconvtang}
\end{align}
In the above inequalities, we used the convexity of $\dob$ and the fact that $\pi$ integrates functions with polynomial growth (simple consequence of Lemma \ref{lem:expbounds1cont}). Thus $t\mapsto |P_t f(x)-\pi(f)|$ is integrable on $[0,+\infty)$ and $g$ is thus well defined.\\

\noindent  {Then, the Lebesgue differentiability theorem applies. This implies that  $Dg(x)$ is well-defined on $\ER^d$ and that (by \eqref{eq:controleDflipp}
), 
\begin{equation}\label{controledgdex}
\nsp{Dg(x) } \le \deux^{-1-\mte}c_\mte \flip \left(\dob^{r(1+\mte)}(x)+\Upsilon^{r(1+\mte)}\right).
\end{equation}
where $c_\mte$ denotes a positive constant depending only $\mte$ (\textbf{$c_\mte$ may change from line to line but always depends only the parameter $\mte$}).
}

For the second inequality of this assertion, we again use Inequality \eqref{eq:universalboundmoment} of Proposition  \ref{cor:momentsEUler}. More precisely:
\begin{align}\nonumber
\E_{{{x_0}}} [\nsp{Dg(\bX_t)}^2]&\le c_\mte \deux^{-2-2\mte} \flip^2 \left(\E_{x_0}[\dob^{2r(1+\mte)}(\bX_t)]+\Upsilon^{2r(1+\mte)}\right)\\
&\le c_\mte \deux^{-2-2\mte}\flip^2 \left(\dob^{2r(1+\mte)}(x_0)+2\Upsilon^{2r(1+\mte)}\right). \label{eq:argeuler2}
\end{align}
The result follows.

\noindent \underline{Proof of $ii-a)$.}
{By the first order Taylor formula and \eqref{controledgdex},
\begin{align*}
 |g(x)-g(x^\star)|^2&\le \sup_{u\in[0,1]}\nsp{Dg(x^\star+u(x-x^\star)}^2|x-x^\star|^2\\
 &\le c_\mte \deux^{-2-2\mte} \flip^2 \left(\dob^{2r(1+\mte)}(x^\star+u(x-x^\star))+\Upsilon^{2r(1+\mte)}\right)|x-x^\star|^2.
 \end{align*}
Since $\dob$ is a convex function, $\dob(x^\star+u(x-x^\star))\le \dob(x)$. Thus, using \eqref{eq:asympdob}, we get
\begin{align*}
 \frac{1}{\flip^2} |g(x)-g(x^\star)|^2 &\le c_\mte \deux^{-3-2\mte} \left(\dob^{2r(1+\mte)}(x)+\Upsilon^{2r(1+\mte)}\right) \dob^{1+r}(x)\\
&\le c_\mte \deux^{-3-2\mte}  \left(\dob^{(1+3r)(1+\mte)}(x)+\Upsilon^{(1+3r)(1+\mte)}\right),
\end{align*}
where in the second line, we used the Young inequality with $p=(1+3r)/(2r)$ and $q=(1+3r)/(1+r)$.}

\noindent \underline{Proof of $ii-b)$.}
 We apply the result of $ii-a)$ and obtain that:
$${\E_{x_0}[| g(\bX_t)-g(x^\star)|^2]}{}\le c_\mte
\deux^{-3-2\mte} \flip^2\left(\E_{x_0}[{\dob}^{(1+3r)(1+\mte)}(\bX_t)]+ \Upsilon^{(1+3r)(1+\mte)}\right).
$$
By Proposition \ref{cor:momentsEUler} (Inequality \eqref{eq:universalboundmoment}), this yields
\begin{align*}
\E_{x_0}[| g(\bX_t)-g(x^\star)|^2]
\le c_\mte
\deux^{-3-2\mte} \flip^2 \left({\dob}^{(1+3r)(1+\mte)}(x_0)+2 \Upsilon^{(1+3r)(1+\mte)}\right).
 \end{align*}
The result follows under the assumption $W(x_0)\lesssim_{uc} \Upsilon$.  \\

\noindent Again, by $ii-a)$
$$ |g(x_0)-g(x^\star)|^2\le c_\mte \deux^{-3-2\mte} \flip^2\left({\dob}^{(1+3r)(1+\mte)}(x_0)+\Upsilon^{(1+3r)(1+\mte)}\right),$$
and the fact  that $W(x_0)\lesssim_{uc} \Upsilon$  implies that $ |g(x_0)-g(x^\star)|^2\le c_\mte \deux^{-3-2\mte} \flip^2\Upsilon^{(1+3r)(1+\mte)}.$
The result follows.\\



\noindent \underline{Proof of $iii)$.}
First, since $Df(X_t^y) Y_t^y-Df(X_t^x) Y_t^x= Df(X_t^y) (Y_t^y-Y_t^x)+(Df(X_t^y)-Df(X_t^x))Y_t^x$,
\begin{equation}\label{eq:raslebol24}
 \nsp{Dg(y)-Dg(x)}\le \int_0^{+\infty} \flip \E[\nsp{Y_t^y-Y_t^x}]+ \dflip \E[|X_t^y-X_t^x| \nsp{Y_t^x} ] dt.
 \end{equation}
First, since  $\sup_{u\in\ER^d,t\ge0} \nsp{Y_t^u}\le 1$, $|X_t^y-X_t^x|\le |y-x|$ and with the same argument as in \eqref{controledgdex},
\begin{equation}\label{eq:raslebol25}
\int_0^{+\infty}\E[|X_t^y-X_t^x| \nsp{Y_t^x} ]\le |y-x| \deux^{-1-\mte}c_\mte \flip \left(\dob^{r(1+\mte)}(x)+\Upsilon^{r(1+\mte)}\right).
\end{equation}
Second, set $\Upsilon_{x,y}(s)=\min\left(\underline{\lambda}_{\hesdob}(X_u^x),\underline{\lambda}_{\hesdob}(X_u^y)\right)$.
By Lemma \ref{lem:pathconttang},
\begin{align*}
 \E[\nsp{Y_t^y-Y_t^x}^2]  &\le \tilde{L}^2 {|y-x|^2} \E\left[e^{-\int_0^t  \Upsilon_{x,y}(s) ds} \int_0^t \underline{\lambda}^{-1}_{\hesdob}(X_u^y) ds\right].
 \end{align*}
 By $\Hcu$, Cauchy-Schwarz and Jensen inequalities, this yields:
 \begin{align}
  \E[\nsp{Y_t^y-Y_t^x}^2]  &\le  \tilde{L}^2  {|y-x|^2}\E\left[e^{-2\int_0^t  \Upsilon_{x,y}(s) ds}\right]^{\frac{1}{2}} \E\left[\left(\int_0^t  \frac{\dob^r(X_s^y)}{\deux} ds \right)^2\right]^{\frac 12}\nonumber\\
  &\le   \frac{\tilde{L}^2}{\deux} {|y-x|^2} \E\left[e^{-2\int_0^t  \Upsilon_{x,y}(s) ds}\right]^{\frac{1}{2}} t\sup_{t\ge0} \E[\dob^{2r}(X_t^y)]^{\frac{1}{2}}\nonumber\\
&\lesssim_{uc}  \frac{\tilde{L}^2}{\deux} {|y-x|^2}  \E\left[e^{-2\int_0^t  \Upsilon_{x,y}(s) ds}\right]^{\frac{1}{2}} t \left(\dob^r(y)+\Upsilon^r\right),\label{raslebol23}
\end{align}
 where in the last line, we used Proposition \ref{cor:momentsEUler} (Inequality \eqref{eq:universalboundmoment}). 
For the first right-hand term, we use a similar strategy as for 
\eqref{eq:trozpaze11} and get for every $t\ge1$, for any positive $\delta_1$ and $\delta_2$,
$$\E\left[e^{-2\int_0^t \Upsilon_{x,y}(s) {\rm{d}}u}\right]
\le e^{-2 t^{\delta_1}}+ t^{\delta_2(\delta_1-1)} \sup_{t\ge0} \E[  \Upsilon_{x,y}(t)^{-\delta_2}].$$
By $\Hcu$, it follows that
$$\E\left[e^{-2\int_0^t \Upsilon_{x,y}(s) {\rm{d}}u}\right]
\le e^{-2 t^{\delta_1}}+ \deux^{-\delta_2}t^{\delta_2(\delta_1-1)} \sup_{t\ge0} \E[  \dob^{r\delta_2}(X_t^x)+\dob^{r\delta_2}(X_t^y)].$$
By \eqref{raslebol23}, we deduce that
\begin{align*}
  \frac{\E[\nsp{Y_t^y-Y_t^x}^2]}{|y-x|^2} &\lesssim_{uc}  \frac{\tilde{L}^2}{\deux}\left(\dob^r(y)+\Upsilon^r\right)\left[ te^{-2 t^{\delta_1}}+
   \deux^{-\delta_2}t^{\delta_2(\delta_1-1)+1} \sup_{t\ge0} \E[  \dob^{r\delta_2}(X_t^x)+\dob^{r\delta_2}(X_t^y)]\right].
\end{align*}
For a positive $\mte$, we fix $\delta_2=3(1+\mte)$ and $\delta_1=\mte/(3(1+\mte))$ so that $\delta_2(\delta_1-1)+1=-2(1+\mte)$.
By  Proposition \ref{cor:momentsEUler} (Inequality \eqref{eq:universalboundmoment}), we deduce that:
\begin{align*}
  \frac{\E[\nsp{Y_t^y-Y_t^x}^2]}{|y-x|^2} \lesssim_{uc}  & \frac{\tilde{L}^2}{\deux}\left(\dob^r(y)+\Upsilon^r\right) \times\\
& \left[ te^{-2 t^{\frac{\mte}{3(1+\mte)}}}+
  \deux^{-3(1+\mte)}t^{-2(1+\mte)}\left(\dob^{3r(1+\mte)}(x)+\dob^{3r(1+\mte)}(y)+\Upsilon^{3r(1+\mte)}\right)\right].
\end{align*}
By \eqref{eq:raslebol24}, Jensen Inequality and the elementary inequality $(a+b)^p\le c_p(a^p+b^p)$, we finally obtain
\begin{align*}
 \frac{\int_0^{+\infty}\E[\nsp{Y_t^y-Y_t^x}]dt}{|y-x|}\le&  \,c_\mte \frac{\tilde{L}}{\sqrt{\deux}}
 \left(\dob^\frac{r}{2}(y)+\Upsilon^\frac{r}{2}\right)\times\\
& \left[1+\deux^{-\frac{3}{2}(1+\mte)}\left(\dob^{\frac{3r}{2}(1+\mte)}(x)+\dob^{\frac{3r}{2}(1+\mte)}(y)+\Upsilon^{\frac{3r}{2}(1+\mte)}\right)\right]\\
\le  & \,c_\mte \deux^{-2(1+\mte)}\tilde{L}\left(\dob^{2r(1+\mte)}(x)+\dob^{2r(1+\mte)}(y)+\Upsilon^{2r(1+\mte)}\right),
\end{align*}
where in the last line, we used Young inequality with $p=4/3$ and $q=4$. Thus, by \eqref{eq:raslebol24}, \eqref{eq:raslebol25} and the previous equation, we deduce the announced result.

\noindent \underline{Proof of $iii)-b)$.}  Set $\tX_{{t}}=\bX_{\un{t}}-(t-\un{t}) \nabdob(\bX_{\un{t}})$. We have

\begin{align*}
|(Dg(\bX_{{t}})-&Dg(\tX_{{t}}))(\nabla W(\bX_{\un{t}}+\Delta_{\un{t}t})-\nabla W(\bX_{\un{t}}))|\le
\nsp{Dg(\bX_{{t}})-Dg(\tX_{{t}})}L |\Delta_{\un{t}t}| \\
& \lesssim_{uc} c_\mte  \deux^{-2(1+\mte)}\tilde{L}\left(\dob^{2r(1+\mte)}(\bX_{{t}})+\dob^{2r(1+\mte)}(\tX_{{t}})+\Upsilon^{2r(1+\mte)}\right)L|\Delta_{\un{t}t}|^2(t-\un{t})
\end{align*}
where in the second line, we used $(iii)-a)$. By  Cauchy-Schwarz inequality and the elementary inequality $|a+b+c|^2\le 3(|a|^2+|b|^2+|c|^2$, we deduce that
$$\E_{x_0}\left[|(Dg(\bX_{{t}})-Dg(\tX_{{t}}))(\nabla W(\bX_{\un{t}}+\Delta_{\un{t}t})-\nabla W(\bX_{\un{t}}))|^2\right]\lesssim_{uc} \mathfrak{c}_\mte \frac{(L\tilde{L})^2}{\deux^{2(1+\mte)}} \E[|\Delta_{\un{t}t}|^8]^{\frac{1}{2}} \E_{x_0}[\Xi_t^2]^{\frac{1}{2}},
$$
with
$$\Xi_t=\dob^{{4r}(1+\mte)}(\bX_t)+\dob^{{4r}(1+\mte)}(\tX_{{t}})+\Upsilon^{4r(1+\mte)}.
$$ 
 By Cauchy-Schwarz inequality and Proposition \ref{cor:momentsEUler} (Inequality \eqref{eq:universalboundmoment}) and a slight adaptation for 
$\dob^{p}(\tX_{{t}})$),
\begin{align*}
\E_{x_0}[\Xi_t^2]^{\frac{1}{2}}&\lesssim_{uc} \dob^{4r(1+e)}(x_0)+\Upsilon^{4r(1+\mte)}.
\end{align*}
Note that we again used elementary inequalities (including Young inequality) which involved constants which can be bounded by universal constants (and are thus ``hidden'' in the notation ``$\lesssim_{uc}$'').
The result follows by using that 
$$\E[|\Delta_{\un{t}t}|^8]^{\frac{1}{2}}=4(t-\un{t})^2\E[(\chi_2(d))^4]^{\frac{1}{2}}=4(t-\un{t})^2(d(d+2)(d+4)(d+6))^{\frac{1}{2}}\lesssim_{uc} (t-\un{t})^2d^2.$$
%
\end{proof}

\section{Discretization tools - Strongly convex case}

\begin{proof}[Proof of Proposition \ref{prop:bisbis} {of the main document}]

\noindent $(i)$ By Lemma \ref{lem:pathconttang}, the strong convexity yields
$$ \forall x\in\ER^d
\qquad  \nsp{Y^{x}_t}\le e^{-\rho t}.$$
Thus, by Equation \eqref{eq:controleDflipp},
$$ \nsp{Dg(x)}\le \flip \int_0^{+\infty} e^{-\rho t} dt\le \frac{\flip}{\rho}$$.
The first assertion follows since $\flip\le 1$.

$(ii)$ By $(i)$,
$$\E_{x_0}[|g(\bar{X}_t)-g(x_0)|^2]\le \frac{1}{\rho^2}\E_{x_0}[|\bar{X}_t-x_0|^2].$$
By Lemma 5.1(i) of \cite{egea-panloup}, if $\nabdob(x^\star)=0$,
$$\E[|\bar{X}_t^{x_0}-x^\star|^2]\le |x_0-x^\star|^2 e^{-{\frac{\rho}{2}} t} + {\frac{2d}{\rho}}.$$

\noindent
The result follows by using the elementary inequality $|\bar{X}_t^{x_0}-x_0|^2\le 2 (|\bar{X}_t^{x_0}-x^\star|^2+|{x_0}-x^\star|^2).$

\noindent $(iii)$ By Equation \eqref{eq:raslebol24} and the Cauchy-Schwarz inequality,
\begin{equation}
 \nsp{Dg(y)-Dg(x)}\le \int_0^{+\infty} \flip \E[\nsp{Y_t^y-Y_t^x}]+ \dflip \E[|X_t^y-X_t^x|^2]^\frac{1}{2} \E[\nsp{Y_t^x}^2 ]^\frac{1}{2} dt.
 \end{equation}
On the one hand, by Lemma \ref{lem:pathconttang}$(ii)$ and the fact $\hesdob$ is $\tilde{L}$-Lipschitz for the norm $\nsp{.}$,
  
$$\nsp{Y_t^y-Y_t^x}^2\le\tilde{L}^2 |x-y|^2
e^{-\rho t} \int_0^t \frac{1}{\rho}e^{-\rho s} ds\le \frac{e^{-\rho t}}{\rho^2}{ \tilde{L}^2  |x-y|^2} \quad\forall t\ge0.
$$
On the other hand, by Lemma \ref{lem:pathconttang}$(i)$ and the Jensen inequality,
 $$\E[|X_t^y-X_t^x|^2]=\E\left[\left|\int_0^1 Y_t^{x+\theta(y-x)} (y-x)d\theta \right|^2\right]\le \sup_{\theta\in[0,1]} \E \nsp{Y_t^{x+\theta(y-x)}}^2 |y-x|^2 \le |x-y|^2 e^{-2\rho t}.$$
The Cauchy-Schwarz inequality, $\flip \leq 1$ and $\dflip \leq 1$ yield:
$$ \nsp{Dg(y)-Dg(x)}\le  |y-x| \int_0^{+\infty} \left({\frac{\flip\tilde{L}}{\rho}} e^{-{\frac{\rho}{2}} t} + \dflip e^{-{2 \rho} t} \right)  dt\le |x-y| \left( {\frac{2 \tilde{L} }{\rho^2}} + {\frac{1}{2 \rho}}\right).$$



\bibliographystyle{alpha}
\bibliography{ref_bayesian_langevin}

\end{proof}
\end{document}